\documentclass{lms}
\usepackage{amsmath,amssymb,amscd,mathtools}
\usepackage{tikz}

\usepackage{url}

\DeclareUrlCommand\ULurl{%
  \renewcommand\UrlLeft{\uline\bgroup}%
  \renewcommand\UrlRight{\egroup}}

\newtheorem{theorem}{Theorem}[section] 
\newtheorem{lemma}[theorem]{Lemma}     
\newtheorem{corollary}[theorem]{Corollary}
\newtheorem{proposition}[theorem]{Proposition}

\newnumbered{assertion}{Assertion}    
\newnumbered{conjecture}{Conjecture}  
\newnumbered{definition}{Definition}
\newnumbered{hypothesis}{Hypothesis}
\newnumbered{remark}{Remark}
\newnumbered{note}{Note}
\newnumbered{observation}{Observation}
\newnumbered{problem}{Problem}
\newnumbered{question}{Question}
\newnumbered{algorithm}{Algorithm}
\newnumbered{example}{Example}
\newunnumbered{notation}{Notation} 
\newunnumbered{assumptions}{Assumptions}
\newunnumbered{notationconventions}{Notation and conventions}
\newunnumbered{conventions}{Conventions}

\title[KO-Theory of Complex Flag Varieties of Ordinary Type]
 {KO-Theory of Complex Flag Varieties of Ordinary Type} 

\author{Tobias Hemmert}

\classno{19L64 (primary), 55N15, 14F43 (secondary)}

\newcommand\blfootnote[1]{%
  \begingroup
  \renewcommand\thefootnote{}\footnote{#1}%
  \addtocounter{footnote}{-1}%
  \endgroup
}

\begin{document}
\maketitle
\tableofcontents
\begin{abstract}
\blfootnote{This research was partially conducted in association with \emph{GRK 2240: Algebro-Geometric Methods in Algebra, Arithmetic and Topology}, which is funded by the DFG.}We compute the topological Witt groups of every complex flag manifold of
ordinary type, and thus the interesting (i.e. torsion) part of the
KO-groups of these manifolds.  Equivalently, we compute Balmer's Witt
groups of each flag variety of ordinary type over an algebraically
closed field of characteristic not two.

Our computation is based on an approach developed by Zibrowius.  For
types A, B and C, we obtain a full description not only of the additive
but also of the multiplicative structure of the graded Witt rings. 
\end{abstract}

\section*{Introduction}
As one of the first extraordinary cohomology theories to be discovered, the computation of real topological K-theory $KO^\ast$ is a classical problem. It is therefore surprising that our knowledge of the KO-rings of very common classes of spaces such as homogenous spaces is rather patchy whereas their ordinary cohomology rings are often well-understood. We want to focus here on \emph{complex flag varieties}, that is homogeneous spaces $G/H$ where $G$ is a compact semisimple Lie group and $H=C_G(S)$ is the centraliser of a torus $S\subset G$. One can show (using \cite[Lemma 13.6]{borelhirz1}) that in the compact simple Lie groups of ordinary type, up to conjugation the centralisers of a torus are given as follows:
\vspace{3mm}
\begin{center}
    \begin{tabular}{ccl}
         $G$ & \begin{tabular}{@{}c@{}}Centralisers $H$ of a torus \\ (up to conjugation)\end{tabular} \\
         \hline
         $SU(n)$ &  $S(U(n_1)\times\ldots\times U(n_l))$ & where $n_1+\ldots+n_l=n$   \\
         $SO(2n+1)$ &  $SO(2m+1)\times U(n_1)\times\ldots\times U(n_l)$ & where $m+n_1+\ldots+n_l=n$  \\
         $Sp(n)$ &   $Sp(m)\times U(n_1)\times\ldots \times U(n_l)$ & where $m+n_1+\ldots+n_l=n$ \\
         $SO(2n)$ & $SO(2m)\times U(n_1)\times \ldots\times U(n_l)$ & where $m+n_1+\ldots+n_l=n$\\
    \end{tabular}
\end{center}
\vspace{3mm}
We call the resulting homogeneous spaces $G/H$ \emph{complex flag varieties of ordinary type}. In this paper, we compute the torsion part of their KO-groups. The free part is easily determined from the rational cohomology \cite[Lemma 1.2]{Zibrowius1}, which is well-understood.

Before we go into details, let us put our results into perspective by giving a brief overview of what is already known about the KO-groups of complex flag varieties. In \cite{fujii1}, all KO-groups of complex projective space $\mathbb{CP}^n$ were computed. This was generalised to the KO-groups of complex Grassmannians $SU(m+n)/S(U(m)\times U(n))$ in \cite{konohara1}. Much more recently, the KO-groups of all complex flag varieties of the form $SU(n_1+\ldots+n_l)/S(U(n_1)\times\ldots\times U(n_l))$, i.e. with $G$ of type $A_n$, were computed in \cite{kishimotokonoohsita1}\footnote{Unfortunately, this paper contains some arithmetic mistakes so that the final result is flawed. This can already be seen from the fact that their expressions for the ranks that are given in Table 1 of \cite{kishimotokonoohsita1} do not always yield integers but fractions.}. The KO-theory of full flag varieties $G/T$, where $T$ is a maximal torus in $G$, was obtained for the simple groups $G$ of ordinary type in \cite{kishimotokonoohsita2}. This was extended to include $G=G_2,F_4,E_6$ in \cite{kishimotoohsita1}.

All the results about the KO-theory of complex flag varieties mentioned so far were obtained in essentially the same way: The authors considered the Atiyah-Hirzebruch spectral sequence for $KO^\ast$, computed the $E_2$-page including the differentials, obtained the $E_3$-page from this and then showed in each individual case that the spectral sequence collapses on the third page. The arguments required to make this work are intricate, which is why only partial results are known.

In \cite{Zibrowius1}, Zibrowius develops an alternative approach and computes the KO-groups of all full flag varieties $G/T$ in an essentially type-independent way. The basic strategy is to consider the so-called \emph{Witt groups} $W^i$. For a topological space $X$, they are defined as the cokernels of the realification maps $r_i\colon K^{2i}(X)\to KO^{2i}(X)$:
\[
W^i(X):=KO^{2i}(X)/r_i
\]
Since the KO-groups are 8-periodic, the Witt groups are 4-periodic. The important observations in \cite{Zibrowius1} are:
\begin{itemize}
\item These Witt groups are computable via a result of Bousfield (cf. Lemma \ref{lemmabousfield}).
\item  The Witt groups determine the torsion in the KO-groups of complex flag varieties \cite[Lemma 1.2]{Zibrowius1}: For a complex flag variety $X=G/H$, we have isomorphisms
\begin{align*}
KO^{2i}(X)&\cong W^{i+1}(X)\oplus \text{free part}\\
KO^{2i+1}(X)&\cong W^{i+1}(X)
\end{align*}
\item The \emph{total Witt group}
\[
W^\ast(X):=\bigoplus_{i\in\mathbb{Z}_4} W^i(X)
\]
is a graded ring, and the ring structure of $W^\ast$ determines part of the ring structure of $KO^\ast$ \cite[Remark 1.3]{Zibrowius1}.
\end{itemize}
The whole computation of the Witt ring in \cite{Zibrowius1} is essentially representation-theoretic. The result for full flag varieties can be concisely stated as follows:
\begin{theorem*}[{\cite[Theorem 3.3]{Zibrowius1}}]
Let $G$ be a simply connected compact Lie group and let $T\subset G$ be a maximal torus. The Witt ring of $G/T$ is an exterior algebra on $b_\mathbb{H}$ generators of degree 1 and $\frac{b_\mathbb{C}}{2}+b_\mathbb{R}$ generators of degree 3, where $b_\mathbb{C}$, $b_\mathbb{R}$ and $b_\mathbb{H}$ denote the number of fundamental representations of $G$ of complex, real and quaternionic type, respectively.
\end{theorem*}
In this paper we use the approach developed by Zibrowius to compute the Witt rings of all complex flag varieties of types $A_n$, $B_n$ and $C_n$ (cf. Theorems \ref{thmwittringtypea}, \ref{thmwittringtypeb} and \ref{theoremwittringtypecn}). The method also works for type $D_n$ and the arguments are very similar to the other types, but the computation is more technical because one has to distinguish more cases. In some cases of type $D_n$, we have not been able to determine the ring structure, only the Witt \emph{groups}. For type $D_n$, we therefore only briefly state our results in this paper (cf. Theorem \ref{thmwittringtypedn}), and refer to \cite{Hemmert} for details of the computation.

As an example, we obtain the following result for type $C_n$:
\begin{theorem}\label{theoremintrotypec}
Let
$X:=Sp(n)/Sp(m)\times U(n_1)\times \ldots\times U(n_l)$
where $m+n_1+\ldots+n_l=n$ and set
\[
h:=\Bigl\lfloor\frac{m}{2}\Bigr\rfloor +\Bigl\lfloor\frac{n_1}{2}\Bigr\rfloor+\ldots+\Bigl\lfloor\frac{n_l}{2}\Bigr\rfloor \text{ and }f:=\Bigl\lfloor\frac{n}{2}\Bigr\rfloor-h \text{ and }g:=\Bigl\lceil \frac{n-m}{2}\Bigr\rceil.
\]
Then there is an isomorphism of $\mathbb{Z}_4$-graded rings
\[
W^\ast(X)\cong\frac{\mathbb{Z}_2\left[a_1,\ldots,a_{\lfloor m/2\rfloor}\right]\otimes \bigotimes_{p=1}^l\mathbb{Z}_2\left[b_1^{(p)},\ldots,b_{\lfloor n_p/2\rfloor}^{(p)}\right]}{(\mu_1,\ldots,\mu_h)}\otimes_{\mathbb{Z}_2} \bigwedge(u_1,\ldots,u_f,v_1,\ldots,v_g),
\]
where the generators on the right are of degrees $|a_i|=|b_j|=0$, $|u_i|=3$ and $|v_j|=1$. The relations $\mu_j$ are given by
\[
\mu_j=\sum_{c+c_1+\ldots+c_l=j}a_c\cdot b_{c_1}^{(1)}\ldots b_{c_l}^{(l)}+\binom{2n}{2j},
\]
where it is understood that
\begin{align*}
a_0:=1 \text{ and } a_i:=a_{m-i}\text{ for all }\lfloor m/2\rfloor <i\leq m,\\
b_0^{(p)}:=1\text{ for all }1\leq p\leq l\text{ and }b_i^{(p)}:=b_{n_p-i}^{(p)}\text{ for all }1\leq p\leq l,~\lfloor n_p/2\rfloor <i\leq n_p.
\end{align*}
\end{theorem}

The homogeneous spaces $G/H$ we study in this paper have the structure of
complex projective varieties.  They can be written as $G_\mathbb{C}/P$, where $G_\mathbb{C}$
is a simple complex algebraic group containing $G$ as a maximal compact
subgroup, and $P$ is a parabolic subgroup of $G_\mathbb{C}$ \cite[Proposition 4.1]{Zibrowius2}. It is an observation of Zibrowius in \cite{Zibrowius1} that the computation of our topological Witt ring of $G/H$ is completely parallel to a computation of the algebraic Witt ring (in the sense of Balmer \cite{Balmer1}) of $G_\mathbb{C}/P$ and that they are thus isomorphic. More generally, this algebraic computation applies to all algebraically closed fields of characteristic not two (see also Remark \ref{remarkarbitraryfield}).

Let us state this more precisely. Recall that by classical results of Chevalley, over any algebraically closed field, the simple simply connected algebraic groups are in bijective correspondence with the connected Dynkin diagrams. The conjugacy classes of parabolic subgroups of such a simple simply connected algebraic group $G$ are in bijective correspondence with subsets of nodes of the Dynkin diagram of $G$. The following holds:

\begin{theorem}
Fix a connected Dynkin diagram and some subset of its nodes. Now let $k$ be an algebraically closed field of characteristic not two, $G_k$ be the simple simply connected algebraic group over $k$ corresponding to our Dynkin diagram and $P_k$ be a parabolic subgroup of $G_k$ corresponding to our subset of nodes of the Dynkin diagram. Then the (algebraic) Witt ring $W^\ast(G_k/P_k)=\bigoplus_{i\in \mathbb{Z}_4} W^i(G_k/P_k)$ is independent of $k$ and is isomorphic to the topological Witt ring of $G/H$ computed in this paper, where $G\subset G_\mathbb{C}$ is a maximal compact subgroup (thus a compact simple Lie group) and $H=P_\mathbb{C}\cap G$.
\end{theorem}

We now give a brief outline of this paper. In section \ref{sectionwittringstatecoh}, we explain Bousfield's lemma, which states that Witt rings can be computed as Tate cohomology rings. This is the result that makes the Witt ring computable for us. We then collect some technical lemmas on Tate cohomology that we need later for our concrete computations. Section \ref{tatecohsubsect3} comprises the computation of the representation rings of some compact Lie groups and their Tate cohomology, which we will also need for our computations. Finally, in section \ref{chapteroutlinecomputation} we are ready to give an outline of the approach of Zibrowius to the computation of the Witt rings of complex flag varieties. We extend some of his results so that they apply more generally to arbitrary complex flag varieties and not only to full flag varieties. The actual computations of the Witt rings of ordinary complex flag varieties are then performed in section \ref{sectioncomputation}. The proofs of some crucial lemmas at the heart of our computations are postponed to section \ref{sectionpolynomialeqns}.

\begin{conventions}
We write $\mathbb{Z}_n$ for the integers modulo $n$.
\end{conventions}

\begin{acknowledgements}\label{ackref}
The work in this paper is part of my PhD thesis. Many thanks are due to my advisor Marcus Zibrowius for his unwavering patience, his constant support and innumerable discussions. I also wish to thank Stefan Schr\"{o}er for answering my commutative algebra questions.
\end{acknowledgements}

\section{Witt rings and Tate cohomology}\label{sectionwittringstatecoh}
We saw in the introduction that the torsion in the KO-groups of complex flag varieties is determined by their Witt groups. The decisive result in this section is Bousfield's lemma, which says that the Witt ring of a complex flag variety is isomorphic to the Tate cohomology ring of its complex K-theory. We will see subsequently that this makes the Witt ring computable.

We give an outline of this section. In \ref{tatecohsubsect1}, we give a definition of Tate cohomology and explain the aforementioned Bousfield's lemma. In \ref{tatecohsubsect2}, we collect a few technical lemmas about the Tate cohomology of quotient rings that will be essential for our computation of the Witt ring of complex flag varieties.
\subsection{Witt rings are Tate cohomology rings}\label{tatecohsubsect1}
A \emph{$\ast$-ring} $A$ is a commutative unital ring with an involution $\ast\colon A\to A$ which is a ring isomorphism. A \emph{$\ast$-ideal} of a $\ast$-ring is an ideal which is closed under the involution. A \emph{$\ast$-module} is an abelian group with an additive involution. A \emph{$\ast$-module $M$ over a $\ast$-ring A} is an $A$-module $M$ which is also a $\ast$-module, such that the involutions on $A$ and $M$ are compatible with the module structure, i.e. $(a\cdot m)^\ast=a^\ast\cdot m^\ast$ for all $a\in A$ and $m\in M$. For example, any $\ast$-ideal of $A$ is a $\ast$-module over $A$. Every $\ast$-module is a $\ast$-module over $\mathbb{Z}$ with the trivial involution on $\mathbb{Z}$. A \emph{morphism of $\ast$-modules} is a homomorphism of modules that preserves the involution and a \emph{morphism of $\ast$-rings} is a ring homomorphism that preserves the involution. For a $\ast$-module $M$, an element $m\in M$ is \emph{self-dual} if $m^\ast=m$ and \emph{anti-self-dual} if $m^\ast=-m$.
\begin{example}
For any compact Hausdorff space $X$, its complex topological K-theory $K^0(X)$ is a $\ast$-ring with the duality induced by assigning to each complex vector bundle its dual bundle.

For any compact connected Lie group $G$, its complex representation ring $R(G)$ is a $\ast$-ring with the duality induced by assigning to each complex representation its dual representation.
\end{example}
For a $\ast$-module $M$, we define its \emph{Tate cohomology} to be $h^\ast(M)=h^+(M)\oplus h^-(M)$ where
\begin{align*}
    h^+(M):=\frac{\text{ker}(\text{id}-\ast)}{\text{im}(\text{id}+\ast)}~~\text{and}~~
    h^-(M):=\frac{\text{ker}(\text{id}+\ast)}{\text{im}(\text{id}-\ast)}
\end{align*}
In other words, $h^+(M)$ and $h^-(M)$ consist of the self-dual and anti-self-dual elements of $M$ respectively modulo those elements that are self-dual and anti-self-dual respectively for trivial reasons.

If $A$ is a $\ast$-ring, it is easily checked that the multiplication on $A$ induces a ring structure on $h^\ast(A)$ with
\begin{align*}
    h^+(A)\cdot h^+(A)\subset h^+(A)\\
    h^+(A)\cdot h^-(A)\subset h^-(A)\\
    h^-(A)\cdot h^-(A)\subset h^+(A)
\end{align*}
We now prepare to formulate Bousfield's lemma. Let $X$ be a compact Hausdorff space. We have complexification and realification maps $c_i\colon KO^{2i}(X)\to K^0(X)$ and $r_i\colon K^0(X)\to KO^{2i}(X)$, which are the ordinary complexification and realification maps $KO^{2i}(X)\to K^i(X)$ and $K^{2i}(X)\to KO^{2i}(X)$ composed with appropriate powers of the Bott isomorphism for complex K-theory. Since $\ast c_i=(-1)^i c_i$ and $c_ir_i=\text{id}+(-1)^i \ast$, the map $c_i$ descends to a map
\[
\overline{c}_i\colon KO^{2i}(X)/r_i\to h^i(K^0(X))
\]
where $KO^{2i}(X)/r_i$ denotes the cokernel of $r_i$. Since $c_ir_i=\text{id}+(-1)^i \ast$ and $r_i \ast=(-1)^i r_i$, the map $r_i$ descends to a map
\[
\overline{r}_i\colon h^i(K^0(X))\to c_i\backslash KO^{2i}(X)
\]
where $c_i\backslash KO^{2i}(X)$ denotes the kernel of $c_i$.
\begin{lemma}[(Bousfield's lemma)]\label{lemmabousfield}
Let $X$ be a compact Hausdorff space with $K^1(X)=0$. The complexification and realification maps induce isomorphisms
\begin{align*}
    W^0(X)\oplus W^2(X)\xrightarrow[(\overline{c}_0~\overline{c}_2)]{\cong} h^+(K^0(X))\xrightarrow[\left(\begin{smallmatrix}\overline{r}_{-1}\\\overline{r}_1 \end{smallmatrix}\right)]{\cong} c\backslash KO^{-2}(X)\oplus c\backslash KO^2(X)\\
    W^1(X)\oplus W^3(X)\xrightarrow[(\overline{c}_1~\overline{c}_3)]{\cong} h^-(K^0(X))\xrightarrow[\left(\begin{smallmatrix}\overline{r}_{0}\\\overline{r}_2 \end{smallmatrix}\right)]{\cong} c\backslash KO^{0}(X)\oplus c\backslash KO^4(X)
\end{align*}
The composition along each row is given by $\left(\begin{smallmatrix} \eta^2 & 0\\0 &\eta^2\end{smallmatrix}\right)$ where $\eta$ denotes the generator of $KO^{-1}(\text{pt})$. Since complexification is a ring homomorphism, we get an isomorphism of rings $\overline{c}\colon W^\ast(X)\to h^\ast(K^0(X))$.
\end{lemma}
\begin{proof}
A simple proof can be found in \cite[section 1.2]{Zibrowius1}.
\end{proof}
\begin{remark}
For a complex flag variety $X$, we have $K^1(X)=0$ by \cite{Hodgkin1}, Lemma 9.2 and (9.1). Thus Bousfield's lemma applies to complex flag varieties.
\end{remark}
\subsection{Tate cohomology of quotient rings}\label{tatecohsubsect2}

We collect some technical lemmas required to compute the Tate cohomology of quotients of $\ast$-rings by $\ast$-ideals. This will be essential for our computation of Witt rings of complex flag varieties. We start by considering the Tate cohomology of $\ast$-ideals and then move on to quotient rings.

We are first interested in relating the Tate cohomology of $\ast$-ideals to the Tate cohomology of the ambient ring.
\begin{lemma}[{\cite[Proposition 4.1]{Zibrowius1}}]\label{tatecohlemma1}
Let $A$ be a $\ast$-ring.
\begin{enumerate}[(i)]
\item If $\mu\in A$ is self-dual and not a zero divisor, then
$h^{\ast}(A)\xrightarrow{\cdot[\mu]} h^{\ast}(\mu A)$
is a graded isomorphism.
\item If $\lambda,\lambda^\ast$ is a regular sequence in $A$, then
$
h^\ast(A)\xrightarrow{\cdot[\lambda\lambda^\ast]} h^\ast(\lambda,\lambda^\ast)
$
is a graded isomorphism.
\end{enumerate}
\end{lemma}
We will also need the following lemma about the Tate cohomology of an ideal generated by two independent self-dual elements:
\begin{lemma}\label{tatecohlemma2}
Let $A$ be a $\ast$-ring with $h^-(A)=0$ and let $\mu_1,\mu_2\in A$ be self-dual elements that form an $A$-regular sequence. Suppose that $\text{Ann}_{h^+(A)}\left([\mu_1],[\mu_2]\right)=0$. Then
\[
h^-(A\mu_1+A\mu_2)=0
\]
and we have an isomorphism of $h^+(A)$-modules given by
\begin{align*}
    \frac{h^+(A)\oplus h^+(A)}{h^+(A)\cdot ([\mu_2],[\mu_1])}\xrightarrow{\cong} h^+(A\mu_1+A\mu_2),~~~\overline{(x,y)}\mapsto x[\mu_1]+y[\mu_2]
\end{align*}
\end{lemma}
\begin{proof}
We define an $A$-module homomorphism by
\[
\varphi\colon A\oplus A\to A\mu_1+A\mu_2,~~~ (a,b)\mapsto a\mu_1+b\mu_2.
\]
$\varphi$ is clearly onto. Let $(a,b)\in\text{ker}(\varphi)$, i.e. $a\mu_1+b\mu_2=0$. By regularity of the sequence $\mu_1,\mu_2$, we deduce that $b=b'\mu_1$ for some $b'\in A$, so $(a+b'\mu_2)\mu_1=0$. Again by regularity, we have $a+b'\mu_2=0$ and so
\[
\text{ker}(\varphi)=\left\{(x\mu_2,-x\mu_1)\mid x\in A\right\},
\]
which is isomorphic to $A$ as a $\ast$-module via the isomorphism $\kappa\colon A\to \text{ker}(\varphi),~x\mapsto (x\mu_2,-x\mu_1)$. Now the short exact sequence of $\ast$-modules over $A$
\[
0\xrightarrow{} A\xrightarrow{\kappa} A\oplus A\xrightarrow{\varphi} A\mu_1+A\mu_2\xrightarrow{} 0
\]
induces the following exact sequence on Tate cohomology, using that $h^-(A)=0$:
\[
0\to h^-(A\mu_1+A\mu_2)\to h^+(A)\xrightarrow{\kappa_\ast} h^+(A)\oplus h^+(A)\xrightarrow{\varphi_\ast} h^+(A\mu_1+A\mu_2)\to 0
\]
The assertions follow from this exact sequence since
\[
\kappa_\ast\colon h^+(A)\to h^+(A)\oplus h^+(A),~~~x\mapsto (x[\mu_2],x[\mu_1])
\]
is injective as $\text{Ann}_{h^+(A)}([\mu_1],[\mu_2])=0$ by assumption.
\end{proof}
We next consider the Tate cohomology of some quotient rings.
\begin{lemma}\label{tatecohlemma4}
Let $A$ be a $\ast$-ring and $\mu,\lambda\in A$ such that $\mu$ is self-dual and not a zero divisor and $\lambda,\lambda^\ast$ is a regular sequence.
\begin{enumerate}[(i)]
\item If $[\mu]=0$ in $h^+(A)$, then
\begin{align*}
h^\ast(A/(\mu))\cong h^\ast(A)\oplus [\overline{u}]\cdot h^\ast(A)
\end{align*}
where $u\in A$ with $u+u^\ast=\mu$.
\item If $[\lambda\lambda^\ast]=0$ in $h^+(A)$, then
\begin{align*}
h^\ast(A/(\lambda,\lambda^\ast))\cong h^\ast(A)\oplus [\overline{u}]\cdot h^\ast(A)
\end{align*}
where $u\in A$ with $u+u^\ast=\lambda\lambda^\ast$.
\item Suppose further that $h^-(A)=0$. Then $h^-(A/\mu A)=0$ if and only if $[\mu]\in h^+(A)$ is not a zero divisor in $h^+(A)$. Furthermore, we have $h^\ast(A/(\mu))\cong h^+(A)/([\mu])$.
\item Suppose further that $h^-(A)=0$. Then $h^-(A/(\lambda,\lambda^\ast))=0$ if and only if $[\lambda\lambda^\ast]\in h^+(A)$ is not a zero divisor in $h^+(A)$. Furthermore, we have $h^\ast(A/(\lambda,\lambda^\ast))\cong h^+(A)/([\lambda\lambda^\ast])$.
\end{enumerate}
\end{lemma}
\begin{proof}
Parts (i) and (ii) are proved in \cite[Proposition 4.1]{Zibrowius1}. We prove (iii), part (iv) is similar.

The short exact sequence of $\ast$-modules
\begin{align*}
0\rightarrow A\xrightarrow{\cdot \mu} A\rightarrow A/(\mu)\rightarrow 0
\end{align*}
induces a long exact sequence on Tate cohomology. Using that $h^-(A)=0$ by assumption, it can be written as
\begin{align*}
0\rightarrow h^-(A/(\mu))\rightarrow h^+(A)\xrightarrow{\cdot[\mu]} h^+(A)\rightarrow h^+(A/(\mu))\rightarrow 0
\end{align*}
The assertion now follows immediately.
\end{proof}
The following corollary follows immediately by induction from parts (iii) and (iv) of the previous lemma.
\begin{corollary}\label{tatecohlemma5}
Let $A$ be a $\ast$-ring such that $h^-(A)=0$. Let $\mu_1,\ldots,\mu_n\in A$ be self-dual and $\lambda_1,\ldots,\lambda_m\in A$ such that $\mu_1,\ldots,\mu_n,\lambda_1,\lambda_1^\ast,\ldots,\lambda_m,\lambda_m^\ast$ is a regular sequence in $A$. Then the following are equivalent:
\begin{enumerate}[(1)]
\item The sequence $[\mu_1],\ldots,[\mu_n],[\lambda_1\lambda_1^\ast],\ldots,[\lambda_m\lambda_m^\ast]$ is regular in $h^+(A)$.
\item $h^-(A/(\mu_1,\ldots,\mu_k))=0$ for all $1\leq k\leq n$ and $h^-(A/(\mu_1,\ldots,\mu_n,\lambda_1,\lambda_1^\ast,\ldots,\lambda_j,\lambda_j^\ast))=0$ for all $1\leq j\leq m$.
\end{enumerate}
If these equivalent conditions hold, then
\begin{align*}
h^\ast(A/(\mu_1,\ldots,\mu_n,\lambda_1,\lambda_1^\ast,\ldots,\lambda_m,\lambda_m^\ast))\cong h^+(A)/([\mu_1],\ldots,[\mu_n],[\lambda_1\lambda_1^\ast],\ldots,[\lambda_m\lambda_m^\ast])
\end{align*}
\end{corollary}
Let us finally consider a situation where the element by which we divide is not regular in Tate cohomology, but is a zero divisor with special properties.
\begin{lemma}\label{tatecohlemma6}
Let $A$ be a $\ast$-ring with $h^-(A)=0$ and $\mu\in A$ be a self-dual regular element. Suppose that $\text{Ann}_{h^+(A)}([\mu])=h^+(A)\cdot [\mu]$ so that in particular $[\mu]^2=0$. Then
\[
h^\ast(A/\mu A)\cong \frac{h^+(A)}{([\mu])}\oplus \frac{h^+(A)}{([\mu])}\cdot [\overline{u}]
\]
is a free $h^+(A)/([\mu])$-module of rank 2 where $u\in A$ such that $u+u^\ast=\mu^2$.
\end{lemma}
\begin{proof}
The short exact sequence of $A$-modules
\[
0\to A\xrightarrow{\cdot \mu} A\to A/\mu A\to 0
\]
induces an exact sequence of $h^+(A)$-modules given by
\[
0\to h^-(A/\mu A)\xrightarrow{\partial} h^+(A)\xrightarrow{\cdot [\mu]} h^+(A)\to h^+(A/\mu A)\to 0
\]
We immediately obtain $h^+(A/\mu A)\cong h^+(A)/([\mu])$. Now choose $u\in A$ such that $u+u^\ast=\mu^2$. Then $\overline{u}\in A/\mu A$ defines an element $[\overline{u}]\in h^-(A/\mu A)$ and $\partial([\overline{u}])=[\mu]\in h^+(A)$ by definition of the boundary map. Now $\partial$ induces an isomorphism
\[
h^-(A/\mu A)\xrightarrow[\cong]{\partial} \text{im}(\partial)=\text{Ann}_{h^+(A)}([\mu])=h^+(A)\cdot [\mu]
\]
Since $\text{Ann}_{h^+(A)}([\mu])=h^+(A)\cdot [\mu]$, the map $h^+(A)/([\mu])\xrightarrow{\cong} h^+(A)\cdot [\mu],~\overline{x}\mapsto x\cdot [\mu]$ is a well-defined isomorphism. Thus as an $h^+(A)/([\mu])$-module, $h^-(A/\mu A)$ is freely generated by $[\overline{u}]\in h^-(A/\mu A)$.
\end{proof}

\section{Representation rings and their Tate cohomology}\label{tatecohsubsect3}
Let $G$ be a compact connected Lie group and $H\subset G$ be a centraliser of a torus in $G$. Thus $G/H$ is a complex flag variety. In our approach to computing the Witt ring of $G/H$, it will turn out to be crucial to understand the complex representation ring of $H$ and its Tate cohomology. So it is the purpose of this section to compute these in our cases of interest, i.e. where $G$ is $SU(n)$ (type $A_n$), $Spin(2n+1)$ (type $B_n$) or $Sp(n)$ (type $C_n$). We described all the centralisers of a torus in these groups up to conjugation in the introduction, so we will explicitly go through them one by one in subsections \ref{repntatetypea}, \ref{repntatetypeb} and \ref{repntatetypec}.
\begin{notation}
For a compact Lie group $G$, we denote by $R(G)$ its complex representation ring. Assigning to each complex representation its dual representation induces an involution $\ast\colon R(G)\to R(G)$. Let us identify and fix the notation for the complex representation rings of the compact simply connected Lie groups of type $A_n$, $B_n$, $C_n$ (see \cite[Chapter VI.5-6]{brockertomdieck}):
\[
R(SU(n))\cong \mathbb{Z}\left[x_1,\ldots,x_{n-1}\right]
\]
where $x_1$ is the standard complex representation of rank $n$ and $x_i=\Lambda^i(x_1)$ for all $i$. The duality is given via $x_i^\ast=x_{n-i}$ for all $i$.
\[
R(Sp(n))\cong \mathbb{Z}\left[z_1,\ldots,z_n\right]
\]
where $z_1$ is the standard complex representation of rank $2n$ and $z_i=\Lambda^i(z_1)$ for all $i$. All representations are self-dual.
\begin{align*}
    R(Spin(2n+1))&\cong \mathbb{Z}\left[y_1,\ldots,y_{n-1},\Delta\right]
\end{align*}
where $y_1$ is induced by the standard complex representation of $SO(2n+1)$ of rank $2n+1$, $y_i=\Lambda^i(y_1)$ for all $i$, and $\Delta$ is the \emph{spin representation} of rank $2^n$. All representations of $Spin(2n+1)$ are self-dual. Furthermore, we can express $y_n=\Lambda^n(y_1)$ in terms of the polynomial generators above via the following relation:
\begin{align*}
    \Delta^2&=1+y_1+\ldots+y_{n-1}+y_n
\end{align*}
\end{notation}
We prove a few preliminary facts that we need for our computation.
\begin{lemma}\label{lemmarepresentationsandtheirtatecoh}
Let $G$ be a compact connected Lie group and $C\subset G$ be a finite central subgroup.
\begin{enumerate}[(i)]
\item The projection $G\to G/C$ induces an injection $R(G/C)\to R(G)$ with image $R(G)^C\subset R(G)$, the subring additively generated by all irreducible representations of $G$ via which $C$ acts trivially.
\item $R(G)$ is integral over $R(G/C)$.
\item $h^-(R(G))=0$.
\item The projection $G\to G/C$ induces an injective map $h^+(R(G/C))\to h^+(R(G))$.
\item Let $(\rho_i)_{i\in I}$ be a $\mathbb{Z}$-basis of $R(G)$ where each $\rho_i\colon G\to GL_{n_i}(\mathbb{C})$ is an actual (not just a virtual) representation of $G$. Suppose that each $c\in C$ acts by multiplication by a scalar via all $\rho_i$. Letting $J\subset I$ denote the set of $j\in I$ such that $C$ acts trivially via $\rho_j$, a $\mathbb{Z}$-basis of $R(G/C)$ is given by $(\rho_j)_{j\in J}$.
\end{enumerate}
\end{lemma}
\begin{proof}
Part (i) follows easily from the observation that an irreducible representation of $G$ descends to an irreducible representation of $G/C$ if and only if the restriction of the representation to $C$ yields a trivial representation.

For (ii), suppose $\rho$ is an irreducible representation of $G$. By Schur's lemma, each $c\in C$ acts via $\rho$ by multiplication by some $n_c$th root of unity. Let $N:=\text{lcm}\left\{n_c\mid c\in C\right\}$, then $C$ acts trivially via $\rho^N$. This shows that all irreducible representations of $G$ are integral over $R(G/C)$. Since they generate $R(G)$, the whole of $R(G)$ is integral over $R(G/C)$.

For (iii) note that as a group, $R(G)$ is the free group on isomorphism classes of irreducible representations of $G$. Any irreducible representation is either self-dual or dual to another irreducible representation. This immediately implies $h^-(R(G))=0$.

For (iv), recall that we saw in (i) that $R(G/C)$ is the subgroup of $R(G)$ generated by all isomorphism classes of irreducible representations via which $C$ acts trivially. As a $\mathbb{Z}_2$-vector space, $h^+(R(G))$ has a basis comprising all irreducible self-dual representations of $G$, and $h^+(R(G/C))$ has a basis comprising all irreducible self-dual representations of $G$ via which $C$ acts trivially. This shows that $h^+(R(G/C))\to h^+(R(G))$ is an inclusion. If $\rho\in R(G)$ is an irreducible self-dual representation, then by Schur's Lemma, $C$ acts trivially via $\rho^N$ for some $N\in\mathbb{N}$. Hence $[\rho]^N\in h^+(R(G/C))$ and so $[\rho]\in h^+(R(G))$ is integral over $h^+(R(G/C))$.

It remains to prove (v). Every $\rho_i$ can be written as a sum of irreducible representations. Since every $c\in C$ acts by multiplication by some scalar $\zeta_{c,i}$ via $\rho_i$, it also acts by multiplication by $\zeta_{c,i}$ via every irreducible summand of $\rho_i$. In particular, $C$ acts trivially via $\rho_i$ if and only if it acts trivially via each irreducible summand. This shows that
\begin{align*}
    \bigoplus_{j\in J}\mathbb{Z} \rho_j\subset \bigoplus_{\substack{\sigma \in R(G/C)\\\text{irred.}}}\mathbb{Z}\sigma~\text{ and }~\bigoplus_{i\in I\setminus J} \mathbb{Z}\rho_i\subset \bigoplus_{\substack{\tau\in R(G)\setminus R(G/C)\\\text{irred.}}}\mathbb{Z}\tau
\end{align*}
Since the $\rho_i$ form a basis of $R(G)$, we have
\begin{align*}
    R(G)=\bigoplus_{j\in J}\mathbb{Z} \rho_j\oplus \bigoplus_{i\in I\setminus J} \mathbb{Z}\rho_i\subset \bigoplus_{\substack{\sigma \in R(G/C)\\\text{irred.}}}\mathbb{Z}\sigma\oplus \bigoplus_{\substack{\tau\in R(G)\setminus R(G/C)\\\text{irred.}}}\mathbb{Z}\tau=R(G)
\end{align*}
So the inclusions must in fact be equalities. In particular, $\bigoplus_{j\in J}\mathbb{Z} \rho_j= \bigoplus_{\substack{\sigma \in R(G/C)\\\text{irred.}}}\mathbb{Z}\sigma$. But this is $R(G/C)$ by (i).
\end{proof}
We are now ready to compute the representations rings of centralisers of tori in compact simple Lie groups and their Tate cohomology.
\subsection{Type $A_n$}\label{repntatetypea}
A maximal torus $T$ in $H_A:=S(U(n_1)\times \ldots\times U(n_l))$ is given by all diagonal matrices
\[
\text{diag}\left(z_1^{(1)},\ldots,z_{n_1}^{(1)},z_1^{(2)},\ldots,z_{n_2}^{(2)},\ldots,z_{n_l}^{(l)}\right)
\]
where $z_j^{(p)}=e^{2\pi ia_j^{(p)}}$ with $\sum_{p=1}^l\sum_{j=1}^{n_p} a_j^{(p)}=0$. Simple roots are given by $a_j^{(p)}-a_{j+1}^{(p)}$ for $1\leq p\leq l$ and $1\leq j<n_p$ and so we can deduce that the Weyl group $W$ is $S_{n_1}\times \ldots\times S_{n_l}$, where each symmetric group $S_{n_p}$ acts by permuting the $z_j^{(p)}$ for $1\leq j\leq n_p$. Thus $R(T)^W$ is generated by the elementary symmetric polynomials $\sigma_k^{(p)}=\sum_{i_1<\ldots <i_k}z_{i_1}^{(p)}\ldots z_{i_k}^{(p)}$ together with $\left(\sigma_{n_p}^{(p)}\right)^{-1}=\left(\prod_{i=1}^{n_p}z_{i}^{(p)}\right)^{-1}$ where $1\leq p\leq l$ and $1\leq k\leq n_p$. One can check that $\sigma_k^{(p)}$ is precisely the character of the representation $x_k^{(p)}:=\Lambda^k\left(x_1^{(p)}\right)$, where $x_1^{(p)}$ is the representation of $S(U(n_1)\times \ldots \times U(n_l))$ where the $p$th block $U(n_p)$ acts via the standard representation on $\mathbb{C}^{n_p}$.
\begin{proposition}\label{propositionrepnringtypea}
There is an isomorphism
\[
R(H_A)\cong \bigotimes_{p=1}^l \mathbb{Z}\left[x_1^{(p)},\ldots,x_{n_p}^{(p)}\right]/\left(\prod_{p=1}^l x_{n_p}^{(p)}-1\right)
\]
The duality on $R(H_A)$ is given by  $\left(x_k^{(p)}\right)^\ast=\left(x_{n_p}^{(p)}\right)^{-1} x_{n_p-k}^{(p)}$ under this isomorphism, using the convention that $x_0^{(p)}=1$.
\end{proposition}
\begin{proof}
Letting $S:=\bigotimes_{p=1}^l \mathbb{Z}\left[\left(z_1^{(p)}\right)^{\pm1},\ldots,\left(z_{n_p}^{(p)}\right)^{\pm1}\right]$ be a Laurent ring, the map
\begin{align*}
f\colon \bigotimes_{p=1}^l \mathbb{Z}\left[X_1^{(p)},\ldots,X_{n_p-1}^{(p)}\right]\otimes \bigotimes_{p=1}^l \mathbb{Z}\left[X_{n_p}^{(p)},\left(X_{n_p}^{(p)}\right)^{-1}\right]\to S,~~~ X_i^{(p)}\mapsto \sigma_i^{(p)}
\end{align*}
is injective since the elementary symmetric polynomials are algebraically independent. Now $R(T)=S/\left(\prod_{p=1}^l \prod_{i=1}^{n_p} z_i^{(p)}-1\right)$. Denoting by $\pi\colon S\to R(T)$ the projection map, by the comments above, $\pi\circ f$ maps surjectively onto $R(T)^W$. We have $\text{ker}(\pi\circ f)=f^{-1}\left(\left(\prod_{p=1}^l \prod_{i=1}^{n_p} z_i^{(p)}-1\right)\right)=\left(\prod_{p=1}^l X_{n_p}^{(p)}-1\right)$. The second equality here follows from the fact that $f$ is just the embedding of the subring of $S$ of all elements invariant under the action of $S_{n_1}\times\ldots\times S_{n_l}$. This yields the isomorphism. The fact that the duality is as described follows immediately from considering the duals of the $\sigma_i^{(p)}\in R(T)^W$.
\end{proof}
\begin{proposition}\label{propositiontatecohreptypea}
If not all of $n_1,\ldots, n_l$ are even, we have an isomorphism
\begin{align}\label{equationisotatetypea1}
\bigotimes_{p=1}^l\mathbb{Z}_2\left[\alpha_1^{(p)},\ldots,\alpha_{\lfloor n_p/2\rfloor}^{(p)}\right]\to h^+(R(H_A))
\end{align}
where $\alpha_i^{(p)}\mapsto \left[x_i^{(p)}\left(x_i^{(p)}\right)^\ast\right]$.

If all of $n_1,\ldots,n_l$ are even, we have an isomorphism
\begin{align}\label{equationisotatetypea2}
\frac{\bigotimes_{p=1}^l\mathbb{Z}_2\left[\alpha_1^{(p)},\ldots,\alpha_{\frac{n_p}{2}}^{(p)}\right]\otimes \mathbb{Z}_2[\epsilon]}{\left(\epsilon^2+\prod_p \alpha_{\frac{n_p}{2}}\right)}\to h^+(R(H_A))
\end{align}
where $\alpha_i^{(p)}\mapsto \left[x_i^{(p)}\left(x_i^{(p)}\right)^\ast\right]$ and $\epsilon\mapsto \left[\prod_{p=1}^l x_{\frac{n_p}{2}}^{(p)}\right]$.
\end{proposition}
\begin{proof}
Recall the expression for $R(H_A)$ from Proposition \ref{propositionrepnringtypea}. The monomials in indeterminates $x_k^{(p)}$ and $\left(x_{n_q}^{(q)}\right)^{\pm1}$ for $1\leq p\leq l$ and $1\leq k<n_p$ and $1\leq q<l$ (note that we are missing out $x_{n_l}^{(l)}$) form a $\mathbb{Z}$-basis of $R(H_A)$. Since the duality takes monomials to monomials, a $\mathbb{Z}_2$-basis of $h^+(R(H_A))$ is given by the non-zero self-dual monomials in these indeterminates. Now let
\[
z:=\prod_{p=1}^l \prod_{i=1}^{n_p-1}\left(x_i^{(p)}\right)^{a_i^{(p)}}\cdot \prod_{q=1}^{l-1} \left(x_{n_q}^{(q)}\right)^{b_q}
\]
where $a_i^{(p)}\in\mathbb{N}_0$ and $b_q\in\mathbb{Z}$. Then
\begin{align*}
    z^\ast=&\prod_{p=1}^{l-1} \prod_{i=1}^{n_p-1}\left(x_{n_p}^{(p)}\right)^{-a_i^{(p)}}\left(x_{n_p-i}^{(p)}\right)^{a_i^{(p)}}\cdot\left(x_{n_{l}}^{(l)}\right)^{-a_1^{(l)}-\ldots -a_{n_l-1}^{(l)}}\cdot  \prod_{j=1}^{n_l-1}\left(x_{n_l-j}^{(l)}\right)^{a_j^{(l)}}\cdot \prod_{q=1}^{l-1} \left(x_{n_q}^{(q)}\right)^{-b_q}\\
    =&\prod_{p=1}^{l-1} \prod_{i=1}^{n_p-1}\left(x_{n_p}^{(p)}\right)^{-a_i^{(p)}}\left(x_{n_p-i}^{(p)}\right)^{a_i^{(p)}}\cdot  \prod_{j=1}^{n_l-1}\left(x_{n_l-j}^{(l)}\right)^{a_j^{(l)}}\cdot \prod_{q=1}^{l-1} \left(x_{n_q}^{(q)}\right)^{-b_q+a_1^{(l)}+\ldots+a_{n_l-1}^{(l)}}
\end{align*}
recalling that $x_{n_l}^{(l)}= \left(\prod_{p=1}^{l-1} x_{n_p}^{(p)}\right)^{-1}$ in $R(H_A)$.

So $z=z^\ast$ if and only if the following conditions hold:
\begin{align*}
    a_i^{(p)}=a_{n_p-i}^{(p)}~\text{ and }~2b_q=a_1^{(l)}+\ldots+a_{n_l-1}^{(l)}-a_1^{(q)}-\ldots-a_{n_q-1}^{(q)}
\end{align*}
for all $1\leq p\leq l$ and $1\leq i<n_p$ and all $1\leq q<l$.

If $z=z^\ast$ and some $n_p$ is odd, then we deduce from the above conditions that $a_{\frac{n_q}{2}}^{(q)}$ must be even for all $q$ where $n_q$ is even. In this case, it follows that $z$ is a monomial in indeterminates $x_{i}^{(p)}\left(x_i^{(p)}\right)^\ast$ where $1\leq p\leq l$ and $1\leq i\leq \lfloor n_p/2\rfloor$. This shows that the map (\ref{equationisotatetypea1}) above maps a $\mathbb{Z}_2$-basis to a $\mathbb{Z}_2$-basis, so it is an isomorphism.

If all of $n_1,\ldots,n_l$ are even, then we deduce that the self-dual monomials are in bijective correspondence with the $z$ of the form
\[
z=\left(\prod_{p=1}^l x_{\frac{n_p}{2}}^{(p)}\right)^d\cdot\prod_{q=1}^l\prod_{j=1}^{n_q/2}\left(x_j^{(q)}\left(x_j^{(q)}\right)^\ast\right)^{a_j^{(q)}}
\]
where $d\in\left\{0,1\right\}$ and $a_j^{(p)}\in\mathbb{N}_0$.
Thus we see that the map (\ref{equationisotatetypea2}) maps a $\mathbb{Z}_2$-basis to a $\mathbb{Z}_2$-basis, so it is an isomorphism.
\end{proof}
\subsection{Type $B_n$}\label{repntatetypeb}
We denote by $\tilde{U}(n)$ the connected double cover of $U(n)$ and recall the representation rings of these groups:
\begin{proposition}\label{proprepnringun}
The representation rings of $U(n)$ and $\tilde{U}(n)$ are
\begin{align*}
    R(U(n))\cong~ &\mathbb{Z}\left[x_1,\ldots,x_{n-1},x_n,x_n^{-1}\right]\\
    R(\tilde{U}(n))\cong~ &\mathbb{Z}\left[x_1,\ldots,x_{n-1},(x_n)^{\frac{1}{2}},(x_n)^{-\frac{1}{2}}\right]
\end{align*}
with duality in both cases given by $x_i^\ast=x_n^{-1}\cdot x_{n-i}$. Here $x_1\in R(U(n))$ is the standard representation of $U(n)$ and $x_i:=\Lambda^i(x_1)$. Under this identification, the projection $\pi_n\colon \tilde{U}(n)\to U(n)$ induces the inclusion indicated by the chosen notation, and the non-trivial element in $\text{ker}(\pi_n)$ acts trivially via $x_i\in R(\tilde{U}(n))$ for $1\leq i\leq n$ and by multiplication by $-1$ via $(x_n)^{\frac{1}{2}}$.
\end{proposition}
Let $\tilde{H}_B:=\tilde{U}(n_1)\times \ldots \times\tilde{U}(n_l)\times Spin(2m+1)$. By the product theorem for complex representation rings and Proposition \ref{proprepnringun}, we have
\begin{align*}
    R(\tilde{H}_B)\cong & \bigotimes_{p=1}^l \mathbb{Z}\left[x_1^{(p)},\ldots,x_{n_p-1}^{(p)},\left(x_{n_p}^{(p)}\right)^{\pm\frac{1}{2}}\right]\text{ if }m=0\\
    R(\tilde{H}_B)\cong &\bigotimes_{p=1}^l \mathbb{Z}\left[x_1^{(p)},\ldots,x_{n_p-1}^{(p)},\left(x_{n_p}^{(p)}\right)^{\pm\frac{1}{2}}\right]\otimes \mathbb{Z}[y_1,\ldots,y_{m-1},\Delta]\text{ if }m>0
\end{align*}
By $\epsilon$ we denote the nontrivial element in $\tilde{U}(n_p)$ or $Spin(2m+1)$ in the kernel of $\tilde{U}(n_p)\to U(n_p)$ or $Spin(2m+1)\to SO(2m+1)$. Let $C_B$ be the subgroup of $\tilde{H}_B$ consisting of all elements $(a_1,\ldots,a_l,b)$ where $a_i,b\in\left\{1,\epsilon\right\}$ such that $\epsilon$ occurs an even number of times.

\begin{proposition}\label{propositionrepnringtypeb}
If $m>0$, there is an isomorphism
\begin{align*}
    \bigotimes_{p=1}^l \mathbb{Z}\left[x_1^{(p)},\ldots,x_{n_p-1}^{(p)},\left(x_{n_p}^{(p)}\right)^{\pm1}\right]\otimes \mathbb{Z}[y_1,\ldots,y_{m-1}]\otimes\mathbb{Z}[t]\to R(\tilde{H}_B/C_B)
\end{align*}
where $x_i^{(p)}\mapsto x_i^{(p)}$ for all $1\leq p\leq l$ and $1\leq i\leq n_p$, $y_j\mapsto y_j$ for $1\leq j\leq m-1$ and $t\mapsto \Delta\cdot \prod_{p=1}^l \left(x_{n_p}^{(p)}\right)^{-\frac{1}{2}}$.

If $m=0$, there is an isomorphism
\begin{align*}
    \frac{\bigotimes_{p=1}^l \mathbb{Z}\left[x_1^{(p)},\ldots,x_{n_p-1}^{(p)},\left(x_{n_p}^{(p)}\right)^{\pm1}\right]\otimes\mathbb{Z}[t]}{\left(t^2-\left(\prod_p x_{n_p}^{(p)}\right)^{-1}\right)}\to R(\tilde{H}_B/C_B)
\end{align*}
where $x_i^{(p)}\mapsto x_i^{(p)}$ for all $1\leq p\leq l$ and $1\leq i\leq n_p$ and $t\mapsto \prod_{p=1}^l \left(x_{n_p}^{(p)}\right)^{-\frac{1}{2}}$.

Under the above isomorphisms, the duality is given by
\begin{align*}
\left(x_i^{(p)}\right)^\ast&=\left(x_{n_p}^{(p)}\right)^{-1}x_{n_p-i}^{(p)}~&\text{ for all }&1\leq p\leq l,~1\leq i\leq n_p\\
y_j^\ast&=y_j~&\text{ for all }&1\leq j\leq m-1,\\t^\ast&=t\cdot\prod_p x_{n_p}^{(p)}&&
\end{align*}
\end{proposition}
\begin{remark}
By Lemma \ref{lemmarepresentationsandtheirtatecoh}(i), the ring $R(\tilde{H}_B/C_B)$ is a subring of $R(\tilde{H}_B)$. This justifies our notation, for example regarding $x_i^{(p)}$ as a representation of both $\tilde{H}_B$ and $\tilde{H}_B/C_B$.
\end{remark}
\begin{proof}[of Proposition \ref{propositionrepnringtypeb}]
Suppose $m>0$. Observe that $C_B\subset \tilde{H}_B$ acts trivially via $x_i^{(p)}\in R(\tilde{H}_B)$ and $y_j\in R(\tilde{H}_B)$ for all $1\leq p\leq l$ and $1\leq i\leq n_p$ and for all $1\leq j<m$. Furthermore, $C_B$ acts trivially via $\Delta^b\cdot \prod_{p=1}^l \left(x_{n_p}^{(p)}\right)^{\frac{1}{2}a_p}\in R(\tilde{H}_B)$ if and only if all of $a_p,b$ are even or all of $a_p,b$ are odd, and otherwise some elements of $C_B$ act trivially and the others act by multiplication by $-1$.

This shows by Lemma \ref{lemmarepresentationsandtheirtatecoh}(vi), taking all the monomials as a $\mathbb{Z}$-basis of $R(\tilde{H}_B)$, that the described map is surjective. The ring on the left is an integral domain of Krull dimension $1+m+\sum_p n_p$. By Lemma \ref{lemmarepresentationsandtheirtatecoh}(ii), $R(\tilde{H}_B/C_B)$ is an integral domain with the same Krull dimension as $R(\tilde{H}_B)$, i.e. also $1+m+\sum_p n_p$. So the map must be an isomorphism.

Similar arguments apply for $m=0$.
\end{proof}
We see immediately that $h^+(R(\tilde{H}_B))$ is a polynomial ring over $\mathbb{Z}_2$ in $m+\sum_p \lfloor n_p/2\rfloor$ indeterminates, so $h^+(R(\tilde{H}_B))$ has Krull dimension $m+\sum_p \lfloor n_p/2\rfloor$. So $h^+(R(\tilde{H}_B/C_B))$ has Krull dimension $m+\sum_p \lfloor n_p/2\rfloor$ by Lemma \ref{lemmarepresentationsandtheirtatecoh}(iv).
\begin{proposition}\label{propositiontatecohreptypeb}
If not all of $n_1,\ldots,n_l$ are even, there is an isomorphism
\begin{align}\label{eqnsection2.3typeBn1}
\bigotimes_{p=1}^l\mathbb{Z}_2\left[\alpha_1^{(p)},\ldots,\alpha_{\lfloor n_p/2\rfloor}^{(p)}\right]\otimes \mathbb{Z}_2[\beta_1,\ldots,\beta_{m}]\to h^+(R(\tilde{H}_B/C_B))
\end{align}
where $\alpha_i^{(p)}\mapsto \left[x_i^{(p)}\left(x_i^{(p)}\right)^\ast\right]$ and
\[
\beta_j\mapsto \begin{cases} [y_j]\text{ if }1\leq j\leq m-1\\\left[\Lambda^m(y_1)\right]=\left[tt^\ast -1-y_1-\ldots-y_{m-1}\right]\text{ if }j=m \end{cases}
\]

If all of $n_1,\ldots,n_l$ are even, there is an isomorphism
\begin{align}\label{eqnsection2.3typeBn2}
\frac{\bigotimes_{p=1}^l \mathbb{Z}_2\left[\alpha_1^{(p)},\ldots, \alpha_{ n_p/2}^{(p)}\right]\otimes\mathbb{Z}_2[\beta_1,\ldots,\beta_m]\otimes  \mathbb{Z}_2[\delta]}{\left(\delta^2+(1+\beta_1+\ldots+\beta_m)\alpha_{n_1/2}^{(1)}\ldots \alpha_{n_l/2}^{(l)}\right)}\to h^+(R(\tilde{H}_B/C_B))
\end{align}
where $\alpha_i^{(p)}$ and $\beta_j$ are mapped as in the previous homomorphism and 
\[
\delta\mapsto\begin{cases}\left[\Delta\prod_p \left(x_{n_p}^{(p)}\right)^{-\frac{1}{2}}x_{\frac{n_p}{2}}^{(p)}\right]=\left[t\prod_p x_{\frac{n_p}{2}}^{(p)}\right]\text{ if }m>0\\\left[\prod_p \left(x_{n_p}^{(p)}\right)^{-\frac{1}{2}}x_{\frac{n_p}{2}}^{(p)}\right]=\left[t\prod_p x_{\frac{n_p}{2}}^{(p)}\right]\text{ if }m=0\end{cases}
\]
\end{proposition}
\begin{proof}
Suppose $m>0$. In Proposition \ref{propositionrepnringtypeb}, we saw that $R(\tilde{H}_B/C_B)$ is a tensor product of a Laurent and polynomial ring. The duality takes monomials to monomials, so the non-zero self-dual monomials are a $\mathbb{Z}_2$-basis of $h^+(R(\tilde{H}_B/C_B))$. Suppose 
\[
z=t^a\prod_{p=1}^l \prod_{i=1}^{n_p-1}\left(x_i^{(p)}\right)^{b_i^{(p)}}\cdot \prod_{q=1}^l \left(x_{n_q}^{(q)}\right)^{c_q}\cdot \prod_{q=1}^{m-1}y_q^{d_q}
\]
is a self-dual monomial where $a,b_i^{(p)},d_q\in\mathbb{N}_0$ and $c_p\in\mathbb{Z}$. Then
\[
z=z^\ast=t^a\prod_{r=1}^l\left(x_{n_r}^{(r)}\right)^{a}\prod_{p=1}^l\prod_{i=1}^{n_p-1} \left(x_{n_p}^{(p)}\right)^{-b_i^{(p)}}\left(x_{n_p-i}^{(p)}\right)^{b_i^{(p)}}\cdot \prod_{q=1}^l \left(x_{n_q}^{(q)}\right)^{-c_q}\cdot \prod_{q=1}^{m-1}y_q^{d_q}
\]
This equality holds if and only if
\begin{align*}
    2c_p=a-b_1^{(p)}-\ldots-b_{n_p-1}^{(p)}~\text{ and }~b_i^{(p)}=b_{n_p-i}^{(p)}
\end{align*}
for all $1\leq p\leq l$ and $1\leq i<n_p$. 

Suppose there is some $1\leq p\leq l$ so that $n_p$ is odd. Then $\sum_{k=1}^{n_p-1} b_k^{(p)}=2\sum_{k=1}^{(n_p-1)/2} b_k^{(p)}$ is even and so $a$ is even. Then we must also have that $b_{n_r/2}^{(r)}$ is even when $n_r$ is even. This shows that $z$ can in fact be written as a product of powers of
\[
tt^\ast~\text{ and }~x_j^{(p)}\left(x_j^{(p)}\right)^\ast~\text{ and }~y_q
\]
for $1\leq p\leq l$ and $1\leq j\leq \lfloor n_p/2\rfloor$ and $1\leq q<m$. This proves that the above map (\ref{eqnsection2.3typeBn1}) is surjective, noting that under the isomorphism of Proposition \ref{propositionrepnringtypeb}, $tt^\ast$ corresponds to $\Delta^2=1+y_1+\ldots+y_m$.

Now suppose that $n_1,\ldots,n_l$ are all even. If we choose $a$ odd, we must also choose $b_{n_p/2}^{(p)}$ odd for all $1\leq p\leq l$. We deduce that $z$ can be written as a product of powers of
\[
tt^\ast~\text{ and }~t\prod_{p=1}^l x_{\frac{n_p}{2}}^{(p)}~\text{ and }~x_j^{(p)}\left(x_j^{(p)}\right)^\ast~\text{ and }~y_q
\]
for $1\leq p\leq l$ and $1\leq i<n_p$ and $1\leq q<m$. This proves that the above map (\ref{eqnsection2.3typeBn2}) is surjective, again noting that under the isomorphism of Proposition \ref{propositionrepnringtypeb}, $tt^\ast$ corresponds to $\Delta^2=1+y_1+\ldots+y_m$.

Note that in both cases, both the domain and codomain of the surjective maps are integral domains with the same Krull dimension $m+\sum_{i=1}^l n_i$. Hence the maps must in fact be isomorphisms.

Similar arguments apply for $m=0$.
\end{proof}

\subsection{Type $C_n$}\label{repntatetypec}
Let $H_C:=U(n_1)\times\ldots\times U(n_l)\times Sp(m)$. From the product theorem for complex representation rings, we immediately obtain:
\begin{proposition}\label{propositionrepnringtypec}
There is an isomorphism
\[
\bigotimes_{p=1}^l \mathbb{Z}\left[x_1^{(p)},\ldots,x_{n_p-1}^{(p)},\left(x_{n_p}^{(p)}\right)^{\pm1}\right]\otimes \mathbb{Z}[z_1,\ldots,z_m]\to R(H_C)
\]
where $x_1^{(p)}$ corresponds to the standard action of the $p$th block $U(n_p)$ and $x_i^{(p)}=\Lambda^i\left(x_1^{(p)}\right)$, and $z_1$ corresponds to the standard representation of the block $Sp(m)$ on $\mathbb{C}^{2m}$ and $z_j=\Lambda^j(z_1)$. Under this isomorphism, the duality on $R(H_C)$ is given by $\left(x_i^{(p)}\right)^\ast=\left(x_{n_p}^{(p)}\right)^{-1}x_{n_p-i}^{(p)}$ and $z_j^\ast=z_j$.
\end{proposition}
\begin{proposition}\label{propositiontatecohreptypec}
We have an isomorphism
\[
\bigotimes_{p=1}^l \mathbb{Z}_2\left[\alpha_1^{(p)},\ldots,\alpha_{\lfloor n_p/2\rfloor}^{(p)}\right]\otimes \mathbb{Z}_2\left[\gamma_1,\ldots,\gamma_m\right]\to h^+(R(H_C))
\]
where $\alpha_i^{(p)}\mapsto \left[x_i^{(p)}\left(x_i^{(p)}\right)^\ast\right]$ and $\gamma_j\mapsto [z_j]$.
\end{proposition}

\section{Outline of the Computation of Witt Rings}\label{chapteroutlinecomputation}
We now outline our method of computation of the Witt ring of complex flag varieties in detail. We use the approach developed in \cite{Zibrowius1}. There the author computes the Witt ring of all full flag varieties. We slightly generalise the approach to be able to apply it to all flag varieties.

Throughout this section, let $G$ be a compact simply connected Lie group and $H$ be a closed connected subgroup of maximal rank. We denote by $i\colon H\to G$ the inclusion map. We have seen that by Bousfield's lemma, $W^\ast(G/H)\cong h^\ast(K^0(G/H))$. So as a first step, we need to be able to compute $K^0(G/H)$. This is done via a theorem of Hodgkin reducing this to a computation with representation rings of $G$ and $H$, which are well-understood. We then use the results of section \ref{tatecohsubsect2} to compute $h^\ast(K^0(G/H))$ and finally see how to determine the Witt grading of $h^\ast(K^0(G/H))\cong W^\ast(G/H)$.

\subsection{K-theory of $G/H$}
The key to the complex K-theory of $G/H$ is a certain ring homomorphism $\alpha\colon R(H)\to K^0(G/H)$ which was already considered by Atiyah and Hirzebruch in \cite{AtiyahHirzebruch1}. It is defined as follows: Let $\rho\colon H\to GL_n(\mathbb{C})$ be a complex representation of $H$. Then we may regard $G\times \mathbb{C}^n$ as an $H$-space by letting $H$ act on $G$ by right multiplication and on $\mathbb{C}^n$ via $\rho$. This yields an $n$-dimensional complex vector bundle $(G\times \mathbb{C}^n)/H\to G/H$ and thus induces an additive homomorphism $\alpha\colon R(H)\to K^0(G/H)$. 

It can be checked that $\alpha$ is a ring homomorphism and preserves the dualities on $R(H)$ and $K^0(G/H)$, so it is a morphism of $\ast$-rings. Letting $\mathfrak{a}(G)\subset R(H)$ be the ideal generated by all $i^\ast(\sigma)-\text{rk}(\sigma)\in R(H)$ for $\sigma\in R(G)$, one can see that $\mathfrak{a}(G)\subset\text{ker}(\alpha)$, so $\alpha$ induces a map
\[
\overline{\alpha}\colon R(H)/\mathfrak{a}(G)\to K^0(G/H)
\]
Recalling our standing assumption that $G$ is compact, simply connected and that $H\subset G$ is of maximal rank, Hodgkin's theorem states:
\begin{theorem}[{\cite[Thm 3]{pittie1}}] \label{Hodgkin}
$\overline{\alpha}$ is an isomorphism.
\end{theorem}
\begin{remark}
The above construction can be mimicked for KO-theory. Fist, we extend the real representation ring to a graded ring using equivariant K-theory:
\begin{align*}
RO^j(H):= KO^j_H(\text{pt})~\text{ for }j\in\mathbb{Z}
\end{align*}
Similarly to the non-equivariant case, equivariant real K-theory is 8-periodic.
As outlined in \cite[§2.2 and §2.3]{Zibrowius1} for example, we can now apply the aforementioned construction analogously to real representations of $H$ and real vector bundles over $G/H$ to obtain maps
$\alpha_O^j\colon RO^j(H)\to KO^j(G/H)$.
Let $\widetilde{RO}^j(G)$ be the kernel of the restriction map $RO^j(G)\to RO^j(\left\{1\right\})$ and denote by
$
\mathfrak{a}_O^\ast(G)\subset RO^\ast(H)
$
the ideal generated by the image of $\widetilde{RO}^\ast(G)$ under the restriction map $i^\ast$ and by the images of $\mathfrak{a}(G)$ under the realification map $r_\ast$. Then $\alpha_O^j$ induces a map
\[
\overline{\alpha}_O^j\colon RO^j(H)/\mathfrak{a}_O^j(G)\to KO^j(G/H)
\]
See \cite[§2.3]{Zibrowius1} for all of this. However, \cite[Proposition 2.2, Example 2.3]{Zibrowius1} shows that $\overline{\alpha}_O^\ast$ is often far from being surjective.
\end{remark}
A crucial ingredient in the proof of Theorem \ref{Hodgkin} is the following result, which we shall also need later:
\begin{theorem}[{\cite[Thm. 1.1]{steinberg1}}]\label{Steinberg}
Let $K$ be a connected compact Lie group with $\pi_1(K)$ free and $U$ be a closed connected subgroup of maximal rank. Then $R(U)$ is free as a module over $R(K)$ by restriction.
\end{theorem}
Let us now apply Hodgkin's theorem \ref{Hodgkin} concretely. Since $G$ is simply connected, its representation ring is a polynomial ring:
\[
R(G)\cong \mathbb{Z}[\lambda_1,\lambda_1^\ast,\ldots,\lambda_a,\lambda_a^\ast,\sigma_{a+1},\ldots,\sigma_n]
\]
Here we suppose that the $\sigma_j$ are self-dual. If $\zeta$ is a representation, let us write $\tilde{\zeta}:=\zeta-\text{rk}(\zeta)$ for the reduced virtual representation of rank 0. As $R(H)$ is a free $R(G)$-module by Theorem \ref{Steinberg}, it follows that
\[
i^\ast(\tilde{\lambda}_1),i^\ast(\tilde{\lambda}_1^\ast),\ldots,i^\ast(\tilde{\lambda}_a),i^\ast(\tilde{\lambda}_a^\ast),i^\ast(\tilde{\sigma}_{a+1}),\ldots,i^\ast(\tilde{\sigma}_n)
\]
is an $R(H)$-regular sequence. Letting $I$ be the ideal generated by all these elements, we obtain by Theorem \ref{Hodgkin} that
\begin{align}\label{K0}
\overline{\alpha}\colon R(H)/I\xrightarrow[]{\cong} K^0(G/H)
\end{align}
is an isomorphism of $\ast$-rings.

\subsection{Tate cohomology of $K^0(G/H)$}
Via the isomorphism (\ref{K0}), it is now possible to compute $h^\ast(K^0(G/H))$ by computing $h^\ast(R(H)/I)$ using the results of section \ref{tatecohsubsect2}. We define the following elements in $h^+(R(H))$:
\begin{align*}
\mu_j:=\begin{cases}\left[i^\ast\left(\tilde{\lambda}_j\tilde{\lambda}_j^\ast\right)\right] &\text{ if }1\leq j\leq a\\\left[i^\ast\left(\tilde{\sigma}_j\right)\right]&\text{ if }a<j\leq n\end{cases}
\end{align*}
We make the following assumptions:
\begin{assumptions}
\item[(A1)] There is a subset $S\subset \left\{1,\ldots,n\right\}$ such that the $\mu_s$ for $s\in S$ form an $h^+(R(H))$-regular sequence in some order.
\item[(A2)] For every $t\in \left\{1,\ldots,n\right\}$, the element $\mu_t$ is contained in the ideal $(\mu_s\mid s\in S)$.
\end{assumptions}
As we will see in subsequent chapters, these assumptions are frequently satisfied in our concrete computations. In fact, showing that they hold will be the essential remaining step. However, we will see that in a few situations, these assumptions are not quite satisfied. In those cases, we will have to make slight adaptions.

Let
\[
I':=\left(\tilde{\lambda}_s,\tilde{\lambda}_s^\ast\mid s\in S,~1\leq s\leq a\right)+\left(\tilde{\sigma}_s\mid s\in S,~a<s\leq n\right)\subset R(H)
\]
From (A1), it follows by Corollary \ref{tatecohlemma5} that as rings,
\[
h^\ast\left(R(H)/I'\right)\cong h^+(R(H)/I')\cong h^+(R(H))/(\mu_s\mid s\in S)
\]
Let $\overline{S}:= \left\{1,\ldots,n\right\}\setminus S$. By (A2), for every $t\in \overline{S}$ there exists $u_t\in R(H)$ such that:
\begin{align*}
\begin{rcases}
\text{If }1\leq t\leq a\text{, then }~\overline{i^\ast\left(\tilde{\lambda}_t\tilde{\lambda}^\ast_t\right)}&=\overline{u}_t+\overline{u}_t^\ast\\
\text{If }a< t\leq n\text{, then }~\overline{i^\ast(\tilde{\sigma}_t)}&=\overline{u}_t+\overline{u}_t^\ast\\
\end{rcases}\text{ in }R(H)/I'
\end{align*}
Now by Lemma \ref{tatecohlemma4}, we have that as modules over $h^\ast(R(H)/I')$,
\begin{align}\label{isotatecoh}
h^\ast(R(H)/I)\cong h^\ast(R(H)/I')\otimes \bigwedge_{t\in\overline{S}} \left([\overline{u}_{t}]\right)
\end{align}
where $[\overline{u}_{t}]\in h^-(R(H)/I)$ for all $t\in\overline{S}$.

To show that (\ref{isotatecoh}) is also a ring isomorphism, it suffices to show that
$
[\overline{u}_t]^2=0~\text{ for all }t\in\overline{S}
$
This is immediately implied by the following Proposition, which is also of independent interest:
\begin{proposition}
Let $X$ be a finite cell complex with $K^1(X)=0$ and $x\in W^{-1}(X)\oplus W^{-3}(X)$. Then $x^2=0$.
\end{proposition}
\begin{proof}
Let $x=\overline{x}_{-1}+\overline{x}_{-3}$ where $x_i\in KO^{2i}(X)$ for $i\in\left\{-1,-3\right\}$. We consider the commutative square
\[
\begin{CD}
W^{-1}(X)\oplus W^{-3}(X) @>\overline{c}_{\text{odd}}=(\overline{c}_{-1},\overline{c}_{-3})>\cong> h^-(K^0(X))\\
@VVV @VVV\\
W^0(X)\oplus W^{-2}(X) @>\overline{c}_{\text{even}}=(\overline{c}_{0},\overline{c}_{-2})>\cong> h^+(K^0(X))
\end{CD}
\]
where the vertical maps are the squaring maps and the horizontal maps (induced by complexification) are isomorphisms by Bousfield's lemma \ref{lemmabousfield}. We have
\begin{align}\label{eqnsquarezero1}
\overline{c}_{\text{even}}(x^2)=\overline{c}_{\text{odd}}(x)^2=[c_{-1}(x_{-1})]^2+[c_{-3}(x_{-3})]^2
\end{align}
In $K^0(X)$, we have for $i\in\left\{-1,-3\right\}$ that $c_i(x_i)+c_i(x_i)^\ast=0$. Hence in $h^-(K^0(X))$,
\[
[c_{i}(x_{i})]=[-c_i(x_i)^\ast]=[c_i(x_i)^\ast]
\]
where the second equality holds since $h^-(K^0(X))$ is 2-torsion. Thus from (\ref{eqnsquarezero1}),
\begin{align}\label{eqnsquarezero2}
\overline{c}_{\text{even}}(x^2)=[c_{-1}(x_{-1})\cdot c_{-1}(x_{-1})^\ast]+[c_{-3}(x_{-3})\cdot c_{-3}(x_{-3})^\ast]
\end{align}
Now for any $y\in K^0(X)$, it is a well-known fact that $y\cdot y^\ast\in\text{im}(c_0)$. So (\ref{eqnsquarezero2}) implies that $\overline{c}_{\text{even}}(x^2)\in \text{im}(\overline{c}_0)$. Thus $x^2\in W^0(X)$.

On the other hand,
$
x^2=[x_{-1}]^2+[x_{-3}]^2\in W^{-2}(X)
$.
Consequently, $x^2\in W^0(X)\cap W^{-2}(X)$. So $x^2=0$ as asserted.
\end{proof}
In summary, we have proved:
\begin{proposition}\label{propositiontatecohofktheory}
Using the notation introduced above and assuming that (A1) and (A2) hold, we have a ring isomorphism
\begin{align*}
h^\ast(R(H)/I)\cong\frac{h^+(R(H))}{(\mu_s\mid s\in S)}\otimes \bigwedge_{t\in\overline{S}} \left([\overline{u}_{t}]\right)
\end{align*}
where the first factor in the tensor product is completely contained in $h^+$ and all the generators of the exterior algebra on the right are contained in $h^-$.
\end{proposition}

\subsection{Witt ring of $G/H$}
We have now computed the Tate cohomology of $R(H)/I$ and thus also the Tate cohomology of $K^0(G/H)$ via the isomorphism
\[
[\overline{\alpha}]\colon h^\ast(R(H)/I)\xrightarrow{\cong}h^\ast(K^0(G/H))
\]
induced by (\ref{K0}). It remains to determine the Witt grading of $h^\ast(K^0(G/H))$ under the isomorphism
\[
W^\ast(G/H)\xrightarrow{\overline{c}} h^\ast(K^0(G/H))
\]
of Bousfield's lemma. Using the notation as introduced in this chapter, we prove two lemmas to this end. They are generalisations of the assertion about the Witt grading of the Tate cohomology of full flag varieties in \cite[Thm 3.3]{Zibrowius1}.
\begin{lemma}\label{gradinglemma1}
Let $x\in R(H)$ be self-dual. Then it gives rise to an element $[\alpha(x)]\in h^+(K^0(G/H))$.\\If $x$ is a real representation, then $[\alpha(x)]$ corresponds to an element in $W^0(G/H)$ under Bousfield's isomorphism $\overline{c}$.\\If $x$ is a quaternionic representation, then $[\alpha(x)]$ corresponds to an element in $W^{-2}(G/H)$ under Bousfield's isomorphism $\overline{c}$.
\end{lemma}
\begin{proof}
Suppose $x$ is real. Then there is $x^\mathbb{R}\in RO^0(H)$ such that $x=c_0\left(x^\mathbb{R}\right)$. From the commutative diagram
\begin{align*}
\begin{CD}
RO^0(H) @>c_0>> R(H)\\
@VV\alpha_O^0V @VV\alpha V\\
KO^0(G/H) @>c_0>> K^0(G/H)
\end{CD}
\end{align*}
we deduce that 
$
\alpha(x)=\alpha\left(c_0\left(x^\mathbb{R}\right)\right)=c_0\left(\alpha_O^0\left(x^\mathbb{R}\right)\right)
$.
Thus $\alpha(x)\in \text{im}(c_0)$. Passing to Tate cohomology, we obtain
\[
[\alpha(x)]\in \text{im}(\overline{c}_0\colon W^0(G/H)\to h^+(K^0(G/H)))
\]
and hence the claim follows.

Similarly, if $x$ is quaternionic, we obtain from the commutative square
\begin{align*}
\begin{CD}
RO^4(H) @>c_2>> R(H)\\
@VV\alpha_O^4V @VV\alpha V\\
KO^4(G/H) @>c_2>> K^0(G/H)
\end{CD}
\end{align*}
that $\alpha(x)\in \text{im}(c_2)$ and hence $[\alpha(x)]\in h^+(K^0(G/H))$ corresponds to an element in $W^{-2}(G/H)$ under Bousfield's isomorphism $\overline{c}$.
\end{proof}
\begin{lemma}\label{gradinglemma2}
Let $u\in R(H)$ be such that there are $\nu_j\in R(H)$ and $\tilde{\mu}_j\in R(G)$ with $\text{rk}\left(\tilde{\mu}_j\right)=0$ such that
$
u+u^\ast=\sum_{j=1}^n \nu_j \cdot i^\ast\left(\tilde{\mu}_j\right)
$.
Then $u$ gives rise to $[\alpha(u)]\in h^-(K^0(G/H))$. 

If all the $\nu_j,\tilde{\mu}_j$ are of real type, then $[\alpha(u)]\in h^-(K^0(G/H))$ corresponds to an element in $W^{-1}(G/H)$ under Bousfield's isomorphism. 

If all the $\tilde{\mu}_j$ are of real type and all the $\nu_j$ are of quaternionic type, then $[\alpha(u)]\in h^-(K^0(G/H))$ corresponds to an element in $W^{-3}(G/H)$ under Bousfield's isomorphism.
\end{lemma}
\begin{remark}
Note that to apply this lemma to the $\left[\overline{u}_t\right]$ where $t\in\overline{S}$, we need to assume something stronger than (A2). We will see that in our concrete computations, this stronger condition is satisfied.
\end{remark}
\begin{proof}
First suppose all the $\nu_j,\tilde{\mu}_j$ are of real type. Then there exist $\nu_j^\mathbb{R}\in RO^0(H)$ and $\tilde{\mu}_j^\mathbb{R}\in RO^0(G)$ with $\nu_j=c_0\left(\nu_j^\mathbb{R}\right)$ and $\tilde{\mu}_j=c_0\left(\tilde{\mu}_j^\mathbb{R}\right)$. Thus we have
\[
c_0(r_0(u))=u+u^\ast=c_0\left(\sum_j \nu_j^\mathbb{R}\cdot i^\ast\left(\tilde{\mu}_j^\mathbb{R}\right)\right)
\]
Since $c_0$ is injective, we have $r_0(u)=\sum_j \nu_j^\mathbb{R}\cdot i^\ast\left(\tilde{\mu}_j^\mathbb{R}\right)\in\mathfrak{a}_O^0(G)$. From the commutative square
\begin{align*}
\begin{CD}
R(H)/\mathfrak{a}(G) @>\overline{\alpha}>> K^0(G/H)\\
@VV\overline{r}_0V @VVr_0V\\
RO^0(H)/\mathfrak{a}_O^0(G) @>\overline{\alpha}_O^0>> KO^0(G/H)
\end{CD}
\end{align*}
we see that $\alpha(u)\in \text{ker}(r_0)$. Hence\footnote{Recall that $c\backslash KO^0(G/H)$ denotes the kernel of $c_0\colon KO^0(G/H)\to K^0(G/H)$.} $[\alpha(u)]\in \text{ker}\left(\overline{r}_0\colon h^-(K^0(G/H))\to c\backslash KO^0(G/H)\right)$ and so $[\alpha(u)]\in \text{im}\left(\overline{c}_3\colon W^3(G/H)\to h^-(K^0(G/H))\right)$ by Bousfield's lemma \ref{lemmabousfield}. Thus the claim follows.

Now suppose the $\tilde{\mu}_j$ are real and the $\nu_j$ are quaternionic. Then there are $\tilde{\mu}_j^\mathbb{R}\in RO^0(H)$ and $\nu_j^\mathbb{H}\in RO^4(H)$ such that $c_0\left(\tilde{\mu}_j^\mathbb{R}\right)=\tilde{\mu}_j$ and $c_2\left(\nu_j^\mathbb{H}\right)=\nu_j$. Thus we have
\[
c_2(r_2(u))=u+u^\ast=\sum_{j=1}^n c_2\left(\nu_j^\mathbb{H}\right)\cdot c_0\left(i^\ast\left(\tilde{\mu}_j^\mathbb{R}\right)\right)=c_2\left(\sum_{j=1}^n \nu_j^\mathbb{H}\cdot i^\ast\left(\tilde{\mu}_j^\mathbb{R}\right)\right)
\]
Since $c_2$ is injective, we obtain
$
r_2(u)=\sum_{j=1}^n \nu_j^\mathbb{H}\cdot i^\ast\left(\tilde{\mu}_j^\mathbb{R}\right)\in\mathfrak{a}_O^4(G)
$.
From the commutative square
\begin{align*}
\begin{CD}
R(H)/\mathfrak{a}(G) @>\overline{\alpha}>> K^0(G/H)\\
@VV\overline{r}_2V @VVr_2V\\
RO^4(H)/\mathfrak{a}_O^4(G) @>\overline{\alpha}_O^4>> KO^4(G/H)
\end{CD}
\end{align*}
we see that $\alpha(u)\in\text{ker}(r_2)$. So by Bousfield's lemma \ref{lemmabousfield}, the claim follows.
\end{proof}
\begin{remark}\label{remarkarbitraryfield}
As mentioned in the introduction, our computations could also be carried out more generally in the algebraic context for flag varieties over an algebraically closed field of characteristic not two. Let us explain this in more detail. Suppose $G$ is a semisimple algebraic group over an algebraically closed field of characteristic not two and $P$ is a parabolic subgroup.
\begin{itemize}
\item An analogue of Bousfield's lemma holds for $G/P$, the algebraic K-theory and algebraic Witt ring (in the sense of Balmer) replacing their topological counterparts \cite[Thm 2.3]{Zibrowius3}.
\item Panin's theorem \cite[Thm 2.2]{panin} gives an analogue of Hodgkin's theorem, computing the algebraic K-theory of $G/P$ in terms of the algebraic representation rings of $G$ and $P$.
\item The representation ring of $G$ and $P$ is essentially independent of the ground field (cf. \cite[Th\'{e}or\`{e}me 4]{Serre2} and \cite[Lemma 3.4]{Zibrowius3}).
\end{itemize}
The above facts show that our computations also yield the algebraic Witt rings of the respective flag varieties over all algebraically closed fields of characteristic not two.
\end{remark}

\section{Computation of Witt rings}\label{sectioncomputation}
We now apply the approach explained in the previous section to compute the Witt ring of all complex flag varieties $G/H$ where $G=SU(n),~Spin(2n+1)$ or $Sp(n)$, i.e. where $G$ is simple of ordinary type $A_n,~B_n$ or $C_n$. We also state the result for type $D_n$ and refer to \cite{Hemmert} for a detailed proof.
\subsection{Type $A_n$}\label{sectioncomputationtypea}
\begin{notationconventions}
Up to conjugation every centraliser of a torus in $SU(n)$ is of the form $H_A=S(U(n_1)\times \ldots \times U(n_l))$ where $n_1+\ldots+n_l=n$. We let $k$ be the number of integers among $n_1,\ldots,n_l$ that are odd and $i_A\colon H_A\to SU(n)$ be the inclusion and $X_A:= SU(n)/H_A$ be the corresponding complex flag variety. We use the notation for $R(H_A)$ and $h^+(R(H_A))$ introduced in Propositions \ref{propositionrepnringtypea} and \ref{propositiontatecohreptypea} and write $R(SU(n))=\mathbb{Z}\left[x_1,\ldots,x_{n-1}\right]$ where $x_1$ is the standard complex representation of $SU(n)$ of rank $n$ and $x_i=\Lambda^i(x_1)$ for all $i$. Furthermore, it will be convenient to make the following definitions in $h^+(R(H_A))$:
\begin{align}\label{defineallbetas}
\begin{split}
\alpha_0^{(p)}:=1&\text{ for }1\leq p\leq l,\\
\alpha_{i}^{(p)}:=\alpha_{n_p-i}^{(p)}&\text{ for }1\leq p\leq l\text{ and } \lfloor n_p/2\rfloor<i\leq n_p
\end{split}
\end{align}
Note that with these definitions, we have $\alpha_i^{(p)}=\left[x_i^{(p)}\left(x_i^{(p)}\right)^\ast\right]$ for all $1\leq p\leq l$ and all $0\leq i\leq n_p$.
\end{notationconventions}
To compute $K^0(X)$ using Hodgkin's theorem, we need to determine the induced map
\begin{align*}
i_A^\ast\colon \mathbb{Z}\left[x_1,\ldots,x_{n-1}\right]\cong R(SU(n))\longrightarrow R(H_A)\cong \bigotimes_{p=1}^l \mathbb{Z}\left[x_1^{(p)},\ldots,x_{n_p}^{(p)}\right]/\left(\prod_{p=1}^l x_{n_p}^{(p)}-1\right)
\end{align*}
We see directly that $i_A^\ast(x_1)=x_1^{(1)}+\ldots +x_1^{(l)}$. Hence we obtain for $1\leq j\leq n-1$:
\begin{align*}
i_A^\ast\left(x_j-\text{rk}(x_j)\right)
= \Lambda^j\left(x_1^{(1)}+\ldots+x_1^{(l)}\right)-\binom{n}{j}
= \sum_{a_1+\ldots+a_l=j}x_{a_1}^{(1)}\ldots x_{a_l}^{(l)}-\binom{n}{j}=: P_j
\end{align*}
Then by Hodgkin's theorem \ref{Hodgkin},
\begin{align*}
K^0(X)\cong R(H_A)/(P_1,\ldots ,P_{n-1})
\end{align*}
We want to compute the Tate cohomology of this ring using Proposition \ref{propositiontatecohofktheory}. We have $P_j^\ast=P_{n-j}$ for $1\leq j\leq n-1$. Thus if $n$ is odd, we have $(n-1)/2$ mutually conjugate pairs of $P_j$'s and if $n$ is even, we have $n/2-1$ mutually conjugate pairs of $P_j$'s and one self-conjugate one, namely $P_{n/2}$. So in order to use Proposition \ref{propositiontatecohofktheory}, we need to consider the elements $\left[P_jP_j^\ast\right]$ in $h^+(R(H_A))$. Using repeatedly that $[a+a^\ast]=0$ in $h^+(R(H))$ for all $a\in R(H)$, we compute:
\begin{align*}
\left[P_jP_j^\ast\right]=&\left[\left(\sum_{a_1+\ldots+a_l=j}x_{a_1}^{(1)}\ldots x_{a_l}^{(l)}-\binom{n}{j}\right)\cdot\left(\sum_{b_1+\ldots+b_l=j}x_{b_1}^{(1)}\ldots x_{b_l}^{(l)}-\binom{n}{j}\right)^\ast\right]\\
=&\left[\sum_{\substack{a_1+\ldots+a_l=j\\b_1+\ldots+b_l=j}}x_{a_1}^{(1)}\left(x_{b_1}^{(1)}\right)^\ast\ldots x_{a_l}^{(l)}\left(x_{b_l}^{(l)}\right)^\ast+\binom{n}{j}\right]\\
=& \left[\sum_{a_1+\ldots+a_l=j}x_{a_1}^{(1)}\left(x_{a_1}^{(1)}\right)^\ast\ldots x_{a_l}^{(l)}\left(x_{a_l}^{(l)}\right)^\ast+\binom{n}{j}\right]\\
=&\sum_{a_1+\ldots+a_l=j}\alpha_{a_1}^{(1)}\ldots\alpha_{a_l}^{(l)}+\binom{n}{j}=:\mu_j
\end{align*}
For the sum defining the $\mu_j$, remember the definitions we made in (\ref{defineallbetas}).

If $n$ is even, we also need to consider
\begin{align*}
[P_{n/2}]=\left[\sum_{a_1+\ldots+a_l=n/2}x_{a_1}^{(1)}\ldots x_{a_l}^{(l)}-\binom{n}{\frac{n}{2}}\right]=
\begin{cases}
0&\text{ if }k>0\\\epsilon&\text{ if }k=0
\end{cases}
\end{align*}
This equality holds for the following reason: If $k>0$, then none of the summands $x_{a_1}^{(1)}\ldots x_{a_l}^{(l)}$ with $a_1+\ldots+a_l=\frac{n}{2}$ is self-dual, and if $k=0$, the only self-dual summand is $x_{n_1/2}^{(1)}\ldots x_{n_l/2}^{(l)}$.

In order to finally apply Proposition \ref{propositiontatecohofktheory}, we need to investigate the relations between the $\mu_j\in h^\ast(R(H_A))$. In other words, we need to check that (A1) and (A2) from Chapter \ref{chapteroutlinecomputation} are satisfied. We state the result now, postponing the proof to Propositions \ref{propmuregular} and \ref{propositionrelationsmu} in the next section.
\begin{proposition}\label{polynomialrelationstypea}
Setting $m:=\lfloor n_1/2\rfloor+\ldots+\lfloor n_l/2\rfloor$, the elements $\mu_1,\ldots,\mu_m$ form an $h^\ast(R(H_A))$-regular sequence and each $\mu_i$ for $m<i\leq \lfloor n/2\rfloor$ can be written as a $\mathbb{Z}_2$-linear combination of $\mu_1,\ldots,\mu_m$.
\end{proposition}
\begin{remark}
If $k=0$, then $\mu_m=\epsilon^2$ and so we have that $\mu_1,\ldots,\mu_{m-1},\epsilon=\left[P_{n/2}\right]$ also form an $h^\ast(R(H_A))$-regular sequence.
\end{remark}
Let $r:=\lfloor n/2\rfloor -m$. Note that if $k=0$, then $r=0$. As a consequence of Proposition \ref{polynomialrelationstypea}, we can find $u_1,\ldots,u_r\in R(H_A)$  such that
\begin{align}\label{defineutypea}
u_j+u_j^\ast-P_{m+j}P_{m+j}^\ast&\in\mathbb{Z}\cdot P_1P_1^\ast+\ldots+\mathbb{Z}\cdot P_mP_m^\ast\text{ for }\begin{cases}1\leq j\leq r\text{ if }n\text{ is odd}\\1\leq j<r\text{ if }n\text{ is even}\end{cases}\\
u_r+u_r^\ast&=P_{n/2}\text{ if }n\text{ is even}\nonumber
\end{align}
These give rise to $[\overline{u}_i]\in h^-(R(H_A)/(P_1,\ldots,P_{n-1}))$. Now from Proposition \ref{propositiontatecohofktheory}, we immediately obtain:
\begin{proposition}
Let $m=\lfloor n_1/2\rfloor+\ldots+\lfloor n_l/2\rfloor$ and $r=\lfloor n/2\rfloor-m$. Then
\begin{align}\label{tatecohomologytypea}
h^\ast(R(H_A)/(P_1,\ldots,P_{n-1}))\cong \frac{\bigotimes_{p=1}^l \mathbb{Z}_2\left[\alpha_1^{(p)},\ldots,\alpha_{\lfloor n_p/2\rfloor}^{(p)}\right]}{(\mu_1,\ldots,\mu_m)}\otimes \bigwedge([\overline{u}_1],\ldots,[\overline{u}_r])
\end{align}
where
\[
\mu_j=\sum_{a_1+\ldots+a_l=j}\alpha_{a_1}^{(1)}\ldots \alpha_{a_l}^{(l)}+\binom{n}{j}\text{ for }1\leq j\leq m,
\]
recalling and sticking to the definitions we made in (\ref{defineallbetas}). We have that $\alpha_j^{(p)}\in h^+$ for all $1\leq p\leq l$ and $1\leq j\leq n_p$, and $u_i\in h^-$ for all $1\leq i\leq r$.
\end{proposition}
Recall from Bousfield's lemma \ref{lemmabousfield} that there is an isomorphism
\[
W^\ast(X_A)\xrightarrow[\cong]{\overline{c}}h^\ast(K^0(X_A))\xleftarrow[\cong]{[\overline{\alpha}]}h^\ast(R(H_A)/(P_1,\ldots,P_{n-1}))
\]
We have computed the 2-graded Tate cohomology on the right and now want to determine the 4-periodic grading of the Witt ring.

Let $b_i^{(p)}\in W^\ast(X_A)$ be the element corresponding to $\overline{\alpha}_i^{(p)}\in h^\ast(R(H_A)/(P_1,\ldots,P_{n-1}))$ under the above isomorphism for $1\leq p\leq l$ and $0\leq i\leq n_p$. We see that $\overline{\alpha}_i^{(p)}\in h^\ast(R(H_A)/(P_1,\ldots,P_{n-1}))$ is represented by a real representation for all $p$ and $i$ as $\rho\rho^\ast$ is always of real type for any complex representation $\rho$. So by Lemma \ref{gradinglemma1},
\[
b_i^{(p)}\in W^0(X_A)\text{ for all }1\leq p\leq l\text{ and }0\leq i\leq n_p
\]
For $1\leq i\leq r$, let $v_i\in W^\ast(X_A)$ be the element corresponding to $[\overline{u}_i]\in h^\ast(R(H_A)/(P_1,\ldots,P_{n-1}))$ under the above isomorphism. Using (\ref{defineutypea}) and the fact that $P_jP_j^\ast$ is real for all $j$, Lemma \ref{gradinglemma2} shows that
\[
v_i\in W^{-1}(X_A)\text{ for }1\leq i\leq r-1
\]
If $n$ is even, $P_{n/2}$ is real if $n\equiv 0$ (mod 4) and quaternionic if $n\equiv2$ (mod 4). We deduce from Lemma \ref{gradinglemma2} that
\[
v_r\in\begin{cases} W^{-1}(X_A)\text{ if }n\not \equiv2 \text{ (mod 4)}\\ W^{-3}(X_A)\text{ if } n\equiv 2\text{ (mod 4)}\end{cases}
\]
Summing up, we have proved:
\begin{theorem}\label{thmwittringtypea}
Let $m:=\lfloor n_1/2\rfloor+\ldots+\lfloor n_l/2\rfloor$ and $r:=\lfloor n/2\rfloor-m$. Then as a ring,
\begin{align*}
W^\ast(X_A(n_1,\ldots,n_l))\cong \frac{\bigotimes_{p=1}^l \mathbb{Z}_2\left[b_1^{(p)},\ldots,b_{\lfloor n_p/2\rfloor}^{(p)}\right]}{(\mu_1,\ldots,\mu_m)}\otimes \bigwedge(v_1,\ldots,v_r)
\end{align*}
where
\[
\mu_j=\sum_{a_1+\ldots+a_l=j}b_{a_1}^{(1)}\ldots b_{a_l}^{(l)}+\binom{n}{j}\text{ for }1\leq j\leq m,
\]
where we define $b_0^{(p)}:=1$ for $1\leq p\leq l$ and $b_i^{(p)}:=b_{n_p-i}^{(p)}$ for $1\leq p\leq l$ and $\lfloor n_p/2\rfloor <i\leq n_p$ (just as in (\ref{defineallbetas})). We have
\begin{align*}
b_i^{(p)}\in&~ W^0(X_A)\text{ for all }1\leq p\leq l\text{ and }0\leq i\leq n_p,\\
v_j\in&~ W^{-1}(X_A)\text{ for all }1\leq j\leq r-1,\\
v_r\in&
\begin{cases}
W^{-1}(X_A)&\text{ if }n\not\equiv2\text{ (mod 4)}\\
W^{-3}(X_A)&\text{ if }n\equiv2\text{ (mod 4)}
\end{cases}
\end{align*}
\end{theorem}
We want to tabulate the ranks of the Witt groups in different degrees. From the Appendix and Proposition \ref{propositiondimensionmu}, we immediately deduce:
\begin{theorem}
Let 
\begin{align*}
m:=\lfloor n_1/2\rfloor+\ldots+\lfloor n_l/2\rfloor,~~~a:=\frac{m!}{\lfloor n_1/2\rfloor!\cdot\ldots\cdot \lfloor n_l/2\rfloor!},~~~r:=\lfloor n/2\rfloor-m.
\end{align*}
We have $W^i(X_A(n_1,\ldots,n_l))\cong \mathbb{Z}_2^{a\cdot z_i}$ where $z_i$ is given as follows:

If $r=0$ (i.e. at most one $n_j$ is odd), then $z_0=1$ and $z_{-1}=z_{-2}=z_{-3}=0$.

If $r>0$ and $n\not\equiv 2$ (mod 4), we have:
\begin{center}
\resizebox{.9\textwidth}{!}{%
\begin{tabular}{c|c|c|c|c}
$r$ (4) & $z_0$ & $z_{-1}$ & $z_{-2}$ & $z_{-3}$\\\hline
$0$ & $2^{r-2}-2\cdot (-4)^{\frac{r-4}{4}}$ & $2^{r-2}$ & $2^{r-2}+2\cdot (-4)^{\frac{r-4}{4}}$ & $2^{r-2}$\\
$1$ & $2^{r-2}-2\cdot (-4)^{\frac{r-5}{4}}$ & $2^{r-2}-2\cdot (-4)^{\frac{r-5}{4}}$ & $2^{r-2}+2\cdot (-4)^{\frac{r-5}{4}}$ & $2^{r-2}+2\cdot (-4)^{\frac{r-5}{4}}$\\
$2$ & $2^{r-2}$ & $2^{r-2}+(-4)^{\frac{r-2}{4}}$ & $2^{r-2}$ & $2^{r-2}-(-4)^{\frac{r-2}{4}}$\\
$3$ & $2^{r-2}-(-4)^{\frac{r-3}{4}}$ & $2^{r-2}+(-4)^{\frac{r-3}{4}}$ & $2^{r-2}+(-4)^{\frac{r-3}{4}}$ & $2^{r-2}-(-4)^{\frac{r-3}{4}}$
\end{tabular}}\\
\end{center}
If $r>0$ and $n\equiv2$ (mod 4), we have:
\begin{center}
\resizebox{.9\textwidth}{!}{%
\begin{tabular}{c|c|c|c|c}
$r$ (4) & $z_0$ & $z_{-1}$ & $z_{-2}$ & $z_{-3}$\\\hline
$0$ & $2^{r-2}$ & $2^{r-2}+2\cdot (-4)^{\frac{r-4}{4}}$  & $2^{r-2}$& $2^{r-2}-2\cdot (-4)^{\frac{r-4}{4}}$\\
$1$ & $2^{r-2}-2\cdot (-4)^{\frac{r-5}{4}}$ & $2^{r-2}+2\cdot (-4)^{\frac{r-5}{4}}$ & $2^{r-2}+2\cdot (-4)^{\frac{r-5}{4}}$ & $2^{r-2}-2\cdot (-4)^{\frac{r-5}{4}}$\\
$2$  & $2^{r-2}+(-4)^{\frac{r-2}{4}}$ & $2^{r-2}$ & $2^{r-2}-(-4)^{\frac{r-2}{4}}$& $2^{r-2}$\\
$3$ & $2^{r-2}+(-4)^{\frac{r-3}{4}}$ & $2^{r-2}+(-4)^{\frac{r-3}{4}}$ & $2^{r-2}-(-4)^{\frac{r-3}{4}}$ & $2^{r-2}-(-4)^{\frac{r-3}{4}}$
\end{tabular}}
\end{center}
\end{theorem}
\subsection{Type $B_n$}
\begin{notationconventions}
For any $k\in\mathbb{N}$, we denote by $\epsilon\in Spin(k)$ the non-trivial preimage of $1\in SO(k)$ under the covering map $Spin(k)\to SO(k)$. Up to conjugation, every centraliser of a torus in $Spin(2n+1)$ is of the form $H_B=Spin(2m+1)\times \tilde{U}(n_1)\times \ldots \times \tilde{U}(n_l)/Z$ where $m+n_1+\ldots+n_l=n$ and $Z$ is the subgroup of all elements $(x_0,x_1,\ldots,x_l)\in Spin(2m+1)\times \tilde{U}(n_1)\times \ldots \times \tilde{U}(n_l)$ such that $x_i\in \left\{1,\epsilon\right\}$ for all $0\leq i\leq l$ and $x_j=\epsilon$ for an even number of $j$. We suppose that precisely $k$ of the integers $n_1,\ldots,n_l$ are odd and let $i_B\colon H_B\to Spin(2n+1)$ be the inclusion map. We denote by $X_B=Spin(2n+1)/H_B$ the associated complex flag variety. We write $R(Spin(2n+1))=\mathbb{Z}\left[\hat{y}_1,\ldots,\hat{y}_{n-1},\Delta\right]$ where $\hat{y}_1$ is induced by the standard representation of $SO(2n+1)$ of rank $2n+1$ and $\hat{y}_i=\Lambda^i(\hat{y}_1)$ and $\Delta$ is the spin representation. For $R(H_B)$ and $h^+(R(H_B))$, we use the notation introduced in Propositions \ref{propositionrepnringtypeb} and \ref{propositiontatecohreptypeb}. In addition, it will be convenient to define the following elements in $h^+(R(H_B))$:
\begin{align}\label{definitionalphabetatypeb}
\begin{split}
    \beta_0:=1\text{ and }\beta_{m+j}:=\beta_{m+1-j}\text{ for all }1\leq j\leq m+1\\
    \alpha_0^{(p)}:=1\text{ and }\alpha_i^{(p)}:=\alpha_{n_p-i}^{(p)}\text{ for all }1\leq p\leq l\text{ and }\lfloor n_p/2\rfloor <i\leq n_p
\end{split}
\end{align}
\end{notationconventions}
We want to compute the Tate cohomology of $K^0(X_B)$ following section \ref{chapteroutlinecomputation}. As a first step, we need to compute the restriction map $i_B^\ast\colon R(Spin(2n+1))\to R(H_B)$. We see directly that
\[
i_B^\ast(\hat{y}_1)=y_1+\sum_{p=1}^l x_1^{(p)}+\left(\sum_{p=1}^l x_1^{(p)}\right)^\ast
\]
Consequently, in $h^+(R(H_B))$ we have
\begin{align*}
    [i_B^\ast(\hat{y}_f-\text{rk}(\hat{y}_f))]=&\left[\Lambda^f\left(i^\ast(\hat{y}_1)\right)-\text{rk}(\hat{y}_f)\right]\\
    = &\sum_{j=0}^{\lfloor f/2\rfloor} \beta_{f-2j}\cdot \sum_{a_1+\ldots+a_l=j}\alpha_{a_1}^{(1)}\ldots \alpha_{a_l}^{(l)}+\binom{2n+1}{f}=:\xi_f,
\end{align*}
keeping in mind the definitions we made in (\ref{definitionalphabetatypeb}).
We still need to compute the restriction of the spin representation $\Delta$: We have that in $h^\ast(R(H))$,
\begin{align}\label{equationsection5.2restndelta}
    [i^\ast(\Delta)]=
    \begin{cases}
    0&\text{ if }k>0\\ \delta &\text{ if }k=0
    \end{cases}
\end{align}
In order to apply the results of section \ref{chapteroutlinecomputation} to compute $h^\ast(K^0(X_B))$, we need to investigate the relations between $\xi_1,\ldots,\xi_{n-1},[i^\ast(\Delta)]\in h^\ast(R(H_B))$.
\begin{proposition}\label{propositionrelationsregularityxi}
Let
\[
S:=\left\{i\in\mathbb{Z}\mid 1\leq i\leq m\right\}\cup \left\{i\in\mathbb{Z}\mid m< i\leq 2\left(\left\lfloor \frac{m}{2}\right\rfloor+\left\lfloor \frac{n_1}{2}\right\rfloor+\ldots+\left\lfloor \frac{n_l}{2}\right\rfloor\right)\text{ even}\right\}
\]
and consider $R:=\bigotimes_{p=1}^l \mathbb{Z}_2\left[\beta_1^{(p)},\ldots,\beta_{\lfloor n_p/2\rfloor}^{(p)}\right]$ as a subring of $h^\ast(R(H_B))$.
\begin{description}
\item[Suppose $k>0$.] Then the elements $\xi_j$ for $j\in S$ form an $h^\ast(R(H_B))$-regular sequence in some order. If $i\in \left\{1,2,\ldots,n-1\right\}\setminus S$, then $\xi_i$ is an $R$-linear combination of the $\xi_j$ for $j\in S$. We have $[i^\ast(\Delta)]=0$.
\item[Suppose $k=0$ and $m$ is even.] Then the elements $\xi_j$ for $j\in S\setminus \left\{n\right\}$ together with $\left[i^\ast(\Delta)\right]=\delta$ form an $h^\ast(R(H_B))$-regular sequence in some order. If $i\in\left\{1,\ldots,n-1\right\}\setminus S$, then $\xi_i$ is an $R$-linear combination of the $\xi_j$ where $j\in S\setminus \left\{n\right\}$.
\item[Suppose $k=0$ and $m$ is odd.] Then the elements $\xi_j$ for $j\in S$ form an $h^\ast(R(H_B))$-regular sequence in some order. Let $I$ be the ideal generated by these elements. Then $\overline{\delta}\in h^+(R(H_B))/I$ is a zero divisor with annihilator the ideal generated by $\overline{\delta}$ itself. The elements $\xi_i$ for $i\in \left\{1,2,\ldots,n-1\right\}\setminus S$ are $R$-linear combinations of the $\xi_j$ for $j\in S$.
\end{description}
\end{proposition}
\begin{proof}
The claims for $k>0$ are immediate from Propositions \ref{propositionregularityxi} and \ref{propositionrelationsxi}, Remark \ref{remarkxi} and equation (\ref{equationsection5.2restndelta}).

Suppose $k=0$ and $m$ is even. Proposition \ref{propositionregularityxi} shows that
\begin{align}\label{eqnregularsequencexi1}
\xi_1,\xi_2,\ldots,\xi_{m-1},\xi_m,\xi_{m+2},\xi_{m+4},\ldots,\xi_{n-2},\xi_n
\end{align}
is a regular sequence in the subring
\[
A:=\mathbb{Z}_2[\beta_1,\ldots,\beta_m]\otimes \bigotimes_{p=1}^l \mathbb{Z}_2\left[\alpha_1^{(p)},\ldots,\alpha_{\lfloor n_p/2\rfloor}^{(p)}\right]\subset h^\ast(R(H_B))
\]
In $h^\ast(R(H_B))$, we have
\[
\delta^2=\sum_{i=0}^m \beta_i\cdot \alpha_{\frac{n_1}{2}}^{(1)}\ldots\alpha_{\frac{n_l}{2}}^{(l)}
\]
Remark \ref{remarkregularityxi} shows that in the sequence (\ref{eqnregularsequencexi1}), we may replace $\xi_n$ by $\delta^2$ and still have an $A$-regular sequence. But since $h^\ast(R(H_B))$ is a free module of rank 2 over $A$, the sequence is also $h^\ast(R(H_B))$-regular. But then replacing $\delta^2$ by $\delta$ clearly still gives an $h^\ast(R(H_B))$-regular sequence, as required.

The statement about the linear relations is immediate from Proposition \ref{propositionrelationsxi} and Remark \ref{remarkxi}.

Lastly, suppose $k=0$ and $m$ is odd. From Proposition \ref{propositionregularityxi}, we deduce that the $\xi_j$ for $j\in S$ form an $h^\ast(R(H_B))$-regular sequence. Proposition \ref{propositionrelationsxi} and Remark \ref{remarkxi} imply that the elements $\xi_i$ for $i\in\left\{1,2,\ldots,n-1\right\}\setminus S$ are $R$-linear combinations of the $\xi_j$ for $j\in S$. Now from Lemma \ref{lemmaevenxioddxi}, we have that $\beta_i+\beta_{i-1}\in I$ for every odd $1\leq i\leq m$. This shows that
\[
\delta^2=(1+\beta_1+\ldots+\beta_{m-1}+\beta_m)\cdot \alpha_{n_1/2}^{(1)}\ldots \alpha_{n_l/2}^{(l)}\in I
\]
Let us identify the annihilator of $\overline{\delta}$ in $h^+(R(H_B))/I$. Let $A$ be as defined above in the previous case, then $h^\ast(R(H_B))$ is still a free $A$-module of rank 2 with basis $1,\delta$. Since $I$ is generated by elements in $A$, we deduce that $h^+(R(H_B))/I$ is a free $A/A\cap I$-module with basis $1$, $\overline{\delta}$. From this and the fact that $\overline{\delta}^2=0$, it follows immediately that the annihilator of $\overline{\delta}$ in $h^+(R(H_B))/I$ is as claimed.
\end{proof}
Now let $S$ be as in the previous Proposition and 
\[
\overline{S}:=\left\{1,\ldots,n-1\right\}\setminus S \text{ and } B:=\bigotimes_{p=1}^l\mathbb{Z}\left[x_1^{(p)}\left(x_1^{(p)}\right)^\ast,\ldots,x_{\lfloor n_p/2\rfloor}^{(p)}\left(x_{\lfloor n_p/2\rfloor}^{(p)}\right)^\ast\right],
\]
where we regard $B$ as a subring of $R(H_B)$. The previous Proposition implies the following:

In all cases and for all $t\in\overline{S}$, we can find $w_{t}\in R(H_B)$ such that
\begin{align}\label{eqnrestrictionofrepnstypeb1}
w_{t}+w_{t}^\ast-i_B^\ast(\hat{y}_{t}-\text{rk}(\hat{y}_{t}))\in \sum_{s\in S}B\cdot i^\ast(\hat{y}_s-\text{rk}(\hat{y}_s))
\end{align}
Moreover, if $k>0$, we see from (\ref{equationsection5.2restndelta}) that we can find $\eta\in R(H_B)$ such that
\begin{align}\label{eqnrestrictionofrepnstypeb2}
\eta+\eta^\ast=i_B^\ast(\Delta)
\end{align}
If $k=0$ and $m$ is odd, we can find $\kappa\in R(H_B)$ such that
\begin{align}\label{eqnrestrictionofrepnstypeb3}
\kappa+\kappa^\ast-i^\ast(\Delta)^2\in \sum_{s\in S}B\cdot i^\ast(\hat{y}_s-\text{rk}(\hat{y}_s))
\end{align}
We use Proposition \ref{propositionrelationsregularityxi} to show:
\begin{proposition}
If not all of $n_1,\ldots,n_l$ are even, then
\[
h^\ast(K^0(X_B))\cong \frac{\mathbb{Z}_2[\beta_1,\ldots,\beta_m]\otimes\bigotimes_{p=1}^l\mathbb{Z}_2\left[\alpha_1^{(p)},\ldots,\alpha_{\left\lfloor\frac{n_p}{2}\right\rfloor}^{(p)}\right]}{\left(\xi_j\mid j\in S\right)}\otimes \bigwedge_{t\in\overline{S}}(\left[\overline{w}_{t}\right])\otimes\bigwedge(\left[\overline{\eta}\right])
\]
If $m,n_1,\ldots,n_l$ are even, then
\[
h^\ast(K^0(X_B))\cong \frac{\mathbb{Z}_2[\beta_1,\ldots,\beta_m]\otimes\bigotimes_{p=1}^l\mathbb{Z}_2\left[\alpha_1^{(p)},\ldots,\alpha_{\frac{n_p}{2}}^{(p)}\right]}{\left(\xi_j\mid j\in S\setminus\left\{n\right\}\right)+\left(\beta_m\alpha_{n_1/2}^{(1)}\ldots\alpha_{n_l/2}^{(l)}\right)}\otimes \bigwedge_{t\in\overline{S}}(\left[\overline{w}_{t}\right])
\]
If $n_1,\ldots,n_l$ are even and $m$ is odd, then
\[
h^\ast(K^0(X_B))\cong \frac{\mathbb{Z}_2[\beta_1,\ldots,\beta_m]\otimes\bigotimes_{p=1}^l\mathbb{Z}_2\left[\alpha_1^{(p)},\ldots,\alpha_{\frac{n_p}{2}}^{(p)}\right]}{\left(\xi_j\mid j\in S\right)}\otimes \bigwedge_{t\in\overline{S}}(\left[\overline{w}_{t}\right])\otimes\bigwedge(\left[\overline{\kappa}\right])
\]
In the above,
\[
\xi_j=\sum_{i=0}^{\lfloor j/2\rfloor} \beta_{j-2i}\cdot \sum_{a_1+\ldots+a_l=i}\alpha_{a_1}^{(1)}\ldots \alpha_{a_l}^{(l)}+\binom{2n+1}{j}
\]
recalling the definitions we made in (\ref{definitionalphabetatypeb}). In all the above expressions, the left factors of the tensor products are contained in $h^+$ and the generators of the exterior algebras all lie in $h^-$.
\end{proposition}
\begin{proof}
The first case follows immediately from Propositions \ref{propositionrelationsregularityxi} and \ref{propositiontatecohofktheory}. So does the second case, using in addition that $\beta_{2i+1}+\beta_{2i}$ is contained in the ideal $( \xi_j\mid j\in S,~j\leq 2i+1)$ for all $i$ by Lemma \ref{lemmaevenxioddxi}. The third case follows similarly from Proposition \ref{propositionrelationsregularityxi}, using Lemma \ref{tatecohlemma6} in addition.
\end{proof}
Recall from Bousfield's lemma \ref{lemmabousfield} that there is an isomorphism
\[
W^\ast(X_B)\xrightarrow[\cong]{\overline{c}}h^\ast(K^0(X_B))\xleftarrow[\cong]{[\overline{\alpha}]}h^\ast(R(H_B)/\mathfrak{a}(G))
\]
We have computed the Tate cohomology on the right and want to determine the Witt grading. So under the above isomorphism, we let
\begin{itemize}
    \item $a_i\in W^\ast(X_B)$ correspond to $\overline{\alpha}_i$ for $1\leq i\leq m$ in all cases
    \item $b_j^{(p)}\in W^\ast(X_B)$ correspond to $\overline{\beta}_j^{(p)}$ for all $1\leq p\leq l$ and $0\leq j\leq n_p$ in all cases
    \item $u_{t} \in W^\ast(X_B)$ correspond to $\left[\overline{w}_{t}\right]$ for all $t\in \overline{S}$ in all cases
    \item $c_1\in W^\ast(X_B)$ correspond to $[\overline{\eta}]$ if not all of $n_1,\ldots,n_l$ are even
    \item $c_2\in W^\ast(X_B)$ correspond to $[\overline{\kappa}]$ if $n_1,\ldots,n_l$ are even and $m$ is odd
\end{itemize}
Recall that the subring $B$ of $R(H_B)$ consists entirely of real representations and $\lambda_j\in R(Spin(2n+1))$ is real for all $j$.

We deduce from Lemma \ref{gradinglemma1} that $a_i\in W^0(X_B)$ for all $1\leq i\leq m$ and $b_j^{(p)}\in W^0(X)$ for all $1\leq p\leq l$ and $1\leq j\leq \lfloor n_p/2\rfloor$.

From equation (\ref{eqnrestrictionofrepnstypeb1}) and Lemma \ref{gradinglemma2}, we deduce that $u_t\in W^{-1}(X_B)$ for all $t\in \overline{S}$.

Suppose not all of $n_1,\ldots, n_l$ are even. From equation (\ref{eqnrestrictionofrepnstypeb2}) and Lemma \ref{gradinglemma2}, recalling that $\Delta\in R(Spin(2n+1))$ is of real type if $n\equiv 0,3$ (mod 4) and of quaternionic type if $n\equiv 1,2$ (mod 4), we deduce that
\[
c_1\in
\begin{cases}
W^{-1}(X_B)&\text{ if }n\equiv 0,3\text{ (mod 4)}\\
W^{-3}(X_B)&\text{ if }n\equiv 1,2\text{ (mod 4)}
\end{cases}
\]
Suppose $n_1,\ldots,n_l$ are even and $m$ is odd. Then Lemma \ref{gradinglemma2} and equation (\ref{eqnrestrictionofrepnstypeb3}) show that $c_2\in W^{-1}(X_B)$.

In summary, we have proved:
\begin{theorem}\label{thmwittringtypeb}
Let
\begin{align*}
S:=&\left\{i\in\mathbb{Z}\mid 1\leq i\leq m\right\}\cup \left\{i\in\mathbb{Z}\mid m< i\leq 2\left(\left\lfloor \frac{m}{2}\right\rfloor+\left\lfloor \frac{n_1}{2}\right\rfloor+\ldots+\left\lfloor \frac{n_l}{2}\right\rfloor\right)\text{ even}\right\}\\
\overline{S}:=&\left\{1,\ldots,n-1\right\}\setminus S
\end{align*}
If not all of $m,n_1,\ldots,n_l$ are even, then
\[
W^\ast(X_B)\cong \frac{\mathbb{Z}_2[b_1,\ldots,b_m]\otimes\bigotimes_{p=1}^l\mathbb{Z}_2\left[a_1^{(p)},\ldots,a_{\left\lfloor\frac{n_p}{2}\right\rfloor}^{(p)}\right]}{\left(\xi_j\mid j\in S\right)}\otimes \bigwedge_{t\in\overline{S}}(u_{t})\otimes\bigwedge(c)
\]
If $m,n_1,\ldots,n_l$ are even, then
\[
W^\ast(X_B)\cong \frac{\mathbb{Z}_2[b_1,\ldots,b_m]\otimes\bigotimes_{p=1}^l\mathbb{Z}_2\left[a_1^{(p)},\ldots,a_{\frac{n_p}{2}}^{(p)}\right]}{\left(\xi_j\mid j\in S\setminus\left\{n\right\}\right)+\left(b_m a_{\frac{n_1}{2}}^{(1)}\ldots a_{\frac{n_l}{2}}^{(l)}\right)}\otimes \bigwedge_{t\in\overline{S}}(u_{t})
\]
In the above,
\[
\xi_j=\sum_{i=0}^{\lfloor j/2\rfloor} b_{j-2i}\cdot \sum_{q_1+\ldots+q_l=i}a_{q_1}^{(1)}\ldots  a_{q_l}^{(l)}+\binom{2n+1}{j}
\]
where we make the following definitions:
\begin{align*}
\begin{split}
    b_0:=1\text{ and }b_{m+j}:=b_{m+1-j}\text{ for all }1\leq j\leq m+1\\
    a_0^{(p)}:=1\text{ and }a_i^{(p)}:=a_{n_p-i}^{(p)}\text{ for all }1\leq p\leq l\text{ and }\lfloor n_p/2\rfloor <i\leq n_p
\end{split}
\end{align*}
Furthermore, we have that all $b_i,a_j^{(p)}\in W^0(X_B)$ and all $u_{t}\in W^{-1}(X_B)$ and
\[
c\in
\begin{cases}
W^{-3}(X_B)&\text{ if not all of }n_1,\ldots,n_l \text{ are even and }n\equiv 1,2\text{ (mod 4)}\\
W^{-1}(X_B)&\text{ else}
\end{cases}
\]
\end{theorem}
It is easily checked that in all of the above cases, the number of exterior algebra generators in the above expressions is $\sum_{i=1}^l n_i-\lfloor n_i/2\rfloor$.
From Proposition \ref{propositiondimensionquotientxi} and the Appendix, we can tabulate the ranks of the Witt groups in all degrees:
\begin{theorem} Let
\begin{align*}
a:=\frac{\left(\lfloor \frac{m}{2}\rfloor+\lfloor \frac{n_1}{2}\rfloor+\ldots+\lfloor \frac{n_{k+l}}{2}\rfloor\right)!}{\lfloor \frac{m}{2}\rfloor!\cdot \lfloor \frac{n_1}{2}\rfloor!\cdot\ldots\cdot \lfloor \frac{n_{k+l}}{2}\rfloor!}~\text{ and }~ r:=\sum_{i=1}^l n_i-\lfloor n_i/2\rfloor
\end{align*}
We have $W^i(X_B)\cong \mathbb{Z}_2^{a\cdot z_i}$ where $z_i$ is given as follows:\\
If $r=0$, then $z_0=1$ and $z_{-1}=z_{-2}=z_{-3}=0$.\\
If not all of $n_1,\ldots,n_l$ are even and $n\equiv 1\text{ or }2$ (mod 4), then we have:
\begin{center}
\resizebox{.9\textwidth}{!}{%
\begin{tabular}{c|c|c|c|c}
$r$ (4) & $z_0$ & $z_{-1}$ & $z_{-2}$ & $z_{-3}$\\\hline
$0$ & $2^{r-2}$ & $2^{r-2}+2\cdot (-4)^{\frac{r-4}{4}}$  & $2^{r-2}$& $2^{r-2}-2\cdot (-4)^{\frac{r-4}{4}}$\\
$1$ & $2^{r-2}-2\cdot (-4)^{\frac{r-5}{4}}$ & $2^{r-2}+2\cdot (-4)^{\frac{r-5}{4}}$ & $2^{r-2}+2\cdot (-4)^{\frac{r-5}{4}}$ & $2^{r-2}-2\cdot (-4)^{\frac{r-5}{4}}$\\
$2$  & $2^{r-2}+(-4)^{\frac{r-2}{4}}$ & $2^{r-2}$ & $2^{r-2}-(-4)^{\frac{r-2}{4}}$& $2^{r-2}$\\
$3$ & $2^{r-2}+(-4)^{\frac{r-3}{4}}$ & $2^{r-2}+(-4)^{\frac{r-3}{4}}$ & $2^{r-2}-(-4)^{\frac{r-3}{4}}$ & $2^{r-2}-(-4)^{\frac{r-3}{4}}$
\end{tabular}}
\end{center}
Otherwise, we have:
\begin{center}
\resizebox{.9\textwidth}{!}{%
\begin{tabular}{c|c|c|c|c}
$r$ (4) & $z_0$ & $z_{-1}$ & $z_{-2}$ & $z_{-3}$\\\hline
$0$ & $2^{r-2}-2\cdot (-4)^{\frac{r-4}{4}}$ & $2^{r-2}$ & $2^{r-2}+2\cdot (-4)^{\frac{r-4}{4}}$ & $2^{r-2}$\\
$1$ & $2^{r-2}-2\cdot (-4)^{\frac{r-5}{4}}$ & $2^{r-2}-2\cdot (-4)^{\frac{r-5}{4}}$ & $2^{r-2}+2\cdot (-4)^{\frac{r-5}{4}}$ & $2^{r-2}+2\cdot (-4)^{\frac{r-5}{4}}$\\
$2$ & $2^{r-2}$ & $2^{r-2}+(-4)^{\frac{r-2}{4}}$ & $2^{r-2}$ & $2^{r-2}-(-4)^{\frac{r-2}{4}}$\\
$3$ & $2^{r-2}-(-4)^{\frac{r-3}{4}}$ & $2^{r-2}+(-4)^{\frac{r-3}{4}}$ & $2^{r-2}+(-4)^{\frac{r-3}{4}}$ & $2^{r-2}-(-4)^{\frac{r-3}{4}}$
\end{tabular}}
\end{center}
\end{theorem}
\subsection{Type $C_n$}
\begin{notationconventions}
Up to conjugation, every centraliser of a torus in $Sp(n)$ is of the form $H_C=Sp(m)\times U(n_1)\times \ldots\times U(n_l)$ where $m+n_1+\ldots+n_l=n$. Let $i_C\colon H_C\to Sp(n)$ be the inclusion and $X_C=Sp(n)/H_C$ be the associated complex flag variety. We write $R(Sp(n))=\mathbb{Z}\left[\hat{z}_1,\ldots,\hat{z}_n\right]$ where $\hat{z}_1$ is the standard representation of rank $2n$ and $\hat{z}_j=\Lambda^j(\hat{z}_1)$. For $R(H_C)$ and $h^+(R(H_C))$, we use the notation introduced in Propositions \ref{propositionrepnringtypec} and \ref{propositiontatecohreptypec}. In addition, it will be convenient to define the following elements in $h^+(R(H_C))$:
\begin{align}\label{definitionsbetatypec}
\begin{split}
\gamma_0:=1\text{ and }\gamma_{m+i}:=\gamma_{m-i}\text{ for all }1\leq i\leq m,\\
\alpha_0^{(p)}:=1\text{ for all }1\leq p\leq l\text{ and }\alpha_i^{(p)}:=\alpha_{n_p-i}^{(p)}\text{ for all }1\leq p\leq l,~\lfloor n_p/2\rfloor <i\leq n_p
\end{split}
\end{align}
Note that with these definitions, we have $\gamma_j=\left[\Lambda^j(z_1)\right]$ for all $0\leq j\leq 2m$ and $\alpha_i^{(p)}=\left[x_i^{(p)}\left(x_i^{(p)}\right)^\ast\right]$ for all $1\leq p\leq l$ and $0\leq i\leq n_p$.
\end{notationconventions}
We need to determine the induced map $i_C^\ast$ on representation rings. We see directly that
\[
i_C^\ast(\hat{z}_1)=z_1+\sum_{p=1}^l x_1^{(p)}+\sum_{p=1}^l \left(x_1^{(p)}\right)^\ast
\]
Hence we obtain for $1\leq j\leq n$ that
\begin{align*}
i^\ast(\hat{z}_j-\text{rk}(\hat{z}_j))=&\Lambda^j\left(z_1+\sum_{p=1}^l x_1^{(p)}+\sum_{p=1}^l \left(x_1^{(p)}\right)^\ast\right)-\binom{2n}{j}\\
=&\sum_{c+d_1+d_1'+\ldots+d_l+d_l'=j} z_c\cdot x_{d_1}^{(1)}\left(x_{d_1'}^{(1)}\right)^\ast\ldots x_{d_l}^{(l)}\left(x_{d_l'}^{(l)}\right)^\ast-\binom{2n}{j}=:P_j
\end{align*}
Using that $[a+a^\ast]=0$ in $h^+(R(H_C))$ for all $a\in R(H)$, we deduce that in $h^+(R(H_C))$,
\begin{align*}
    [P_j]=&\sum_{c+2d_1+\ldots +2d_l=j}\gamma_c\alpha_{d_1}^{(1)}\ldots \alpha_{d_l}^{(l)}+\binom{2n}{j}\\
    =& \sum_{d=0}^{\lfloor j/2\rfloor}\gamma_{j-2d}\sum_{d_1+\ldots+d_l=d}\alpha_{d_1}^{(1)}\ldots \alpha_{d_l}^{(l)}+\binom{2n}{j}=:\nu_j
\end{align*}
where we keep in mind the definitions we made in (\ref{definitionsbetatypec}).

In order to apply Proposition \ref{propositiontatecohofktheory} to compute $h^\ast(K^0(X_C))$, we need to determine the relations between the $\nu_j\in h^+(R(H_C))$, i.e. show that (A1) and (A2) from section \ref{chapteroutlinecomputation} are satisfied. We state the result now and postpone the proof to Propositions \ref{propositionregularitynu} and \ref{propositionrelationsnu}.
\begin{proposition}\label{relationsbetweennu}
Let
\[
S:= \left\{k\in\mathbb{N}\mid 1\leq k\leq m\right\}\cup\left\{k\in\mathbb{N}\mid m<k\leq 2(\lfloor m/2\rfloor+\lfloor n_1/2\rfloor+\ldots+\lfloor n_l/2\rfloor)\text{ even}\right\}
\]
Then the $\nu_s$ for $s\in S$ form an $h^+(R(H))$-regular sequence in some order.\\
If $k$ is even, then $\nu_k$ is a $\mathbb{Z}_2$-linear combination of the $\nu_s$ for even $s\in S$.\\
If $k$ is odd, then $\nu_k$ is a $\bigotimes_{p=1}^l\mathbb{Z}_2\left[\beta_1^{(p)},\ldots,\beta_{\lfloor n_p/2\rfloor}^{(p)}\right]$-linear combination of the $\nu_s$ for odd $s\in S$.\\
In particular, all $\nu_k$ are contained in the ideal $(\nu_s\mid s\in S)$.
\end{proposition}
Let 
\[
\overline{S}:=\left\{1,\ldots,n\right\}\setminus S \text{ and } A:=\bigotimes_{p=1}^l\mathbb{Z}\left[x_1^{(p)}\left(x_1^{(p)}\right)^\ast,\ldots,x_{\lfloor n_p/2\rfloor}^{(p)}\left(x_{\lfloor n_p/2\rfloor}^{(p)}\right)^\ast\right],
\]
where $A$ is regarded as a subring of $R(H_C)$. Then Proposition \ref{relationsbetweennu} shows that for all $t\in\overline{S}$, we can find $w_{t}\in R(H)$ such that
\begin{align}\label{polynomialsPvanishtypec}
\begin{split}
    w_{t}+w_{t}^\ast-P_{t}&\in \sum_{i\in S \text{ even}} \mathbb{Z}\cdot P_i\text{ if }t\in \overline{S}\text{ is even},\\
    w_{t}+w_{t}^\ast-P_{t}&\in \sum_{i\in S \text{ odd}} A\cdot P_i\text{ if }t\in \overline{S}\text{ is odd}
\end{split}
\end{align}
Now from Proposition \ref{propositiontatecohofktheory}, we deduce that
\[
h^\ast(R(H_C)/(P_1,...,P_n))\cong \frac{\mathbb{Z}_2[\alpha_1,\ldots,\alpha_m]\otimes \bigotimes_{p=1}^l\mathbb{Z}_2\left[\beta_1^{(p)},\ldots,\beta_{\lfloor n_p/2\rfloor}^{(p)}\right]}{(\nu_s\mid s\in S)}\otimes \bigwedge_{t\in\overline{S}}(\left[\overline{w}_{t}\right])
\]
We will show in Lemma \ref{lemmanuregular1} that
\[
(\nu_s\mid s\in S\text{ odd})=(\gamma_i\mid 1\leq i\leq m\text{ odd})
\]
Hence setting $\zeta_i:=\gamma_{2i}$ for $0\leq i\leq m$, we can simplify the above expression and deduce:
\begin{proposition}
Let $h:=\lfloor m/2\rfloor +\lfloor n_1/2\rfloor +\ldots+\lfloor n_l/2\rfloor$. There is a ring isomorphism
\[
h^\ast(X_C)\cong \frac{\mathbb{Z}_2\left[\zeta_1,\ldots,\zeta_{\lfloor m/2\rfloor}\right]\otimes \bigotimes_{p=1}^l\mathbb{Z}_2\left[\alpha_1^{(p)},\ldots,\alpha_{\lfloor n_p/2\rfloor}^{(p)}\right]}{\left(\mu_1,\ldots,\mu_h\right)}\otimes \bigwedge_{t\in\overline{S}}(\left[\overline{w}_{t}\right])
\]
with
\[
\mu_j=\sum_{c+c_1+\ldots c_l=j}\zeta_c\alpha_{c_1}^{(1)}\ldots \alpha_{c_l}^{(l)}+\binom{2n}{2j}~\text{ for }1\leq j\leq h,
\]
where we recall the definitions we made in (\ref{definitionsbetatypec}) and that $\zeta_0=1$, $\zeta_{i}=\zeta_{m-i}$ for $\lfloor m/2\rfloor<i\leq m$.\\The left factor in the tensor product is contained in $h^+$ and $[\overline{w}_{t}]\in h^-$ for all $t\in \overline{S}$.
\end{proposition}
Recall from Bousfield's lemma \ref{lemmabousfield} that there is an isomorphism
\[
W^\ast(X_C)\xrightarrow[\cong]{\overline{c}}h^\ast(K^0(X_C))\xleftarrow[\cong]{[\overline{\alpha}]}h^\ast(R(H_C)/(P_1,\ldots,P_{n}))
\]
We have computed the Tate cohomology on the right and want to determine the Witt grading.

Let $k$ be the number of odd integers among $m,n_1,\ldots, n_l$. We set
\[
f:=\lfloor k/2\rfloor=\Bigl\lfloor\frac{n}{2}\Bigr\rfloor-\Bigl\lfloor\frac{m}{2}\Bigr\rfloor -\Bigl\lfloor\frac{n_1}{2}\Bigr\rfloor-\ldots-\Bigl\lfloor\frac{n_l}{2}\Bigr\rfloor\text{ and }g:=\Bigl\lceil \frac{n-m}{2}\Bigr\rceil
\]
Then $f$ is the number of even integers in $\overline{S}$ and $g$ is the number of odd integers in $\overline{S}$.

Under the above isomorphism, let
\begin{itemize}
    \item $b_i^{(p)}\in W^\ast(X_C)$ correspond to $\overline{\alpha}_i^{(p)}\in h^\ast(R(H_C)/(P_1,\ldots, P_n))$ for all $i$, $p$.
    \item $a_i\in W^\ast(X_C)$ correspond to $\overline{\zeta}_i\in h^\ast(R(H_C)/(P_1,\ldots, P_n))$ for all $i$.
    \item $u_i\in W^\ast(X_C)$ for $1\leq i\leq f$ correspond to the $\left[\overline{w}_{t}\right] \in h^\ast(R(H_C)/(P_1,\ldots,P_n))$ for all even $t\in\overline{S}$.
    \item $v_j\in W^\ast(X_C)$ for $1\leq j\leq g$ correspond to the $\left[\overline{w}_{t}\right]\in h^\ast(R(H_C)/(P_1,\ldots,P_n))$ for all odd $t\in\overline{S}$.
\end{itemize}
We see that $\overline{\alpha}_i^{(p)}\in h^\ast(R(H_C)/(P_1,\ldots,P_n))$ is represented by a real representation (as $\rho\rho^\ast$ is real for any complex representation $\rho$). Thus by Lemma \ref{gradinglemma1}, $b_i^{(p)}\in W^0(X_C)$ for all $1\leq p\leq l$ and $1\leq i \leq \lfloor n_p/2\rfloor$.

As $\zeta_i=\gamma_{2i}$ is represented by a real representation, Lemma \ref{gradinglemma1} implies that $a_i\in W^0(X_C)$ for all $i$.

We know that $P_i\in R(H_C)$ is quaternionic for odd $i$ and real for even $i$. Thus we deduce from (\ref{polynomialsPvanishtypec}) and Lemma \ref{gradinglemma2} that $u_i\in W^{-1}(X_C)$ for $1\leq i\leq f$ and that $v_j\in W^{-3}(X_C)$ for $1\leq j\leq g$.

All in all, we have shown:
\begin{theorem}\label{theoremwittringtypecn}
Let
\[
h:=\Bigl\lfloor\frac{m}{2}\Bigr\rfloor +\Bigl\lfloor\frac{n_1}{2}\Bigr\rfloor+\ldots+\Bigl\lfloor\frac{n_l}{2}\Bigr\rfloor \text{ and }f:=\Bigl\lfloor\frac{n}{2}\Bigr\rfloor-h \text{ and }g:=\Bigl\lceil \frac{n-m}{2}\Bigr\rceil
\]
There is a ring isomorphism
\[
W^\ast(X_C)\cong\frac{\mathbb{Z}_2\left[a_1,\ldots,a_{\lfloor m/2\rfloor}\right]\otimes \bigotimes_{p=1}^l\mathbb{Z}_2\left[b_1^{(p)},\ldots,b_{\lfloor n_p/2\rfloor}^{(p)}\right]}{(\mu_1,\ldots,\mu_h)}\otimes \bigwedge(u_1,\ldots,u_f,v_1,\ldots,v_g)
\]
where
\[
\mu_j=\sum_{c+c_1+\ldots+c_l=j}a_c\cdot b_{c_1}^{(1)}\ldots b_{c_l}^{(l)}+\binom{2n}{2j},
\]
recalling that we set
\begin{align*}
a_0:=1 \text{ and } a_i:=a_{m-i}\text{ for all }\lfloor m/2\rfloor <i\leq m\\
b_0^{(p)}:=1\text{ for all }1\leq p\leq l\text{ and }b_i^{(p)}:=b_{n_p-i}^{(p)}\text{ for all }1\leq p\leq l,~\lfloor n_p/2\rfloor <i\leq n_p
\end{align*}
The left factor in the above tensor product is contained in $W^0(X_C)$. Furthermore, $u_i\in W^{-1}(X_C)$ for all $1\leq i\leq f$ and $v_j\in W^{-3}(X_C)$ for all $1\leq j\leq g$.
\end{theorem}
We now tabulate the ranks of the Witt groups in the different degrees. This immediately follows from Proposition \ref{propositiondimensionquotientnu} and the Appendix.
\begin{theorem}
Let $h$, $f$ and $g$ be as in the previous theorem and
\[
a:=\frac{h!}{\lfloor m/2\rfloor!\lfloor n_1/2\rfloor!\ldots\lfloor n_l/2\rfloor!}
\]
We have $W^i(X_C(m,n_1,\ldots,n_l))\cong \mathbb{Z}_2^{a\cdot z_i}$ where $z_i$ is given as follows:

If $f=g=0$, then $z_0=1$ and $z_{-1}=z_{-2}=z_{-3}=0$.

If $(f,g)\neq (0,0)$, then the $z_i$ are given as in the Appendix.
\end{theorem}
\subsection{Type $D_n$}
The computations for type $D_n$ are largely analogous to the other types. They are more elaborate because more cases need to be distinguished, so we only state the results and refer to \cite{Hemmert} for full details. Note that in one particular case, we only determine the Witt groups without the ring structure.
\begin{notationconventions}
We let $X_D=SO(2n)/SO(2m)\times U(n_1)\times \ldots \times U(n_l)$. Every complex flag variety which is a quotient of $SO(2n)$ is of this form.
\end{notationconventions}
\begin{theorem}\label{thmwittringtypedn}
Let $N:=\left\lfloor \frac{m}{2}\right\rfloor+\left\lfloor \frac{n_1}{2}\right\rfloor+\ldots+\left\lfloor \frac{n_l}{2}\right\rfloor$ and $T:=\left\{j\mid m\leq j\leq n-1\right\}\setminus \left\{i\mid i\leq 2\left(\left\lfloor\frac{m}{2}\right\rfloor+\left\lfloor\frac{n_1}{2}\right\rfloor+\ldots+\left\lfloor\frac{n_l}{2}\right\rfloor\right)\right\}$.

If $n$ and $m>2$ are odd or if $n$ is odd and $m=0$, we have a ring isomorphism
\[
W^\ast(X_D)\cong \frac{\mathbb{Z}_2\left[b_1,\ldots,b_{\left\lfloor\frac{m}{2}\right\rfloor}\right]\otimes\bigotimes_{p=1}^l \mathbb{Z}_2\left[\beta_1^{(p)},\ldots,\beta_{\left\lfloor\frac{n_p}{2}\right\rfloor}^{(p)}\right]}{(\mu_i\mid 1\leq i\leq N)}\otimes \bigwedge_{t\in T}(u_t)
\]
If $n$ is odd and $m\geq 2$ is even, we have a ring isomorphism
\[
W^\ast(X_D)\cong \frac{\mathbb{Z}_2\left[b_1,\ldots,b_{\frac{m}{2}},d_1,d_2\right]\otimes\bigotimes_{p=1}^l \mathbb{Z}_2\left[a_1^{(p)},\ldots,a_{\left\lfloor\frac{n_p}{2}\right\rfloor}^{(p)}\right]}{(d_1+d_2+b_{\frac{m}{2}},~d_1 d_2)+(\mu_i\mid 1\leq i\leq N)}\otimes \bigwedge_{t\in T}(u_t)
\]
If $n$ is even and $m>2$ is odd or if $n$ is even and $m=0$ and not all of $n_1,\ldots,n_l$ are even, we have a ring isomorphism
\[
W^\ast(X_D)\cong \frac{\mathbb{Z}_2\left[b_1,\ldots,b_{\left\lfloor\frac{m}{2}\right\rfloor}\right]\otimes\bigotimes_{p=1}^l \mathbb{Z}_2\left[a_1^{(p)},\ldots,a_{\left\lfloor\frac{n_p}{2}\right\rfloor}^{(p)}\right]}{(\mu_i\mid 1\leq i\leq N)}\otimes \bigwedge_{t\in T\setminus \left\{n-1\right\}}(u_t)\bigwedge(v_+,v_-)
\]
If $n$ is even, $m\geq2$ is even and not all of $n_1,\ldots,n_l$ are even, we have a ring isomorphism
\[
W^\ast(X_D)\cong \frac{\mathbb{Z}_2\left[b_1,\ldots,b_{\frac{m}{2}},d_1,d_2\right]\otimes\bigotimes_{p=1}^l \mathbb{Z}_2\left[a_1^{(p)},\ldots,a_{\left\lfloor\frac{n_p}{2}\right\rfloor}^{(p)}\right]}{(d_1+d_2+b_{\frac{m}{2}},~d_1d_2)+(\mu_i\mid 1\leq i\leq N)}\otimes \bigwedge_{t\in T\setminus \left\{n-1\right\}}(u_t)\otimes\bigwedge(v_+,v_-)
\]
If $m=0$ and $n_1,\ldots,n_l$ are even, we have a ring isomorphism
\[
W^\ast(X_D)\cong \frac{\bigotimes_{p=1}^l \mathbb{Z}_2\left[a_1^{(p)},\ldots,a_{\left\lfloor\frac{n_p}{2}\right\rfloor}^{(p)}\right]}{(\mu_i\mid 1\leq i\leq N)}\otimes \bigwedge_{t\in T\setminus\left\{n-1\right\}}(u_t)\otimes\bigwedge(v)
\]
In all of the above cases, the generators are of degrees $|a_i|=|b_j|=|d_1|=|d_2|=0$, $|u_t|=-1$ and
\[
|v|=|v_-|=|v_+|= 
\begin{cases}
-1 \text{ if }n\equiv 0\text{ (mod 4)}\\
-3 \text{ if }n\equiv 2\text{ (mod 4)}
\end{cases}
\]
The relations $\mu_i$ are given by
\[
\mu_i:=\sum_{p+q_1+\ldots+q_l=i} b_{p}a_{q_1}^{(1)}\ldots a_{q_l}^{(l)}+\binom{2n}{2i}
\]
where it is understood that
\begin{align*}
b_0:=1\text{ and }b_{i}:=b_{m-i}\text{ for all }\lfloor m/2\rfloor <i\leq m,\\
a_0^{(p)}:=1\text{ for all }1 \leq p\leq l\text{ and }a_i^{(p)}:=a_{n_p-i}^{(p)}\text{ for all }1\leq p\leq l,~\lfloor n_p/2\rfloor <i\leq n_p
\end{align*}
\end{theorem}
For the remaining case, we have only been able to determine the additive structure of the Witt groups:
\begin{theorem}
Suppose $m,n_1,\ldots,n_l$ are even where $0<m<n$. Let
\begin{align*}
 b:=&\frac{(n/2-1)!}{(m/2)!\cdot (n_1/2)!\cdot \ldots \cdot (n_l/2)!}\cdot \left(\frac{n_1}{2}+\ldots +\frac{n_l}{2}\right)\\
 c:=&\frac{(n/2-1)!}{(m/2)!\cdot (n_1/2)!\cdot \ldots \cdot (n_l/2)!}\cdot \frac{m}{2}\\
a:=&\frac{(n/2)!}{(m/2)!\cdot (n_1/2)!\cdot \ldots \cdot (n_l/2)!}+b
\end{align*}
Letting $z_i$ be the rank of the 4-periodically graded exterior algebra on $\frac{n_1+\ldots+n_l}{2}-1$ generators of degree $-1$, we have:
\[
\text{dim}_{\mathbb{Z}_2}(W^i(X_D))=
\begin{cases}
a(z_i+z_{i+1})\text{ if }n\equiv 0\text{ (mod 4)}\\
az_i+cz_{i+1}+2bz_{i+3}\text{ if }n\equiv 2\text{ (mod 4)}
\end{cases}
\]
\end{theorem}
\section{Polynomial relations}\label{sectionpolynomialeqns}
The purpose of this section is to prove the remaining Propositions from the previous section, thus completing our computations. These propositions each consist of two parts: a statement about the regularity of a certain sequence of polynomials and the relations that hold when extending this sequence by more polynomials. So in the first subsection, we prove a general result on regular sequences of inhomogeneous elements in graded rings. This is crucial to prove the regularity statements about our polynomials, which will be proved together with the statements about relations between them in the subsequent subsections. 
\subsection{Regularity of (in)homogeneous sequences in graded rings}
In this subsection, we prove a few results about regular sequences that we need. First we prove the decisive result about sequences of inhomogeneous elements in a graded ring. Roughly speaking, we show that if the highest homogeneous components of the inhomogeneous elements form a regular sequence, then so do the inhomogeneous elements themselves. This is a useful result because it is often easier to establish regularity of a sequence of homogeneous elements. So after establishing this result, we consider a certain sequence of homogeneous polynomials and prove that it is regular. Combining these two results will almost immediately imply the statements about regularity that we will need.

Suppose $R=\bigoplus_{p\in\mathbb{Z}} R_p$ is a commutative graded ring with $R_p=0$ for $p<0$. We define an ascending filtration of $R$ by subgroups by
\[
F^a R:=\bigoplus_{p\leq a} R_p
\]
This is not a filtration by ideals, but $F^a\cdot F^b\subset F^{a+b}$ for all $a,b\in\mathbb{Z}$. So $\text{gr}_F R:=\bigoplus_{a\in\mathbb{Z}} F^aR/F^{a-1}R$ inherits a ring structure. It is easy to see that $\text{gr}_F R\cong R$. Suppose now we are given elements $y_1,\ldots,y_n\in R$ with
\[
y_i=x_i+\text{lower degree terms}~~~(1\leq i\leq n)
\]
where $x_i\in R_{k_i}$ is homogeneous.
\begin{lemma}\label{lemmainhomogeneoussequence}
Suppose $x_1,\ldots,x_n$ is a regular sequence in every order in $R$ and $r_1,\ldots,r_l\in R$ with
\[
z:= r_1y_1+\ldots+r_ly_l\in F^aR
\]
Then the term of degree $a$ of $z$ is in $(x_1,\ldots,x_l)$.
\end{lemma}
\begin{proof}
We write $r_i=r_i'+r_i''$ where $r_i'$ is homogeneous of highest degree in $r_i$ and $r_i''$ is a sum of terms of lower degree. We prove the claim by induction on
\[
M:=\text{max}\left(\left\{\text{deg}(r_i')\text{deg}(x_i)\mid 1\leq i\leq l\right\}\right)
\]
If $M\leq a$, then the term in $z$ of degree $a$ is a sum of some of the terms $r_i'x_i$ and so is in $(x_1,\ldots,x_l)$.\\Suppose $M>a$ and let $1\leq i_1,\ldots,i_k\leq l$ be all the distinct indices such that $\text{deg}(r_{i_p}')\text{deg}(x_{i_p})=M$ for $1\leq p\leq k$. Then $r_{i_1}'x_{i_1}+\ldots+r_{i_k}'x_{i_k}$ is the term in $z$ of highest degree $M$ and so
\[
r_{i_1}'x_{i_1}+\ldots+r_{i_k}'x_{i_k}=0
\]
since $M>a$ and $z\in F^a R$. Since the $x_i$ are a regular sequence in any order, we deduce that $r_{i_p}'\in (x_1,\ldots,x_l)$ for $1\leq p\leq k$. But then consider
\[
z':=z-\sum_{p=1}^k r_{i_p}'y_{i_p}=r_1y_1+\ldots+r_ly_l-\sum_{p=1}^k r_{i_p}'y_{i_p}
\]
By induction, we deduce that the term of degree $a$ in $z'$ is in $(x_1,\ldots,x_l)$. But since $r_{i_p}'\in (x_1,\ldots,x_n)$ for $1\leq p\leq k$, this also holds for $z$.
\end{proof}
\begin{remark}
The assertion is not true if we do not assume regularity of the sequence $x_i$. For example, consider
$R=\mathbb{Z}_2[\alpha,\beta]\text{   with deg}(\alpha)=\text{deg}(\beta)=1$,
and set $y_1=\alpha$ and $y_2=\alpha^2+\beta$ so that $x_1=\alpha$ and $x_2=\alpha^2$.
Then $\alpha y_1+y_2=\beta\in F^1 R$, but $\beta\not\in(x_1,x_2)$.
\end{remark}
The filtration $F^\ast R$ induces a filtration $F^\ast R_l$ of $R_l:=R/(y_1,\ldots,y_l)$ for every $1\leq l\leq n$ via the quotient map $q_l\colon R\to R_l$.
\begin{proposition}\label{isogradedring}
We use the notation introduced in this section.
\begin{enumerate}[(i)]
\item There is a surjective ring homomorphism
\[
R/(x_1,\ldots,x_l)\to \text{gr}_F(R_l)=\bigoplus_{a\in \mathbb{Z}} F^aR_l/F^{a-1}R_l
\]
of graded rings for every $1\leq l\leq n$.
\item If $x_1,\ldots,x_n$ is an $R$-regular sequence in every order, the surjection from (i) is an isomorphism.
\end{enumerate}
\end{proposition}
\begin{remark}
Since the $x_i$ are homogeneous,
\[
R/(x_1,\ldots,x_l)\cong \bigoplus_{a\in\mathbb{Z}} \frac{R_a}{R_a\cap(x_1,\ldots,x_l)}
\]
is naturally a graded ring. Furthermore, the condition in (ii) that $x_1,\ldots,x_n$ be regular in \emph{every} order is not much stronger than that $x_1,\ldots,x_n$ be regular in \emph{some} order: If $R$ is Noetherian and $R_0$ is a field, these conditions are equivalent.
\end{remark}
\begin{proof}
We first prove (i). For every $a\in\mathbb{Z}$, we have a surjective additive homomorphism
\begin{align*}
\frac{F^aR}{F^{a-1}R}\twoheadrightarrow \frac{\text{im}(F^aR\xrightarrow{q_l}R/(y_1,\ldots,y_l))}{\text{im}(F^{a-1}R\xrightarrow{q_l}R/(y_1,\ldots,y_l))}=\frac{F^aR_l}{F^{a-1}R_l},~~
[x]\mapsto [q_l(x)]
\end{align*}
These homomorphisms induce a surjective ring homomorphism $\text{gr}_F R\to \text{gr}_F R_l$ and thus we obtain a surjective ring homomorphism
\[
R= \bigoplus_{a\in \mathbb{Z}} R_a\cong \text{gr}_F R\to \text{gr}_F R_l
\]
We observe that $x_1,\ldots,x_l$ are all in the kernel of this map. Hence this induces a surjective ring homomorphism
$
f\colon R/(x_1,\ldots,x_l)\twoheadrightarrow \text{gr}_F R_l
$.

We now prove (ii). We define an inverse of the map $f$ in (i). First, for every $a\in\mathbb{Z}$ we define a map
\begin{align*}
F^aR_l=\text{im}(F^aR\xrightarrow{q_l}R/(y_1,\ldots,y_l))&\to R_a/(R_a\cap(x_1,\ldots,x_l)),\\
q_l(x+\text{lower degree terms})&\mapsto [x]~~~\text{(where}~x\in R_a\text{)}
\end{align*}
This is well-defined by Lemma \ref{lemmainhomogeneoussequence}. It is clearly an additive homomorphism for every $a\in\mathbb{Z}$ and induces a homomorphism
\[
\frac{F^aR_l}{F^{a-1}R_l}\to \frac{R_a}{R_a\cap(x_1,\ldots,x_l)}
\]
for every $a\in \mathbb{Z}$. All these maps induce a ring homomorphism
$
g\colon \text{gr}_F R_l\to R/(x_1,\ldots,x_l)
$
which is, by construction, inverse to the homomorphism $f$ constructed in (i).
\end{proof}
We can now prove the main result of this section:
\begin{corollary}\label{corollaryinhomogeneousregseq}
Let $R=\bigoplus_{p\in\mathbb{Z}}$ be a commutative graded ring with $R_p=0$ for $p<0$. Let $y_1,\ldots,y_n\in R$ be such that for all $1\leq i\leq n$,
\[
y_i=x_i+\text{lower degree terms}
\]
where $x_i\in R_{k_i}$ is homogeneous.
If $x_1,\ldots,x_n$ is $R$-regular in every order, then so is $y_1,\ldots,y_n$.
\end{corollary}
\begin{proof}
Suppose $y_1,\ldots,y_l$ is an $R$-regular sequence and suppose $r\in R$ with $ry_{l+1}\in (y_1,\ldots,y_l)$. Suppose $q_l(r)\neq 0$ in $R_l=R/(y_1,\ldots,y_l)$. Then there is a minimal $p\in \mathbb{Z}$ such that $q_l(r)\in F^p R_l$. Consider $[q_l(r)]\in F^pR_l/F^{p-1}R_l$ and $[q_l(y_{l+1})]=[q_l(x_{l+1})]\in F^{k_{l+1}}R_l/F^{k_{l+1}-1}R_l$. Then
\[
0=[q_l(r)]\cdot [q_l(y_{l+1})]=[q_l(r)]\cdot [q_l(x_{l+1})]~~~\text{in }~\frac{F^{p+k_{l+1}}R_l}{F^{p+k_{l+1}-1}R_l}\subset\text{gr}_F R_l.
\]
Applying the isomorphism $g$ from the proof of Proposition \ref{isogradedring} (ii) and using that $x_1,\ldots,x_{l+1}$ is $R$-regular, we deduce that $q_l(r)=0$ in $F^p R_l/F^{p-1}R_l$ and so $q_l(r)\in F^{p-1}R_l$. This contradicts the minimality of $p$.
\end{proof}
Now we turn to a particular sequence of polynomials and prove its regularity. Let $K$ be a field. We consider the polynomial ring
\[
A=K[X_{1,1},X_{1,2},\ldots,X_{1,n_1},X_{2,1},\ldots,X_{m,1},\ldots,X_{m,n_m}]
\]
as a graded ring with grading given by $|X_{i,j}|=j$. We define certain polynomials $Q_a$ by
\[
Q_a=\sum_{i_1+\ldots+i_m=a}X_{1,i_1}X_{2,i_2}\ldots X_{m,i_m}~~~~~\left(1\leq a\leq N:=\sum_{j=1}^m n_j\right)
\]
where in this sum, we follow the convention that $X_{k,0}=1$ for $1\leq k\leq m$. Note that $Q_a$ is homogeneous of degree $a$.
\begin{proposition}\label{Qaregularsequence}
The $Q_a$ for $1\leq a\leq N$ form a regular sequence.
\end{proposition}
\begin{proof}
Since the $Q_a$ are homogeneous, it is sufficient to show that the ideal generated by the $Q_a$ is of maximal height. So let $P$ be a prime ideal with $Q_a\in P$ for $1\leq a\leq N$. It is sufficient to show that $P=(X_{i,j}\mid 1\leq i\leq m,~1\leq j\leq n_i)$.\\
Suppose this is false and let $\left\{i_1,\ldots,i_r\right\}=\left\{i\mid \exists j\colon X_{i,j}\not\in P\right\}$. By assumption, this set is non-empty. For every $1\leq k\leq r$, let $j_k$ be the largest $j$ with $X_{i_k,j}\not \in P$.\\
Let $J=\sum_{k=1}^r j_k$. Then every monomial in $Q_J$ either contains some $X_{i,j}$ with $i\not\in\left\{i_1,\ldots,i_r\right\}$, in which case this monomial is in $P$, or it contains only monomials $X_{i,j}$ with $i\in \left\{i_1,\ldots,i_r\right\}$. In the latter case, there is either some $1\leq k\leq r$ such that the monomial contains $X_{i_k,j}$ with $j>j_k$ (in which case this monomial is in $P$ according to the maximality condition on $j_k$) or the monomial is equal to $X_{i_1,j_1}\ldots X_{i_r,j_r}$, which is not in $P$ since $P$ is prime. So there is only one monomial in $Q_J$ which is not in $P$. Thus $Q_J\not\in P$. But this is a contradiction. Hence $P=(X_{i,j})$ as required.
\end{proof}
Let $I$ be the ideal in $A$ generated by the $Q_a$ for $1\leq a\leq N$. As the $Q_a$ form a regular sequence, we have $\text{ht}(I)=N$ and so $A/I$ has Krull dimension 0. Hence $A/I$ is a finite-dimensional $K$-vector space. From \cite[Corollary 3.3]{stanley} it is not hard to compute the vector space dimension of $A/I$ and we obtain:
\begin{proposition}\label{propdimensionQaalgebra}
$\text{dim}_K(A/I)=F(A/I,1)=\frac{(n_1+\ldots+n_m)!}{n_1!\cdot\ldots\cdot n_m!}$
\end{proposition}

\subsection{The polynomials $\mu_j$}
Let $n_i\in \mathbb{N}$ for each $i\in \mathbb{N}$ be such that for some $l\in\mathbb{N}_0$, the integer $n_i$ is even if $1\leq i\leq l$ and odd if $i>l$. For every $k\in \mathbb{N}_0$, we define
\[
R_k:=\frac{\bigotimes_{p=1}^{k+l}\mathbb{Z}_2\left[\beta_0^{(p)},\ldots,\beta_{n_p}^{(p)}\right]}{\left(\beta_i^{(p)}+\beta_{n_p-i}^{(p)}\mid 1\leq p\leq k+l,~0\leq i\leq n_p\right)+\left(\beta_0^{(p)}+1\mid 1\leq p\leq k+l\right)}
\]
Clearly, $R_k$ is isomorphic to a polynomial ring in $\sum_{p=1}^{k+l}\lfloor n_p/2\rfloor$ indeterminates. The relations in $R_k$ ensure that there is a sort of mirror symmetry among each family of generators.

The ring $R_k$ reflects of course the Tate cohomology ring of representation rings of centralisers of tori in $SU(n)$ in our computations for type $A_n$. We define a mod-2 rank ring homomorphism via
\begin{align*}
\text{rk}\colon R_k\to \mathbb{Z}_2,~~~\beta_i^{(p)}\mapsto \binom{n_p}{i}
\end{align*}
This is, of course, also a reflection of the mod-2 rank function induced on Tate cohomology by the rank function on the representation ring of centralisers of tori.

Now for $k\geq 1$, we define an inclusion map
\[
\kappa_{k-1}\colon R_{k-1}\to R_k,~~~\beta_i^{(p)}\mapsto \beta_i^{(p)}
\]
The following polynomials are the main objects of study in this section: For $m\in\mathbb{Z}$ and $k\in\mathbb{N}_0$, we define
\[
\mu_m^{(k)}:=\sum_{a_1+\ldots+a_{k+l}=m}\beta_{a_1}^{(1)}\ldots \beta_{a_{k+l}}^{(k+l)}\in R_k
\]
From the interpretation of the $\mu_m^{(k)}$ as restrictions of representations of $SU(n)$ to centralisers of tori in Tate cohomology or directly from a combinatorial interpretation, it is clear that $\text{rk}\left(\mu_m^{(k)}\right)=\binom{n_1+\ldots+n_{k+l}}{m}$. We define a reduced version of the above polynomials by
\[
\tilde{\mu}_m^{(k)}:=\mu_m^{(k)}+\text{rk}\left(\mu_m^{(k)}\right)
\]
These are precisely the polynomials we considered in section \ref{sectioncomputationtypea}. Note that $\mu_m^{(k)}=0$ if $m\not\in\left\{0,1,\ldots,n_1+\ldots+n_{k+l}\right\}$ and $\mu_0^{(k)}=1$. Furthermore,
\begin{align*}
\mu_m^{(k)}=&\sum_{a_1+\ldots+a_{k+l}=m}\beta_{a_1}^{(1)}\ldots \beta_{a_{k+l}}^{(k+l)}=\sum_{a_1+\ldots+a_{k+l}=m}\beta_{n_1-a_1}^{(1)}\ldots \beta_{n_{k+l}-a_{k+l}}^{(k+l)}\\
=&\sum_{b_1+\ldots+b_{k+l}=n_1+\ldots+n_{k+l}-m}\beta_{b_1}^{(1)}\ldots \beta_{b_{k+l}}^{(k+l)}=\mu_{n_1+\ldots+n_{k+l}-m}^{(k)}
\end{align*}
So we will only consider the polynomials $\mu_m^{(k)}$ for $0\leq m\leq \lfloor (n_1+\ldots+n_{k+l})/2\rfloor$.
\begin{example}\label{exampleformu}
Let $l=0$ and $k=2$ and $n_1=3$, $n_2=5$. We write the polynomials $\mu_j$ in terms of the indeterminates $\beta_i^{(p)}$ where $i\leq \lfloor n_p/2\rfloor$:
\begin{align*}
    \mu_1&=\mu_7=\beta_1^{(1)}+\beta_1^{(2)}\\
    \mu_2&=\mu_6=\beta_1^{(1)}+\beta_1^{(1)}\beta_1^{(2)}+\beta_2^{(2)}\\
    \mu_3&=\mu_5=1+\beta_1^{(1)}\beta_1^{(2)}+\beta_1^{(1)}\beta_2^{(2)}+\beta_2^{(2)}\\
    \mu_4&=0
\end{align*}
\end{example}
Already in the above example, the polynomials $\mu_j$ appear to be quite complicated. We will however show that there is a quite orderly pattern of how they relate to one another.
\subsubsection{Regularity and dimension}
We noted above that $R_k$ is just a polynomial ring in $\lfloor n_1/2\rfloor+\ldots+\lfloor n_{k+l}/2\rfloor$ indeterminates. Hence this is also the maximal length of a regular sequence in $R_k$. In fact, we show:
\begin{proposition}\label{propmuregular}
The elements
\[
\mu_1^{(k)},\mu_2^{(k)},\ldots,\mu_{\lfloor n_1/2\rfloor+\ldots+\lfloor n_{k+l}/2\rfloor}^{(k)}\in R_k
\]
form an $R_k$-regular sequence. The same holds for
\[
\tilde{\mu}_1^{(k)},\tilde{\mu}_2^{(k)},\ldots,\tilde{\mu}_{\lfloor n_1/2\rfloor+\ldots+\lfloor n_{k+l}/2\rfloor}^{(k)}\in R_k
\]
\end{proposition}
\begin{proof}
We regard $R_k$ as a graded ring where the grading is defined by setting $\lvert \beta_i^{(p)}\rvert=i$ for all $1\leq p\leq k+l$ and $0\leq i\leq \lfloor n_p/2\rfloor$. Note that the $\mu_m^{(k)}$ are not homogeneous with respect to this grading. But if $1\leq m\leq \lfloor n_1/2\rfloor+\ldots+\lfloor n_{k+l}/2\rfloor$, then $\mu_m^{(k)}$ has highest homogeneous component of degree $m$  given by
\[
\sum_{\substack{a_1+\ldots+a_{k+l}=m\\a_p\leq \lfloor n_p/2\rfloor\text{ for }1\leq p\leq k+l}}\beta_{a_1}^{(1)}\ldots \beta_{a_{k+l}}^{(k+l)}=:h_m^{(k)}
\]
In Proposition \ref{Qaregularsequence}, we showed precisely that these $h_m^{(k)}$ form a regular sequence for $1\leq m\leq \lfloor n_1/2\rfloor+\ldots+\lfloor n_{k+l}/2\rfloor$. Hence we deduce from Corollary \ref{corollaryinhomogeneousregseq} that the $\mu_m^{(k)}$ form an $R_k$-regular sequence for $1\leq m\leq \lfloor n_1/2\rfloor+\ldots+\lfloor n_{k+l}/2\rfloor$.

The same proof also works for the $\tilde{\mu}_m^{(k)}$ since they have the same homogeneous components as the $\mu_m^{(k)}$ except in degree 0.
\end{proof}
From Proposition \ref{isogradedring}, we see that 
\[
R_k/\left(\tilde{\mu}_1^{(k)},\ldots,\tilde{\mu}_{\lfloor n_1/2\rfloor+\ldots+\lfloor n_{k+l}/2\rfloor}^{(k)}\right)~\text{ and }~R_k/\left(h_1^{(k)},\ldots,h_{\lfloor n_1/2\rfloor+\ldots+\lfloor n_{k+l}/2\rfloor}^{(k)}\right)
\]
have the same $\mathbb{Z}_2$-vector space dimension because the former ring has a filtration such that the associated graded ring is the latter ring. We computed this dimension in Proposition \ref{propdimensionQaalgebra} and thus obtain:
\begin{proposition}\label{propositiondimensionmu}
\[
\text{dim}_{\mathbb{Z}_2}\left(R_k/\left(\tilde{\mu}_1^{(k)},\ldots,\tilde{\mu}_{\lfloor n_1/2\rfloor+\ldots+\lfloor n_{k+l}/2\rfloor}^{(k)}\right)\right)=\frac{\left(\lfloor \frac{n_1}{2}\rfloor+\ldots+\lfloor \frac{n_{k+l}}{2}\rfloor\right)!}{\lfloor \frac{n_1}{2}\rfloor!\cdot\ldots\cdot \lfloor \frac{n_{k+l}}{2}\rfloor!}
\]
\end{proposition}

\subsubsection{Linear combinations}
We now want to show the following:
\begin{proposition}\label{propositionrelationsmu}
For every $\lfloor n_1/2\rfloor+\ldots+\lfloor n_{k+l}/2\rfloor <m\leq \lfloor (n_1+\ldots+ n_{k+l})/2\rfloor$, the polynomial $\mu_m^{(k)}$ (or $\tilde{\mu}_m^{(k)}$) is a $\mathbb{Z}_2$-linear combination of the polynomials 
\[
\mu_0^{(k)},\mu_1^{(k)},\ldots,\mu_{\lfloor n_1/2\rfloor+\ldots+\lfloor n_{k+l}/2\rfloor}^{(k)} \text{ (or }\tilde{\mu}_1^{(k)},\ldots,\tilde{\mu}_{\lfloor n_1/2\rfloor+\ldots+\lfloor n_{k+l}/2\rfloor}^{(k)}\text{)}
\]
in $R_k$.
\end{proposition}
The proof will take up the remainder of this section.

Let us first see how the assertion about the $\mu_m^{(k)}$ implies the assertion about the $\tilde{\mu}_m^{(k)}$. Assuming the claim about the unreduced polynomials, it follows that for every $m$,
\[
\tilde{\mu}_m^{(k)}\in \mathbb{Z}_2\cdot 1+\mathbb{Z}_2\cdot \tilde{\mu}_1^{(k)}+\ldots+\mathbb{Z}_2\cdot \tilde{\mu}_{\lfloor n_1/2\rfloor+\ldots+\lfloor n_{k+l}/2\rfloor}^{(k)}
\]
Applying the rank function rk to both sides and using that $\text{rk}\left(\tilde{\mu}_j^{(k)}\right)=0$ for all $j$, we see that actually,
\[
\tilde{\mu}_m^{(k)}\in \mathbb{Z}_2\cdot \tilde{\mu}_1^{(k)}+\ldots+\mathbb{Z}_2\cdot \tilde{\mu}_{\lfloor n_1/2\rfloor+\ldots+\lfloor n_{k+l}/2\rfloor}^{(k)}
\]
as required.

Thus we are left to prove the assertion about the unreduced polynomials. Let us first change the numbering of the indeterminates and the polynomials in a convenient way. In $R_k$, we define
\begin{align*}
    \alpha_i^{(p)}&:=\beta_{i+\lfloor n_p/2\rfloor}^{(p)}\text{ for all }1\leq p\leq k+l \text{ and }-\lfloor n_p/2\rfloor \leq i\leq \lceil n_p/2\rceil \\\sigma_m^{(k)}&:=\mu_{m+\lfloor n_1/2\rfloor +\ldots+\lfloor n_{k+l}/2\rfloor}^{(k)}\text{ for all }m\in\mathbb{Z}.
\end{align*}
These definitions are chosen so that
\begin{align}\label{symmetryalpha}
\alpha_i^{(p)}=\begin{cases} \alpha_{-i}^{(p)}\text{ if }1\leq p\leq l\\\alpha_{-i+1}^{(p)}\text{ if }l<p\leq k+l\end{cases}
\end{align}
and so that the $\sigma_m^{(k)}$ form an $R_k$-regular sequence for all $-\lfloor n_1/2\rfloor-\ldots-\lfloor n_{k+l}/2\rfloor< m\leq 0 $ and 
\begin{align}\label{symmetrysigma}
\sigma_m^{(k)}=\sigma_{-m+k}^{(k)}\text{ for all }m
\end{align}
Also note that
$
\sigma_m^{(k)}=\sum_{a_1+\ldots+a_{k+l}=m} \alpha_{a_1}^{(1)}\ldots \alpha_{a_{k+l}}^{(k+l)}
$.

Now we rephrase Proposition \ref{propositionrelationsmu} as follows:
\begin{proposition}\label{propositionrelationssigma}
For every $0< m\leq \lfloor k/2\rfloor$, the polynomial $\sigma_m^{(k)}$ is a $\mathbb{Z}_2$-linear combination of the polynomials $\sigma_i^{(k)}$ for $i\leq 0$.
\end{proposition}
We prove a few first lemmas in this direction:
\begin{lemma}\label{lemma1relationsmu}
If $k\geq2$ is even, then $\sigma_{k/2}^{(k)}=0$.
\end{lemma}
\begin{proof}
$\sigma_{k/2}^{(k)}$ is the sum of all monomials of the form $\alpha_{a_1}^{(1)}\ldots\alpha_{a_{k+l}}^{(k+l)}$ with $\sum_{i=1}^{k+l}a_i=k/2$. But by (\ref{symmetryalpha}), we have
\begin{align}\label{sigmarelations1}
\alpha_{a_1}^{(1)}\ldots \alpha_{a_{k+l}}^{(k+l)}=\alpha_{-a_1}^{(1)}\ldots\alpha_{-a_l}^{(l)}\alpha_{-a_{l+1}+1}^{(l+1)}\ldots\alpha_{-a_{k+l}+1}^{(k+l)}
\end{align}
Since $-\sum_{i=1}^{l}a_i+\sum_{j=l+1}^{k+l}(-a_j+1)=k-\sum_{i=1}^{k+l}a_i=k/2$, the right hand side of (\ref{sigmarelations1}) also gives a summand in $\sigma_{k/2}^{(k)}$, which is always distinct from the one on the left hand side since $k>0$. So all the monomial summands in $\sigma_{k/2}^{(k)}$ cancel out each other.
\end{proof}

\begin{lemma}\label{lemma2relationsmus}
If $k\geq 3$ is odd, then $\sum_{j\leq \lfloor k/2\rfloor} \sigma_j^{(k)}=0$
\end{lemma}
\begin{proof}
We have
\begin{align*}
    \sum_{j\leq \lfloor k/2\rfloor}\sigma_j^{(k)}=&\sum_{j\leq\lfloor k/2\rfloor}\sum_{m\in\mathbb{Z}}\left(\alpha_{-m}^{(k+l)}\cdot \kappa_{k-1}\left(\sigma_{j+m}^{(k-1)}\right)\right)\\
    =&\sum_{j\leq\lfloor k/2\rfloor}\sum_{m\geq0}\left(\alpha_{-m}^{(k+l)}\cdot \kappa_{k-1}\left(\sigma_{j+m}^{(k-1)}\right)+\alpha_{m+1}^{(k+l)}\cdot \kappa_{k-1}\left(\sigma_{j-m-1}^{(k-1)}\right)\right)\\
    =&\sum_{m\geq 0}\alpha_{-m}^{(k+l)}\cdot \kappa_{k-1}\left(\sum_{j\leq \lfloor k/2\rfloor}\left(\sigma_{j+m}^{(k-1)}+\sigma_{j-m-1}^{(k-1)}\right)\right)
\end{align*}
Now in $R_{k-1}$,
\begin{align*}
    \sum_{j\leq \lfloor k/2\rfloor}\left(\sigma_{j+m}^{(k-1)}+\sigma_{j-m-1}^{(k-1)}\right)= &\sum_{\lfloor k/2\rfloor-m\leq j\leq \lfloor k/2\rfloor +m}\sigma_j^{(k-1)}\\
    = & \sigma_{\frac{k-1}{2}}^{(k-1)}+\sum_{\frac{k-1}{2}-m\leq j<\frac{k-1}{2}}\left(\sigma_j^{(k-1)}+\sigma_{-j+k-1}^{(k-1)}\right)=0
\end{align*}
where we use (\ref{symmetrysigma}) and Lemma \ref{lemma1relationsmu}. Hence the claim follows.
\end{proof}
The previous lemmas already yield some of the linear relations we need. It turns out that we can deduce more linear relations from the basic ones above by induction on $k$. To see how to obtain these, it is useful to rephrase the problem in terms of power series. Very roughly, we shall see that obtaining a new linear relation from the basic ones above corresponds to a manipulation of the corresponding series that is easy to describe.

Let $\mathbb{Z}_2\langle t\rangle$ denote the ring of series of the form $\sum_{i\in\mathbb{Z}}a_it^ i$ with $a_i\in\mathbb{Z}_2$ for all $i\in\mathbb{Z}$ and $a_i=0$ for $i\gg0$. For each $k\geq 0$, we define a map
\[
\psi_k\colon \mathbb{Z}_2\langle t\rangle \to R_k,~~~\sum_{j\in\mathbb{Z}}a_j t^j\mapsto \sum_{j\in\mathbb{Z}}a_j\sigma_j^{(k)}
\]
This map is well-defined since $\sigma_j^{(k)}=0$ for $j\ll0$. Clearly, $\psi_k$ is a group homomorphism.

Finding linear relations among the $\sigma_j^{(k)}$ is now the same as finding elements in $\text{ker}(\psi_k)$. Let
\[
P_k(t):=\sum_{j\leq \lfloor k/2\rfloor} t^j\in\mathbb{Z}_2\langle t \rangle
\]
We deduce from Lemmas \ref{lemma1relationsmu} and \ref{lemma2relationsmus} and from (\ref{symmetrysigma}):
\begin{description}
\item[(P1)] If $k\geq 3$ is odd, then $P_k(t)\in \text{ker}(\psi_k)$.
\item[(P2)] If $k\geq 2$ is even, then $t^{k/2}\in \text{ker}(\psi_k)$.
\item[(P3)] For every $k\geq 0$ and every $s\in\mathbb{Z}$, we have $t^{-s}+t^{s+k}\in\text{ker}(\psi_k)$.
\end{description}
We want to find more elements in the kernel of $\psi_k$ by induction on $k$. So we need to be able to reduce from $k$ to $k-1$:
\begin{lemma}\label{lemmapsireductiontok-1}
Let $Q(t)\in\mathbb{Z}_2\langle t\rangle$ and $k\geq 1$. Then
\[
\psi_k(Q(t))=\sum_{i=1}^{\infty} \alpha_i^{(k+l)}\cdot \kappa_{k-1}\psi_{k-1}\left((t^{-i}+t^{i-1})Q(t)\right)
\]
\end{lemma}
\begin{proof}
Let $Q(t)=\sum_{j\in\mathbb{Z}}a_j t^j$. Then
\begin{align*}
    \psi_k(Q(t))=& \sum_{j\in \mathbb{Z}} a_j \sigma_j^{(k)}
    = \sum_{j\in\mathbb{Z}} a_j\sum_{i\in\mathbb{Z}} \alpha_i^{(k+l)}\cdot \kappa_{k-1}\left(\sigma_{j-i}^{(k-1)}\right)
    =\sum_{i\in\mathbb{Z}}\alpha_i^{(k+l)}\cdot \kappa_{k-1}\left(\sum_{j\in\mathbb{Z}} a_j\sigma_{j-i}^{(k-1)}\right)\\
    =& \sum_{i\in\mathbb{Z}} \alpha_i^{(k+l)}\cdot \kappa_{k-1}\psi_{k-1}\left(t^{-i}Q(t)\right)
    =\sum_{i\geq1} \alpha_i^{(k+l)}\cdot \kappa_{k-1}\psi_{k-1}\left((t^{-i}+t^{i-1})Q(t)\right)
\end{align*}
as required.
\end{proof}
The following result now yields the elements of $\text{ker}(\psi_k)$ that we need:
\begin{proposition}\label{propositionkernelpsi}
If $k=2m+1\geq3$ is odd, then
\[
\frac{1}{(1+t^{-1})^{2j}}\cdot P_{k-2j}(t)\in\text{ker}(\psi_k)\text{ for all }0\leq j< m
\]
If $k=2m\geq4$ is even, then
\[
\frac{1}{(1+t^{-1})^{2j+1}}\cdot P_{k-2j-1}(t)\in\text{ker}(\psi_k)\text{ for all }0\leq j<m-1
\]
\end{proposition}
Before proving this, let us see how it implies Proposition \ref{propositionrelationssigma}.
\begin{proof}[of Proposition \ref{propositionrelationssigma}.]
Suppose $k=2m+1$ is odd. For all $0\leq j< m$, the highest non-zero term of
\[
\frac{1}{(1+t^{-1})^{2j}}\cdot P_{k-2j}(t)=(1+t^{-1}+t^{-2}+\ldots)^{2j}\cdot\sum_{i\leq m-j} t^i
\]
is of degree $m-j$. Since all these elements are in $\text{ker}(\psi_k)$, this means that for each $0< j\leq m=\lfloor k/2\rfloor$, we can write $\sigma_j^{(k)}$ as a $\mathbb{Z}_2$-linear combination of the $\sigma_i^{(k)}$ for $i<j$. This implies the claim.

For $k=2m$ even and $0\leq j<m-1$, the highest non-zero term of
\[
\frac{1}{(1+t^{-1})^{2j+1}}\cdot P_{k-2j-1}(t)=(1+t^{-1}+t^{-2}+\ldots)^{2j+1}\cdot\sum_{j\leq m-j-1} t^i
\]
is of degree $m-1-j$. Thus for every $0< j<m=k/2$, we can write $\sigma_j^{(k)}$ as a $\mathbb{Z}_2$-linear combination of the $\sigma_i^{(k)}$ for $i<j$. Thus for every $0<j<m$, we can write $\sigma_j^{(k)}$ as a $\mathbb{Z}_2$-linear combination of the $ \sigma_i^{(k)}$ for $i\leq0$. Furthermore, from Lemma \ref{lemma1relationsmu} we know that $\sigma_{m}^{(k)}=0$. So we are done in this case as well.
\end{proof}
\begin{proof}[of Proposition \ref{propositionkernelpsi}.]
We prove this by induction on $k$. For $k=3$, this is (P1).

So suppose $k=2m\geq 4$ is even and let $0\leq j<m-1$. By Lemma \ref{lemmapsireductiontok-1}, it suffices to show that
\[
\frac{t^{-a}+t^{a-1}}{(1+t^{-1})^{2j+1}}\cdot P_{k-2j-1}(t)\in\text{ker}(\psi_{k-1})\text{ for all }a\geq 1.
\]
Note that
\begin{align*}
    \frac{t^{-a}+t^{a-1}}{(1+t^{-1})^{2j+1}}\cdot P_{k-2j-1}(t)=&(t^{-a}+t^{a-1})\cdot \sum_{i\leq 0} t^i\cdot \frac{P_{k-2j-1}(t)}{(1+t^{-1})^{2j}}\\
    =&(t^{-a+1}+t^{-a+2}+\ldots+t^{a-2}+t^{a-1})\cdot Q_{k,j}(t)\\
    =& q_{a-1}(t)\cdot Q_{k,j}(t)
\end{align*}
where $Q_{k,j}(t):=P_{k-2j-1}(t)\cdot (1+t^{-1})^{-2j}$, and $q_{a-1}(t):=\sum_{|i|\leq a-1}t^i$ is a finite sum with $q_{a-1}(t)=q_{a-1}(t^{-1})$.

By induction hypothesis, we know that for all $0\leq b<m-1$,
\begin{align*}
    \text{ker}(\psi_{k-1})\ni& \frac{1}{(1+t^{-1})^{2b}}\cdot P_{k-2b-1}(t)\\
    =&(1+t^{-1})^{-2b}\cdot (1+t^{-1})^{2j}\cdot t^{j-b}\cdot \frac{P_{k-2j-1}(t)}{(1+t^{-1})^{2j}}\\
    =& t^{j-b}(1+t^{-1})^{2(j-b)}\cdot Q_{k,j}(t)\\
    =&r_{j-b}(t)\cdot Q_{k,j}(t)
\end{align*}
where $r_{j-b}(t):=t^{j-b}(1+t^{-1})^{2(j-b)}$. For all $0\leq b\leq j$, the element $r_{j-b}(t)$ is a finite sum with highest non-zero term of degree $j-b$ such that $r_{j-b}(t^{-1})=r_{j-b}(t)$.

Furthermore, by (P3), for every $c\in\mathbb{Z}$, we have that $t^c+t^{-c+k-1}\in\text{ker}(\psi_{k-1})$. We can write
\begin{align*}
    t^c+t^{-c+k-1}=& (t^c+t^{-c+2m-1})\cdot \frac{(1+t^{-1})^{2j}}{P_{k-2j-1}(t)}\cdot \frac{P_{k-2j-1}(t)}{(1+t^{-1})^{2j}}\\
    =&(t^c+t^{-c+2m-1})\cdot(1+t^{-1})^{2j}\cdot t^{-\left\lfloor \frac{k-2j-1}{2}\right\rfloor}(1+t^{-1}) \cdot Q_{k,j}(t)\\
    =&(t^c+t^{-c+2m-1})\cdot(1+t^{-1})^{2j+1}\cdot t^{-(m-j-1)}\cdot Q_{k,j}(t)\\
    =&(t^{c-m}+t^{m-c-1})\cdot t^{j+1}\cdot (1+t^{-1})^{2j+1}\cdot Q_{k,j}(t)\\
    =& s_{c-m,j}(t)\cdot Q_{k,j}(t)
\end{align*}
where $s_{c-m,j}(t):= (t^{c-m}+t^{m-c-1})\cdot t^{j+1}\cdot (1+t^{-1})^{2j+1}$. Note that $s_{c-m,j}(t)$ is a finite sum, we have $s_{c-m,j}(t^{-1})=s_{c-m,j}(t)$ and for $c\geq m$ it has highest term of degree $1+j+c-m$.

Hence we see that
\[
\left\{r_{j-b}(t)\mid 0\leq b\leq j\right\}\cup \left\{s_{c-m,j}(t)\mid c\geq m\right\}
\]
is a basis of the $\mathbb{Z}_2$-vector space
\[
\left\{\sum_{i\in\mathbb{Z}}a_i t\in\mathbb{Z}_2\langle t \rangle\mid a_i=a_{-i}\text{ for all }i\in\mathbb{Z}\right\}
\]
Thus for all $a\in\mathbb{Z}$, the element $q_{a-1}(t)$ can be written as a $\mathbb{Z}_2$-linear combination of these basis elements. But by the above, this shows that $q_{a-1}(t)\cdot Q_{k,j}(t)$ can be written as a $\mathbb{Z}_2$-linear combination of elements in $\text{ker}(\psi_{k-1})$ for all $a\geq 1$. So it is itself in $\text{ker}(\psi_{k-1})$.

The case that $k$ is odd follows similarly.
\end{proof}
\subsection{The polynomials $\nu_j$}
We extend the results of the previous section to the more general polynomials occurring in the computations for types $C_n$ and $D_n$.

Let $m,n_1,\ldots,n_l\in\mathbb{N}$ and define
\[
S:=\frac{\mathbb{Z}_2[\alpha_0,\alpha_1,\ldots,\alpha_{2m}]\otimes\bigotimes_{p=1}^l \mathbb{Z}_2\left[\beta_0^{(p)},\beta_1^{(p)},\ldots,\beta_{n_p}^{(p)}\right]}{I_\alpha+I_\beta}
\]
where
\begin{align*}
    I_\alpha=&(\alpha_i+\alpha_{2m-i}\mid 0\leq i\leq 2m)+(\alpha_0+1)\\
    I_\beta=&\left(\beta_i^{(p)}+\beta_{n_p-i}^{(p)}\mid 1\leq p\leq l,~0\leq i\leq n_p\right)+\left(\beta_0^{(p)}+1\mid 1\leq p\leq l\right)
\end{align*}
We set $n:=m+n_1+\ldots+n_l$. The ring $S$ is isomorphic to a polynomial ring in $m+\lfloor n_1/2\rfloor+\ldots+\lfloor n_l/2\rfloor$ indeterminates. Just as in the previous section, we define a mod-2 rank ring homomorphism
$
\text{rk}\colon S\to \mathbb{Z}_2
$
via $\text{rk}\left(\alpha_i\right):=\binom{2m}{i}$ and $\text{rk}\left(\beta_i^{(p)}\right):=\binom{n_p}{i}$.

We define elements in $S$ by
\begin{align*}
    \mu_k:= \sum_{a_1+\ldots+a_l=k}\beta_{a_1}^{(1)}\ldots \beta_{a_l}^{(l)},~~~~
    \nu_k:=\sum_{i=0}^{\lfloor k/2\rfloor} \alpha_{k-2i}\cdot \mu_i+\binom{2n}{k}
\end{align*}
From the interpretation of the $\nu_k$ as elements in the Tate cohomology of a representation ring and the above mod-2 rank function as computing the rank of representations modulo 2, it is clear that $\text{rk}(\nu_k)=0$ for all $k$. The $\mu_k$ are basically the polynomials that we considered in the previous section. Up to the additive constant, they can be obtained as a special case of the $\nu_k$ by taking $m=0$.
\begin{example}\label{examplenu}
Let $l=2$ and $m=2$, $n_1=3$, $n_2=5$. Then $n=10$ and
\begin{align*}
    \nu_1&=\alpha_1
    &\nu_2&=\alpha_2+\mu_1\\
    \nu_3&=\alpha_1+\alpha_1\mu_1
    &\nu_4&=1+\alpha_2\mu_1+\mu_2+1\\
    \nu_5&=\alpha_1\mu_1+\alpha_1\mu_2
    &\nu_6&=\mu_1+\alpha_2\mu_2+\mu_3\\
    \nu_7&=\alpha_1\mu_2+\alpha_1\mu_3
    &\nu_8&=\mu_2+\alpha_2\mu_3+\mu_4\\
    \nu_9&=\alpha_1\mu_3+\alpha_1\mu_4
    &\nu_{10}&=\alpha_2\mu_4
\end{align*}
The polynomials $\mu_j$ were given in Example \ref{exampleformu}.
\end{example}
\subsubsection{Regularity and dimension}
\begin{lemma}\label{lemmanuregular1}
If $k$ is odd, then $\nu_k$ is a $\bigotimes_{p=1}^l \mathbb{Z}_2\left[\beta_0^{(p)},\beta_1^{(p)},\ldots,\beta_{n_p}^{(p)}\right]$-linear combination of the $\alpha_i$ where $i$ is odd.

We have $(\nu_i\mid 1\leq i\leq m\text{ odd})=(\alpha_i\mid 1\leq i\leq m\text{ odd})$. Moreover, the elements in $\left\{\nu_i\mid 1\leq i\leq m\text{ odd}\right\}$ are $S$-regular in any order.
\end{lemma}
\begin{proof}
Let $k$ be odd. Then $\binom{2n}{k}\equiv 0$ (mod 2) and both $k-2i$ and $2m-(k-2i)$ are odd for every $i$. So from the definition, we obtain that $\nu_k$ is a $\bigotimes_{p=1}^l \mathbb{Z}_2\left[\beta_0^{(p)},\beta_1^{(p)},\ldots,\beta_{n_p}^{(p)}\right]$-linear combination of the $\alpha_i$ where $i$ is odd. Furthermore, note that if $0\leq k\leq m$ is odd, then
\[
\nu_k=\alpha_k+\epsilon_k
\]
where each $\epsilon_k$ is a $\bigotimes_{p=1}^l \mathbb{Z}_2\left[\beta_0^{(p)},\beta_1^{(p)},\ldots,\beta_{n_p}^{(p)}\right]$-linear combination of the $\alpha_i$ where $i$ is odd with $1\leq i<k$. This implies the second part of the lemma.
\end{proof}
\begin{proposition}\label{propositionregularitynu}
Let
\begin{align*}
    A_1:=&\left\{\nu_i\mid 1\leq i\leq m\text{ odd}\right\},\\A_2:=&\left\{\nu_i\mid 2\leq i\leq 2(\lfloor m/2\rfloor +\lfloor n_1/2\rfloor+\ldots+\lfloor n_l/2\rfloor)\text{ even}\right\}
\end{align*}
The elements in $A:=A_1\cup A_2$ form an $S$-regular sequence in some order.
\end{proposition}
\begin{proof}
First note that if $i$ is even, then the expression for $\nu_i$ does not contain any $\alpha_j$ for odd $j$. By Lemma \ref{lemmanuregular1}, the elements in $A_1$ form an $S$-regular sequence. Let
\[
\overline{S}:=S/(A_1)=S/(\alpha_i\mid i\text{ odd})
\]
In $\overline{S}$, we define $\overline{\gamma}_i:=\overline{\alpha}_{2i}$. Then $\overline{\gamma}_{m-i}=\overline{\alpha}_{2(m-i)}=\overline{\alpha}_{2i}=\overline{\gamma}_i$ for all $i$. Note that for every $i$,
\begin{align*}
    \overline{\nu}_{2i}=\sum_{j=0}^i \overline{\alpha}_{2i-2j}\cdot \overline{\mu}_j+\binom{2n}{i}=\sum_{j=0}^{i} \overline{\gamma}_{i-j}\cdot \overline{\mu}_j+\binom{2n}{i}
\end{align*}
Up to the constant $\binom{2n}{i}$, this now is an element in $\overline{S}$ of the type that we considered in the previous section. So from Proposition \ref{propmuregular}, we deduce that $\overline{\nu}_2,\overline{\nu}_4,\ldots,\overline{\nu}_{2(\lfloor m/2\rfloor+\lfloor n_1/2\rfloor+\ldots+\lfloor n_l/2\rfloor)}$ is an $\overline{S}$-regular sequence. This finishes the proof.
\end{proof}
From Propositions \ref{isogradedring} and \ref{propdimensionQaalgebra}, we deduce:
\begin{proposition}\label{propositiondimensionquotientnu}
\[
\text{dim}_{\mathbb{Z}_2}\left(S/\left(A_1\cup A_2\right)\right)=\frac{\left(\lfloor \frac{m}{2}\rfloor+\lfloor \frac{n_1}{2}\rfloor+\ldots+\lfloor \frac{n_{k+l}}{2}\rfloor\right)!}{\lfloor \frac{m}{2}\rfloor!\cdot \lfloor \frac{n_1}{2}\rfloor!\cdot\ldots\cdot \lfloor \frac{n_{k+l}}{2}\rfloor!}
\]
\end{proposition}

\subsubsection{Relations between the $\nu_i$}
Recall the definition of $A_1$ and $A_2$ in Proposition \ref{propositionregularitynu}.
\begin{proposition}\label{propositionrelationsnu}
If $i$ is odd, then $\nu_i$ is a $\bigotimes_{p=1}^l \mathbb{Z}_2\left[\beta_0^{(p)},\beta_1^{(p)},\ldots,\beta_{n_p}^{(p)}\right]$-linear combination of the $\nu_j$ where $1\leq j\leq m$ are odd. In particular, $\nu_i\in (A_1)$. If $i$ is even, then $\nu_i$ is a $\mathbb{Z}_2$-linear combination of elements in $A_2$. In particular, $\nu_i\in (A_2)$.
\end{proposition}
\begin{proof}
The claim about odd $i$ follows immediately from Lemma \ref{lemmanuregular1}. Now consider the $\nu_i$ for even $i$. For each $j$, we define $\gamma_j:=\alpha_{2j}$ in $S$. We have $\gamma_{m-j}=\alpha_{2m-2j}=\alpha_{2j}=\gamma_j$ for all $j$ and then for every $i$,
\[
\nu_{2i}=\sum_{j=0}^i \alpha_{2i-2j}\cdot \mu_j+\binom{2n}{i}=\sum_{j=0}^{i}\gamma_{i-j}\cdot \mu_j+\binom{2n}{i}
\]
and so up to the constant $\binom{2n}{i}$, the $\nu_{2i}$ are polynomials of the form considered in the previous section in the variables $\gamma_j$ and $\beta_j^{(p)}$. Thus Proposition \ref{propositionrelationsmu} implies that each $\nu_{2i}$ is a $\mathbb{Z}_2$-linear combination of 1 and the $\nu_j\in A_2$. But applying the rank function, we see that it is actually a $\mathbb{Z}_2$-linear combination of just the $\nu_j\in A_2$ since $\text{rk}(\nu_j)=0$ for all $j$.
\end{proof}

\subsection{The polynomials $\xi_j$}
We now consider the polynomials occurring in the computation for type $B_n$.

Let $m,n_1,\ldots, n_l\in \mathbb{N}$. We set $n:=m+n_1+\ldots+n_l$ and define
\[
T:=\frac{\mathbb{Z}_2[\alpha_0,\alpha_1,\ldots,\alpha_{2m+1}]\otimes\bigotimes_{p=1}^l \mathbb{Z}_2\left[\beta_0^{(p)},\beta_1^{(p)},\ldots,\beta_{n_p}^{(p)}\right]}{J_\alpha+J_\beta}
\]
where
\begin{align*}
    J_\alpha=&(\alpha_i+\alpha_{2m+1-i}\mid 0\leq i\leq 2m+1)+(\alpha_0+1)\\
    J_\beta=&\left(\beta_i^{(p)}+\beta_{n_p-i}^{(p)}\mid 1\leq p\leq l,~0\leq i\leq n_p\right)+\left(\beta_0^{(p)}+1\mid 1\leq p\leq l\right)
\end{align*}
Note that $T$ is very similar to the ring $S$ from the previous section, but the largest index $2m+1$ occurring for the $\alpha_i$ is odd for $T$. The ring $T$ is isomorphic to a polynomial ring in $m+\lfloor n_1/2\rfloor +\ldots+\lfloor n_l/2\rfloor$ indeterminates. We define a mod-2 rank ring homomorphism
$
\text{rk}\colon T\to \mathbb{Z}_2
$
via $\text{rk}(\alpha_i):=\binom{2m+1}{i}$ and $\text{rk}\left(\beta_j^{(p)}\right):=\binom{n_p}{j}$.
In $T$, we define elements
\begin{align*}
    \mu_k:=\sum_{a_1+\ldots+a_l=k}\beta_{a_1}^{(1)}\ldots \beta_{a_l}^{(l)},~~~~
    \xi_k:=\sum_{i=0}^{\lfloor k/2\rfloor} \alpha_{k-2i}\cdot \mu_i+\binom{2n+1}{k}
\end{align*}
From the interpretation of the $\xi_k$ as elements in the Tate cohomology of a representation ring and the above mod-2 rank function as computing the rank of representations modulo 2, it is clear that $\text{rk}(\xi_k)=0$ for all $k$.
These polynomials $\xi_j$ in $T$ formally look the same as the polynomials $\nu_j$ in $S$. However, due to the fact that the highest occurring index of the $\alpha_i$ is odd in one case and even in the other, the polynomials $\xi_j$ and $\nu_j$ are actually qualitatively different. To illustrate this, compare the following with Example \ref{examplenu}.
\begin{example}
Let $l=2$ and $m=2$, $n_1=3$, $n_2=5$. We then have $n=10$ and
\begin{align*}
    \xi_1&= \alpha_1+1 &\xi_2&=\alpha_2+\mu_1\\
    \xi_3&=\alpha_2+\alpha_1\mu_1 &\xi_4&=\alpha_1+\alpha_2\mu_1+\mu_2+1\\
    \xi_5&=\alpha_2\mu_1+\alpha_1\mu_2 &\xi_6&=\alpha_1\mu_1+\alpha_2\mu_2+\mu_3\\
    \xi_7&=\mu_1+\alpha_2\mu_2+\alpha_1\mu_3 &\xi_8&=\alpha_1\mu_2+\alpha_2\mu_3+\mu_4\\
    \xi_9&=\mu_2+\alpha_2\mu_3+\alpha_1\mu_4 &\xi_{10}&=\alpha_1\mu_3+\alpha_2\mu_4+\mu_3
\end{align*}
The expressions for the $\mu_j$ were given in Example \ref{exampleformu}.
\end{example}

\begin{proposition}\label{propositionregularityxi}
The elements in
\[
\left\{\xi_i\mid 1\leq i\leq m\text{ odd}\right\}\cup\left\{\xi_i\mid 2\leq i\leq 2\left(\left\lfloor \frac{m}{2}\right\rfloor+\left\lfloor \frac{n_1}{2}\right\rfloor+\ldots+\left\lfloor \frac{n_l}{2}\right\rfloor\right)\text{ even}\right\}
\]
form a $T$-regular sequence in some order.
\end{proposition}
\begin{proof}
If $1\leq i\leq m$ is odd, then
\[
\xi_i=\alpha_i+\eta_i+\binom{2n+1}{i}
\]
where $\eta_i\in(\alpha_1,\alpha_3,\ldots,\alpha_{i-2})$ for every odd $1\leq i\leq m$. Hence it is clear that the $\xi_i$ for odd $1\leq i\leq m$ form a $T$-regular sequence.

So now let
\[
\overline{T}:=T/(\xi_i\mid 1\leq i\leq m\text{ odd})
\]
We see from the above that if $1\leq i\leq m$ is odd, then $\overline{\alpha}_i\in\overline{T}$ can be expressed in terms of the $\overline{\beta}_j^{(p)}$ and a constant. Note that $\overline{T}$ is still a polynomial ring. We can define a grading on $\overline{T}$ by setting
\begin{align*}
    \left|\overline{\beta}_i^{(p)}\right|&:=i~~\text{ for all }1\leq p\leq l,~1\leq i\leq \lfloor n_p/2\rfloor\\
    \left|\overline{\alpha}_{2i}\right|&:= i~~\text{ for all }1\leq i\leq \lfloor m/2\rfloor
\end{align*}
Then for all $2\leq 2i\leq 2(\lfloor m/2\rfloor +\lfloor n_1/2\rfloor+\ldots+\lfloor n_l/2\rfloor)$, we see that the highest homogeneous component of $\overline{\xi}_{2i}\in\overline{T}$ is of degree $i$, and Proposition \ref{Qaregularsequence} shows that they form a $\overline{T}$-regular sequence. Hence by Corollary \ref{corollaryinhomogeneousregseq}, the elements $\overline{\xi}_{2i}$ for $2\leq 2i\leq 2(\lfloor m/2\rfloor +\lfloor n_1/2\rfloor+\ldots+\lfloor n_l/2\rfloor)$ form a $\overline{T}$-regular sequence.

This completes the proof.
\end{proof}
\begin{remark}\label{remarkregularityxi}
The proof of the previous Proposition shows that we may replace $\xi_{2i}$ for $1\leq i\leq \lfloor m/2\rfloor+\lfloor n_1/2\rfloor +\ldots+\lfloor n_l/2\rfloor$ by any $\delta\in T$ such that $\overline{\delta}\in\overline{T}$ has the same highest homogeneous component as $\overline{\xi}_{2i}\in\overline{T}$ and still obtain a regular sequence.
\end{remark}
From Propositions \ref{isogradedring} and \ref{propdimensionQaalgebra}, we deduce:
\begin{proposition}\label{propositiondimensionquotientxi}
Let $I$ be the ideal in $T$ generated by the $\xi_i$ where $1\leq i\leq m$ is odd and by the $\xi_i$ where $2\leq i\leq 2(\lfloor m/2\rfloor+\lfloor n_1/2\rfloor+\ldots+\lfloor n_l/2\rfloor)$ is even. Then
\[
\text{dim}_{\mathbb{Z}_2}\left(T/I\right)=\frac{\left(\lfloor \frac{m}{2}\rfloor+\lfloor \frac{n_1}{2}\rfloor+\ldots+\lfloor \frac{n_{k+l}}{2}\rfloor\right)!}{\lfloor \frac{m}{2}\rfloor!\cdot \lfloor \frac{n_1}{2}\rfloor!\cdot\ldots\cdot \lfloor \frac{n_{k+l}}{2}\rfloor!}
\]
\end{proposition}

\begin{lemma}\label{lemmaevenxioddxi}
If $1\leq 2j+1\leq m$, then $\alpha_{2j+1}+\alpha_{2j}\in(\xi_i\mid 1\leq i\leq m)$. More precisely, $\alpha_{2j+1}+\alpha_{2j}$ is a $\bigotimes_{p=1}^l \mathbb{Z}_2\left[\beta_0^{(p)},\beta_1^{(p)},\ldots,\beta_{n_p}^{(p)}\right]$-linear combination of the $\xi_i$ where $1\leq i\leq m$.
\end{lemma}
\begin{proof}
We prove this by induction on $j$.

For $j=0$, the assertion holds as
\[
\alpha_1+\alpha_0=\alpha_1+1=\alpha_1+\binom{2n+1}{1}=\xi_1\in (\xi_i\mid 1\leq i\leq m)
\]
Suppose now $j>0$. Then
\begin{align}\label{equationevenxioddxi}
\begin{split}
    \xi_{2j+1}+\xi_{2j}=&\sum_{i=0}^j \alpha_{2j+1-2i}\cdot \mu_i+\binom{2n+1}{2j+1}+\sum_{i=0}^j \alpha_{2j-2i}\cdot \mu_i+\binom{2n+1}{2j}\\
    =&\sum_{i=0}^j (\alpha_{2j-2i+1}+\alpha_{2j-2i})\mu_i\\
    =&\alpha_{2j+1}+\alpha_{2j}+\sum_{i=1}^j (\alpha_{2j-2i+1}+\alpha_{2j-2i})\mu_i
\end{split}
\end{align}
using that $\binom{2n+1}{2j+1}+\binom{2n+1}{2j}=\binom{2n+2}{2j+1}\equiv 0$ (mod 2). By induction hypothesis, for all $1\leq i\leq j$, the element $\alpha_{2j-2i+1}+\alpha_{2j-2i}$ is a $\bigotimes_{p=1}^l \mathbb{Z}_2\left[\beta_0^{(p)},\beta_1^{(p)},\ldots,\beta_{n_p}^{(p)}\right]$-linear combination of the $\xi_p$ for $1\leq p\leq m$.
Hence we deduce that so is $\alpha_{2j+1}+\alpha_{2j}$. This completes the proof.
\end{proof}
\begin{proposition}\label{propositionrelationsxi}
For all $1\leq j\leq n$,
\[
\xi_j\in (\xi_i\mid 1\leq i\leq m\text{ odd})+\left(\xi_i\mid 2\leq i\leq 2\left(\left\lfloor \frac{m}{2}\right\rfloor+\left\lfloor \frac{n_1}{2}\right\rfloor+\ldots+\left\lfloor \frac{n_l}{2}\right\rfloor\right)\text{ even}\right)=:\hat{I}
\]
\end{proposition}
\begin{proof}
By equation (\ref{equationevenxioddxi}) and Lemma \ref{lemmaevenxioddxi}, the claim for odd $j$ follows from the statement for even $j$. Let $\hat{T}:= T/\hat{I}$. We need to show that the elements $\hat{\xi}_j\in\hat{T}$ represented by $\xi_j\in T$ are zero in $\hat{T}$ for even $j$.

Note that $(\xi_1,\ldots,\xi_m)\subset \hat{I}$. Let $\hat{\alpha}_{i}\in \hat{T}$ be the element represented by $\alpha_i\in T$. For $0\leq i\leq m$, we define $\gamma_i:=\hat{\alpha}_{2i}\in \hat{T}$. By Lemma \ref{lemmaevenxioddxi}, we have $\hat{\alpha}_{2i}=\hat{\alpha}_{2i+1}$ for $1\leq 2i+1\leq m$. Hence, for $m\geq i>\lfloor m/2\rfloor$,
\[
\gamma_i=\hat{\alpha}_{2i}=\hat{\alpha}_{2m+1-2i}=\hat{\alpha}_{2(m-i)+1}=\hat{\alpha}_{2(m-i)}=\gamma_{m-i}\text{ in }\hat{T}
\]
Furthermore, for $1\leq j\leq \lfloor n/2\rfloor$, we can write
\begin{align*}
    \hat{\xi}_{2j}=&\sum_{i=0}^j \hat{\alpha}_{2j-2i}\hat{\mu}_i +\binom{2n+1}{2j}\\
    =&\sum_{i=0}^j \gamma_{j-i}\hat{\mu}_i+\binom{2n+1}{2j}\\
    =&\sum_{a_0+\ldots+a_l=j}\gamma_{a_0}\hat{\beta}_{a_1}^{(1)}\ldots\hat{\beta}_{a_l}^{(l)}+\binom{2n+1}{2j}
\end{align*}
Thus Proposition \ref{propositionrelationsmu} implies that $\hat{\xi}_{2j}$ is a $\mathbb{Z}_2$-linear combination of
\[
1,~\hat{\xi}_2,~\hat{\xi}_4,\ldots,~\hat{\xi}_{2(\lfloor m/2\rfloor+\lfloor n_1/2\rfloor+\ldots+\lfloor n_l/2\rfloor)}
\]
But applying the rank function and using that $\text{rk}(\xi_i)=0$ for all $i$, we see that $\hat{\xi}_{2j}$ is actually a $\mathbb{Z}_2$-linear combination of
\[
\hat{\xi}_2,~\hat{\xi}_4,\ldots,~\hat{\xi}_{2(\lfloor m/2\rfloor+\lfloor n_1/2\rfloor+\ldots+\lfloor n_l/2\rfloor)}
\]
But $\hat{\xi}_{2i}=0$ in $\hat{T}$ for all $1\leq i\leq \lfloor m/2\rfloor+\lfloor n_1/2\rfloor +\ldots+\lfloor n_l/2\rfloor$. So the assertion follows.
\end{proof}
\begin{remark}\label{remarkxi}
We see from Lemma \ref{lemmaevenxioddxi} and the proof of Proposition \ref{propositionrelationsxi} that a slightly stronger claim holds: For all $1\leq j\leq n$, the element $\xi_j$ is a $\bigotimes_{p=1}^l \mathbb{Z}_2\left[\beta_0^{(p)},\beta_1^{(p)},\ldots,\beta_{n_p}^{(p)}\right]$-linear combination of the generators of $\hat{I}$ displayed above. We may even restrict to those generators $\xi_i\in\hat{I}$ for which $i<j$.
\end{remark}

\section*{Appendix A}
Suppose $W^\ast=\bigwedge(v_1,\ldots, v_f,w_1,\ldots,w_g)$ is a $\mathbb{Z}_4$-graded exterior algebra over $\mathbb{Z}_2$ with $v_i\in W^{-1}$ for all $i$ and $w_j\in W^{-3}$ for all $j$. Let $u_k:=\text{dim}_{\mathbb{Z}_2} W^k$ for $0\geq k\geq -3$.

If we regarded $W^\ast$ as a $\mathbb{Z}$-graded algebra with generators in the degrees given above, its Hilbert polynomial would be
\[
H(t):=(1+t^{-1})^f\cdot (1+t^{-3})^g\in \mathbb{Z}[t]
\]
Noting that in $\mathbb{C}$, we have the identities
\begin{align*}
    1^0+(-1)^0+i^0+(-i)^0&=4,\\
    1^1+(-1)^1+i^1+(-i)^1&=0,\\
    1^2+(-1)^2+i^2+(-i)^2&=0,\\
    1^3+(-1)^3+i^3+(-i)^3&=0,
\end{align*}
we deduce that if $(f,g)\neq(0,0)$, then
\begin{align*}
    u_{-k}=&\frac{1}{4}\left(1^k\cdot H(1)+(-1)^k\cdot H(-1)+i^k\cdot H(i)+(-i)^k\cdot H(-i)\right)\\
    =&\frac{1}{4}\left(2^{f+g}+2\cdot\text{Re}(i^k\cdot H(i))\right)\\
    =&\frac{1}{4}\left(2^{f+g}+2\cdot \text{Re}\left(i^k(1-i)^f(1+i)^g\right)\right)
\end{align*}
Let $\zeta:=e^{\frac{\pi i}{4}}$ be a primitive eighth root of unity. Then $1+i=\sqrt{2}\cdot \zeta$ and $1-i=\sqrt{2}\cdot \zeta^{-1}$ and so we deduce
\begin{align*}
u_{-k}=2^{f+g-2}+2^{\frac{f+g-2}{2}}\cdot \text{Re}\left(\zeta^{2k-f+g}\right)
\end{align*}
Concretely, we obtain the following table for $(f,g)\neq (0,0)$. The first two columns contain the values of $f$ and $g$ modulo 4. Throughout, we write $x:=f+g$.
\begin{table}[h]
\resizebox{\textwidth}{!}{%
\begin{tabular}{c|c|c|c|c|c}
$f~(4)$ & $g~(4)$ & $u_0$ & $u_{-1}$ & $u_{-2}$ & $u_{-3}$\\
\hline
0 & 0 & $2^{x-2}-2\cdot (-4)^{\frac{x-4}{4}}$ & $2^{x-2}$ & $2^{x-2}+2\cdot (-4)^{\frac{x-4}{4}}$ & $2^{x-2}$\\
1 & 0 & $2^{x-2}-2\cdot (-4)^{\frac{x-5}{4}}$ & $2^{x-2}-2\cdot (-4)^{\frac{x-5}{4}}$ & $2^{x-2}+2\cdot (-4)^{\frac{x-5}{4}}$ & $2^{x-2}+2\cdot (-4)^{\frac{x-5}{4}}$\\
2 & 0 & $2^{x-2}$ & $2^{x-2}+ (-4)^{\frac{x-2}{4}}$ & $2^{x-2}$ & $2^{x-2}- (-4)^{\frac{x-2}{4}}$\\
3 & 0 & $2^{x-2}- (-4)^{\frac{x-3}{4}}$ & $2^{x-2}+(-4)^{\frac{x-3}{4}}$ & $2^{x-2}+ (-4)^{\frac{x-3}{4}}$ & $2^{x-2}-(-4)^{\frac{x-3}{4}}$\\
0 & 1 & $2^{x-2}-2\cdot (-4)^{\frac{x-5}{4}}$ & $2^{x-2}+2\cdot (-4)^{\frac{x-5}{4}}$ & $2^{x-2}+2\cdot (-4)^{\frac{x-5}{4}}$ & $2^{x-2}-2\cdot (-4)^{\frac{x-5}{4}}$\\
1 & 1 & $2^{x-2}+ (-4)^{\frac{x-2}{4}}$ & $2^{x-2}$ & $2^{x-2}- (-4)^{\frac{x-2}{4}}$ & $2^{x-2}$\\
2 & 1 & $2^{x-2}+(-4)^{\frac{x-3}{4}}$ & $2^{x-2}+(-4)^{\frac{x-3}{4}}$ & $2^{x-2}- (-4)^{\frac{x-3}{4}}$ & $2^{x-2}- (-4)^{\frac{x-3}{4}}$\\
3 & 1 & $2^{x-2}$ & $2^{x-2}+2\cdot (-4)^{\frac{x-4}{4}}$ & $2^{x-2}$ & $2^{x-2}-2\cdot (-4)^{\frac{x-4}{4}}$\\
0 & 2 & $2^{x-2}$ & $2^{x-2}- (-4)^{\frac{x-2}{4}}$ & $2^{x-2}$ & $2^{x-2}+ (-4)^{\frac{x-2}{4}}$\\
1 & 2 & $2^{x-2}+ (-4)^{\frac{x-3}{4}}$ & $2^{x-2}- (-4)^{\frac{x-3}{4}}$ & $2^{x-2}- (-4)^{\frac{x-3}{4}}$ & $2^{x-2}+ (-4)^{\frac{x-3}{4}}$\\
2 & 2 & $2^{x-2}+2\cdot (-4)^{\frac{x-4}{4}}$ & $2^{x-2}$ & $2^{x-2}-2\cdot (-4)^{\frac{x-4}{4}}$ & $2^{x-2}$\\
3 & 2 & $2^{x-2}+2\cdot (-4)^{\frac{x-5}{4}}$ & $2^{x-2}+2\cdot (-4)^{\frac{x-5}{4}}$ & $2^{x-2}-2\cdot (-4)^{\frac{x-5}{4}}$ & $2^{x-2}-2\cdot (-4)^{\frac{x-5}{4}}$\\
0 & 3 & $2^{x-2}- (-4)^{\frac{x-3}{4}}$ & $2^{x-2}- (-4)^{\frac{x-3}{4}}$ & $2^{x-2}+ (-4)^{\frac{x-3}{4}}$ & $2^{x-2}+ (-4)^{\frac{x-3}{4}}$\\
1 & 3 & $2^{x-2}$ & $2^{x-2}-2\cdot (-4)^{\frac{x-4}{4}}$ & $2^{x-2}$ & $2^{x-2}+2\cdot (-4)^{\frac{x-4}{4}}$\\
2 & 3 & $2^{x-2}+2\cdot (-4)^{\frac{x-5}{4}}$ & $2^{x-2}-2\cdot (-4)^{\frac{x-5}{4}}$ & $2^{x-2}-2\cdot (-4)^{\frac{x-5}{4}}$ & $2^{x-2}+2\cdot (-4)^{\frac{x-5}{4}}$\\
3 & 3 & $2^{x-2}-(-4)^{\frac{x-2}{4}}$ & $2^{x-2}$ & $2^{x-2}+ (-4)^{\frac{x-2}{4}}$ & $2^{x-2}$
\end{tabular}}
\end{table}

\affiliationone{
   Tobias Hemmert\\
   Heinrich-Heine-Universit\"at D\"usseldorf\\
   Institut f\"ur Mathematik\\
   Universit\"atsstra\ss{}e 1\\
   40225 D\"usseldorf\\
   Germany
   \email{hemmert@hhu.de}
}

\begin{thebibliography}{99}
%
%
\bibitem{AtiyahHirzebruch1}
{\bibname M. Atiyah \and F. Hirzebruch}, `Vector bundles and homogeneous spaces', {\em Proc. Sympos. Pure Math., Volume III, American Mathematical Society }(1961) 7--38.
%
\bibitem{Balmer1}
{\bibname P. Balmer}, `Triangular Witt groups. I. The 12-term localization exact sequence', {\em K-Theory }19 (2000) no. 4 311--363.
%
\bibitem{borelhirz1}
{\bibname A. Borel \and F. Hirzebruch}, `Characteristic Classes and Homogeneous Spaces, I', {\em American Journal of Mathematics }80 (1958) no. 2 458--538.
%
\bibitem{brockertomdieck}
{\bibname T. Br\"{o}cker \and T. tom Dieck}, {\em Representations of compact Lie groups} (Springer-Verlag, 1985).
%
\bibitem{fujii1}
{\bibname M. Fujii}, `$K_O$-groups of projective spaces', {\em Osaka Journal of Mathematics }4 (1967) no. 1 141--149.
%
\bibitem{Hemmert}
{\bibname T. Hemmert}, `KO-Theory of Complex Flag Varieties', PhD-Thesis, available at \url{https://docserv.uni-duesseldorf.de/servlets/DocumentServlet?id=47769} (2018).
%
\bibitem{Hodgkin1}
{\bibname L. Hodgkin}, `The equivariant K\"{u}nneth theorem in K-theory', in: `Topics in K-Theory', Lecture Notes in Mathematics, Vol. 496, {\em Springer-Verlag }(1975).
%
\bibitem{kishimotokonoohsita1}
{\bibname D. Kishimoto, A. Kono \and A. Ohsita}, `KO-theory of complex partial flag manifolds', {\em The Quarterly Journal of Mathematics }65 (2014) no. 2 327--338.
%
\bibitem{kishimotokonoohsita2}
{\bibname D. Kishimoto, A. Kono \and A. Ohsita}, `KO-theory of flag manifolds', {\em J. Math. Kyoto Univ. }44 (2004) no. 1 217--227.
%
\bibitem{kishimotoohsita1}
{\bibname D. Kishimoto \and A. Ohsita}, `KO-theory of exceptional flag manifolds', {\em Kyoto J. Math. }53 (2013) no. 3 673--692.
%
\bibitem{konohara1}
{\bibname A. Kono \and S. Hara}, `KO-theory of complex Grassmannians', {\em J. Math. Kyoto Univ. }31 (1991) no. 3 827--833.
%
\bibitem{panin}
{\bibname I. A. Panin}, `On the algebraic K-theory of twisted flag varieties`, {\em K-Theory }8 (1994) no. 6 541--585.
%
\bibitem{pittie1}
{\bibname H. Pittie}, `Homogeneous vector bundles on homogeneous spaces', {\em Topology }11 (1972) no. 2 199--203.
%
\bibitem{Serre1}
{\bibname J.-P. Serre}, `Repr\'{e}sentations lin\'{e}aires et espaces homog\`{e}nes k\"{a}hl\'{e}riens des groupes de Lie compacts', {\em S\'{e}minaire N. Bourbaki, }Volume 2 (1951-1954), Talk no. 100, 447--454.
%
\bibitem{Serre2}
{\bibname J.-P. Serre}, `Groupes de Grothendieck des sch\'{e}mas en groupes r\'{e}ductifs d\'{e}ploy\'{e}s`, {\em Inst. Hautes \'{E}tudes Sci. Publ. Math. }34 (1968) 37--52.
%
\bibitem{stanley}
{\bibname R. Stanley}, `Hilbert functions of graded algebras', {\em Advances in Mathematics }28 (1978) no. 1 57--83.
%
\bibitem{steinberg1}
{\bibname R. Steinberg}, `On a Theorem of Pittie', {\em Topology }14 (1975) no. 2 173--177.
%
\bibitem{Zibrowius1}
{\bibname M. Zibrowius}, `KO-Rings of Full Flag Varieties', {\em Trans. Amer. Math. Soc. }367 (2015) 2997--3016.
%
\bibitem{Zibrowius2}
{\bibname M. Zibrowius}, `Witt groups of complex cellular varieties', {\em Documenta Math. }16 (2011) 465--511.
%
\bibitem{Zibrowius3}
{\bibname M. Zibrowius}, `Twisted Witt Groups of Flag Varieties', {\em J. K-Theory }14 (2014) no. 1 139--184.
\end{thebibliography}
\end{document}